\documentclass{article}

\usepackage{amsmath}
\usepackage{amsthm}
\usepackage{amssymb}
\usepackage{amsfonts}

\usepackage{subeqnarray}
\usepackage{bm}
\usepackage{tikz}
\usepackage{dsfont}
\usepackage{enumerate}
\usepackage{upgreek}
\usepackage{booktabs}
\usepackage{fancyhdr}
\usepackage{multirow}
\usepackage{multicol}
\usepackage{float}
\usepackage[inline]{enumitem}   
\usepackage{cases}

\usepackage[ruled,vlined,linesnumbered,inoutnumbered]{algorithm2e}

\usepackage{bbding}
\usepackage{caption}
\usepackage{todonotes}

\usepackage{mathrsfs}

\usepackage{lineno}

\usepackage[top=2.5cm, bottom=2.5cm, left=3cm, right=3cm]{geometry} 

\usepackage{graphics,graphicx,epsf,epstopdf,subfigure}
\graphicspath{{./Figures/}}
\ifpdf
\DeclareGraphicsExtensions{.eps,.pdf,.png,.jpg}
\else
\DeclareGraphicsExtensions{.eps}
\fi

\usepackage[colorlinks,
linkcolor=blue,
anchorcolor=blue,
citecolor=blue
]{hyperref}

\providecommand{\U}[1]{\protect\rule{.1in}{.1in}}

\newtheorem{theorem}{Theorem}[section]

\newtheorem{lemma}[theorem]{Lemma}

\newtheorem{remark}{Remark}

\numberwithin{equation}{section}

\allowdisplaybreaks
\graphicspath{ {./results/} }
\definecolor{Gray}{rgb}{0.5,0.5,0.5}

\newcommand{\tabincell}[2]{
    \begin{tabular}{@{}#1@{}}
        #2
    \end{tabular}
}

\makeatletter

\newcommand{\Rmnum}[1]{\expandafter\@slowromancap\romannumeral #1@}
\makeatother

\begin{document}


\title{Robust Sparse Phase Retrieval: Statistical Guarantee, Optimality Theory and Convergent Algorithm}

\author{Jun Fan\thanks{Institute of Mathematics, Hebei University of Technology, Tianjin 300401, China  (\texttt{junfan@hebut.edu.cn}).} 
\and Ailing Yan\thanks{Institute of Mathematics, Hebei University of Technology, Tianjin 300401, China (\texttt{ailing-yan78@163.com}).}
\and Xianchao Xiu\thanks{School of Mechatronic Engineering and Automation, Shanghai University, Shanghai 200444, China (\texttt{e-mail: xcxiu@shu.edu.cn}).}
\and Wanquan Liu\thanks{School of Intelligent Systems Engineering, Sun Yat-sen University, Guangzhou 528406, China (\texttt{liuwq63@mail.sysu.edu.cn}).}}

\date{} 

\maketitle

\begin{abstract}

\noindent
Phase retrieval (PR) is a popular research topic in signal processing and machine learning. However, its performance degrades significantly when the measurements are corrupted by noise or outliers. To address this limitation, we propose a novel robust sparse PR method that covers both real- and complex-valued cases. The core is to leverage the Huber function to measure the loss and adopt the $\ell_{1/2}$-norm regularization to realize feature selection, thereby improving the robustness of PR. 
In theory, we establish statistical guarantees for such robustness and derive necessary optimality conditions for global minimizers. Particularly, for the complex-valued case, we provide a fixed point inclusion property inspired by Wirtinger derivatives. 
Furthermore, we develop an efficient optimization algorithm by integrating the gradient descent method into a majorization-minimization (MM) framework. It is rigorously proved that the whole generated sequence is convergent and also has a linear convergence rate under mild conditions, which has not been investigated before. 
Numerical examples under different types of noise validate the robustness and effectiveness of our proposed method.
\vskip10pt 


\end{abstract}


\section{Introduction}
	
Over the past few decades, phase retrieval (PR) has received widespread attention from both academia and industry. A large number of PR models, theories and algorithms have been developed; see \cite{grohs2020phase,2015Phase,diederichs2024wirtinger,tu2025deep}. 
Mathematically, PR aims to recover the unknown signal $\mathbf{x}\in \mathbb{H}^p$ ($\mathbb{H}=\mathbb{R}$ or $\mathbb{C}$) from 
\begin{equation}\label{pr}
	b_i=|\langle \mathbf{a}_i, \mathbf{x}\rangle|^2+\varepsilon_i,~1\leq i\leq n,
\end{equation}
where $\mathbf{a}_i\in\mathbb{H}^{p}$ is the sampling vector, $b_i\in \mathbb{R}$ is the observed measurement,  $\varepsilon_i\in\mathbb{R}$ is the noise or outliers, and $\langle \cdot,  \cdot \rangle$ is the inner product.

\subsection{Literature Survey}

Thanks to the rapid development of compressive sensing \cite{donoho2006compressed}, sparse PR has been successfully applied in various fields ranging from optical engineering to signal processing \cite{Yoav2014GESPAR,yang19,cai2023provable}.
Generally speaking, sparse PR methods can be divided into two categories. 
The first category is to use matrix lifting techniques and then impose sparse regularization or constraints on the resulting matrix, which allows PR to be recast into a semidefinite programming (SDP) problem \cite{jaga13}. 
The second category is to enforce sparse regularization or constraints directly to the loss function with the least squares (LS) criterion and then apply gradient-based algorithms. 
In comparison, the second category without lifting techniques has attracted increasing attention because solving SDP problems usually requires substantial computing resources, which limits the scalability of large-scale applications.

In algorithms, the whole convergence analysis has been investigated in \cite{pauwels2017fienup,bolte2018first,chang2018variational}. 
The common technique for complex-valued PR is as follows.
First, lift the $p$-dimensional complex-valued sampling vectors $\mathbf{a}_i$ and signal $\mathbf{x}$ to $2p$-dimensional vectors via separating the real and imaginary parts of the corresponding vectors, and transform each term $|\langle \mathbf{a}_i, \mathbf{x}\rangle|^2$ in  (\ref{pr}) as a real quadratic form with $2p$-dimensional vectors. 
Then, enforce $\ell_1$-norm to the $2p$-dimensional vectors to regularize the above LS problem.  
Finally, apply the proximal gradient method (PGM) to solve the resulting optimization problem and use the Kurdyka-Lojasiewicz (KL) property based on the real-valued $\ell_1$-norm to study the whole convergence and rate. 
Here, the whole sequence is in the sense that the sequence generated by PGM is in the real domain. 
The limit point of the  generated sequence usually satisfies the fixed point equation based on the real-valued soft thresholding operator, which is a necessary optimality condition. 
It is worth mentioning that \cite{chang2018variational} also designed an optimization algorithm based on partially preconditioned proximal alternating linearized minimization and showed that the generated sequence converges entirely and the rate can be achieved via estimating the KL exponent. 
However, as stated in Remark \ref{equi-half} below, the regularized LS problem that is derived directly using the real-valued $\ell_1$-norm is not equivalent to the original problem. In other words, \cite{chang2018variational} guarantees that the generated sequence converges to a vector satisfying a necessary optimality condition for the real-valued $\ell_1$-norm minimization, but does not guarantee that the limit point of the complex-valued generated sequence plays a similar role in the original problem.

For the transformed minimization, the LS part is continuously differential without a global Lipschitz gradient. 
The PGM and its variants are an attractive class of algorithms for minimizing the sum of smooth and nonsmooth functions. 
A key feature in the convergence analysis is that the smooth part has a Lipschitz continuous gradient.
Recent works \cite{bolte2018first,kanzow2022convergence,jia2023convergence,takahashi2025approximate} have studied the convergence for PGM without global Lipschitz continuity on the gradient of the smooth part.
In particular, \cite{bolte2018first} proved that when the objective function satisfies the KL property, the whole sequence generated by PGM converges to a critical point with a general convergence rate.
In addition, \cite{jia2023convergence} obtained a similar result with the help of line search, and \cite{takahashi2025approximate} proved the global convergence when the nonsmooth term is convex.
As mentioned in \cite{bolte2018first}, the KL exponent is an important quantity for analyzing the convergence rate. 
Indeed, in order to guarantee a local linear convergence rate, it is desirable to determine whether a given function has a KL exponent of at most $1/2$. However, the KL exponent of a given function is often hard to estimate.

In practice, measurements are often corrupted by different noise or outliers. In the last few years, many state-of-the-art robust methods without using lifting techniques have been proposed \cite{2016,ZhangHuishuai2018Median,2021Median,2020Robust,kong20221}. Among them, two representatives against arbitrary corruptions are median-reshaped Wirtinger flow (M-RWF) \cite{ZhangHuishuai2018Median} and median-momentum reweighted amplitude flow (M-MRAF) \cite{2021Median}, which are variants of Wirtinger flow (WF) introduced by \cite{candesWF}.
However, both of them only focus on the LS criterion, so if the noise is not Gaussian or symmetric, the results are not optimal.
Therefore, it is suggested to replace LS with least absolute deviation (LAD) to improve the robustness \cite{weller15}.
After that, the $\ell_{1/2}$-norm regularized LAD method was proposed \cite{kong20221}. This is because that $\ell_{1/2}$-norm not only has an efficient solver similar to the classical $\ell_1$-norm \cite{2012L}, but also has better performance in dimensionality reduction and feature selection \cite{li2021sparse,li2025enhanced,chao2025efficient}. 
It should be pointed out that LAD has an obvious disadvantage, that is, it is nonsmooth. This makes it difficult to directly utilize gradient information, and thus brings huge challenges to the development of fast solvers.

\subsection{Our Contributions}

Motivated by recent advances in Huber regression \cite{fan2017estimation,xie2023huber,tyralis2025deep}, we propose the following regularized Huber PR method given by
\begin{equation}\label{min-HPR}
	\mathop {\min }\limits_{\mathbf{x}\in\mathbb{H}^p}~F(\mathbf{x}):=f(\mathbf{x})+\lambda\|\mathbf{x}\|_{1/2}^{1/2},
\end{equation}
where $f(\mathbf{x}):=\frac{1}{n}\sum_{i=1}^nh_{\alpha}(|\langle \mathbf{a}_i,\mathbf{x}\rangle|^2-b_i)$ and
\begin{equation}\nonumber
	{h_{\alpha}(u)} =
	\begin{cases}
		\frac{1}{2}u^2,&{\mbox{if}~|u|\leq\alpha},\\
		{\alpha |u|-\frac{1}{2}\alpha^2,}&{\mbox{if}~|u|>\alpha}
	\end{cases}
\end{equation}
is the Huber function with parameter $\alpha>0$ (A traditional choice is $\alpha=1.345$, see \cite{huber2011robust}), $\|\mathbf{x}\|_{1/2}^{1/2}=\sum_{i=1}^p |x_i|^{1/2}$ is the $\ell_{1/2}$-norm regularized term, and $\lambda>0$ is the regularization parameter. Three fundamental questions naturally arise: 
	(i)  Does the corresponding estimator possess desirable statistical properties? (ii) What optimality conditions characterize global minimizers, especially in complex domains? (iii) How to design a provably convergent optimization algorithm? For the first question, as far as we know, no relevant research has been reported in the existing literature.
For the second question, although numerous studies have been conducted on PGM, as previously mentioned, these investigations have primarily focused on optimization problems with real-valued variables. Consequently, these existing methods cannot be directly applied to solve (\ref{min-HPR}), particularly in the complex-valued case. For the third question, the first term $f(\cdot)$ lacks twice continuous differentiability. In the complex domain, this term is even nondifferentiable, which poses significant challenges in verifying whether the KL exponent of $F(\cdot)$ is 1/2. In light of these considerations, this paper systematically addresses these three fundamental questions through rigorous theoretical analysis and methodological development. 

In summary, the contributions of this paper are as follows.
\begin{enumerate}
  \item \textit{(Statistical Guarantee)} We construct a robust sparse PR method that combines $\ell_{1/2}$-norm regularization with the Huber loss. In fact, this is a novel framework for both real- and complex-valued signals. Especially, we establish the consistency of the corresponding estimators for the real-valued case when the dimension $p$ may increase to infinity with $n$, see Theorem \ref{consistency}.
  \item \textit{(Optimality Theory)} We theoretically establish optimality conditions for global minimizers. Notably, for the complex-valued case, we derive a fixed point inclusion property inspired by Wirtinger derivatives, see Theorem \ref{the1}.
  \item \textit{(Convergent Algorithm)} We develop an efficient MM algorithm and prove that the whole sequence converges to a vector that satisfies the fixed point inclusion. In particular, the convergence rate is linear under a mild condition, see Theorem \ref{conv-who}.
\end{enumerate}

\subsection{Outline}

The rest of this paper is organized as follows. Section \ref{Preliminaries} reviews notations, definitions and properties. Section \ref{Statistical} establishes the statistical properties of consistency. Section \ref{FPE} discusses the optimality conditions of global minimizers. Section \ref{convergence} provides the optimization algorithm with detailed convergence analysis. Section \ref{numerical} validates the robustness by numerical examples. Section \ref{conclusion} concludes this paper with future directions.

\section{Preliminaries}\label{Preliminaries}

\subsection{Notations} 

In this paper, the scalars are denoted by lowercase letters, e.g., $x$, vectors by bold lowercase letters, e.g., $\mathbf{x}$, and matrices by bold uppercase letters, e.g., $\mathbf{X}$.
Let $\mathbf{0}$ represent the vector or matrix with all components equal to zero, $\mathbf{I}_p$ represent the identity matrix of $p\times p$ and $\mathrm{i}$ represent the imaginary unit. 
In addition, let the superscripts $\top$ and $\mathrm{H}$ stand for transpose and conjugate transpose, respectively.
For a complex-valued vector $\mathbf{x}$, let $\mathcal{R}(\mathbf{x})$ and $\mathcal{I}(\mathbf{x})$ denote the real and imaginary parts, $\overline{\mathbf{x}}$ denote the complex conjugate and $\|\mathbf{a}\|$ denote the Euclidean norm.
For a matrix $\mathbf{X}$, let $\|\mathbf{X}\|_2$ denote the spectral norm and $\mathbf{X}^{\Gamma'\Gamma}$ denote the submatrix consisting of the rows and columns indexed by $\Gamma'$ and $\Gamma$.
For a symmetric matrix $\mathbf{X}$, let $\lambda_{\mathrm{min}}(\mathbf{X})$ represent the minimum eigenvalue.  Besides, denote $\mathbf{x}^\diamond=(x_1^\diamond,\cdots,x_p^\diamond)^\top$ as the true signal, $\|\mathbf{x}^\diamond\|_{0}$ as the sparsity of $\mathbf{x}^\diamond$ and $|\mathbf{x}^\diamond|_{\mathrm{min}}=\min_{|x_j^\diamond|>0}|x_j^\diamond|$.

For any $u\in\mathbb{R}$, the derivative of $h_{\alpha}(u)$ is $h'_{\alpha}(u)=\min\{\max\{u,-\alpha \},\alpha\}$. It then follows
\begin{equation}\label{hd-Lip}
	|h'_{\alpha}(u)|\leq\alpha~\mbox{and}~|h'_{\alpha}(u)-h'_{\alpha}(v) \textcolor[rgb]{0.00,0.00,1.00}{|} \leq |u-v|
\end{equation}
for any $u, v\in\mathbb{R}$. The second inequality also implies that the function $h'_{\alpha}(u)$ is Lipschitz continuous with constant $1$.

\subsection{Wirtinger Derivative}
Let $\mathbf{z}=\mathbf{x}+\mathrm{i}\mathbf{y}\in\mathbb{C}^d$ and its complex conjugate $\overline{\mathbf{z}}=\mathbf{x}-\mathrm{i}\mathbf{y}\in\mathbb{C}^d$ where $\mathbf{x}=\mathcal{R}(\mathbf{z})\in\mathbb{R}^d$ and $\mathbf{y}=\mathcal{I}(\mathbf{z})\in\mathbb{R}^d$. For a real- or complex-valued function $\varphi(\mathbf{z})=u(\mathbf{x},\mathbf{y})+\mathrm{i}v(\mathbf{x},\mathbf{y})$, it can be written in the form of $\varphi(\mathbf{z},\overline{\mathbf{z}})$ and the Wirtinger derivative is well defined as long as the real-valued functions
$u(\cdot)$ and $v(\cdot)$ are differentiable with respect to (w.r.t.) the real variables $\mathbf{x}$ and $\mathbf{y}$. In essence, the conjugate coordinates
$$
\begin{bmatrix} \mathbf{z}\\
	\overline{\mathbf{z}} \end{bmatrix}\in\mathbb{C}^d\times\mathbb{C}^d,~\mathbf{z}=\mathbf{x}+\mathrm{i}\mathbf{y},~\mbox{and}
~\overline{\mathbf{z}}=\mathbf{x}-\mathrm{i}\mathbf{y}$$
can serve as an alternative to $\mathbf{z}$.
Under such conditions, the Wirtinger derivatives can be defined as
$$\begin{aligned}
	&\frac{\partial \varphi}{\partial\mathbf{z}}:=\frac{\partial \varphi(\mathbf{z},\overline{\mathbf{z}})}{\partial\mathbf{z}}|_{\overline{\mathbf{z}}\mathrm{constant}}
	=[\frac{\partial\varphi(\mathbf{z},\overline{\mathbf{z}})}{\partial z_1},\cdots,\frac{\partial \varphi(\mathbf{z},\overline{\mathbf{z}})}{\partial z_d}],\\
	&\frac{\partial \varphi}{\partial\overline{\mathbf{z}}}:=\frac{\partial \varphi(\mathbf{z},\overline{\mathbf{z}})}{\partial\overline{\mathbf{z}}}|_{\mathbf{z}\mathrm{constant}}
	=[\frac{\partial \varphi(\mathbf{z},\overline{\mathbf{z}})}{\partial \overline{z}_1},\cdots,\frac{\partial \varphi(\mathbf{z},\overline{\mathbf{z}})}{\partial \overline{z}_d}],
\end{aligned}$$
where the partial derivatives $\frac{\partial}{\partial z_i}$ and $\frac{\partial}{\partial \overline{z}_i}$ are calculated by
$$\frac{\partial}{\partial z_i}=\frac{1}{2}(\frac{\partial}{\partial x_i}-\mathrm{i}\frac{\partial}{\partial y_i}),~\frac{\partial}{\partial \overline{z}_i}=\frac{1}{2}(\frac{\partial}{\partial x_i}+\mathrm{i}\frac{\partial}{\partial y_i}).$$
Here $z_i,\overline{z}_i,x_i,y_i$ are the $i$th coordinates of the vectors $\mathbf{z},\overline{\mathbf{z}},\mathbf{x},\mathbf{y}$, respectively.

Notice that the partial derivatives $\frac{\partial \varphi}{\partial\mathbf{z}}$ and $\frac{\partial \varphi}{\partial\overline{\mathbf{z}}}$ defined as above are row vectors. The Wirtinger gradient is defined as
$$\nabla\varphi(\mathbf{z})=[\frac{\partial \varphi}{\partial\overline{\mathbf{z}}}, \frac{\partial \varphi}{\partial\mathbf{z}}]^\mathrm{H}.$$
In fact, the integral form of Taylor's expansion is as follows.
\begin{lemma}\label{mvt-com}(Lemma A.2 in \cite{sun2018geometric})
	Let $\varphi(\mathbf{z}): \mathbb{C}^d\to\mathbb{R}$ hold the continuous first-order Witinger derivative. For any  $\mathbf{z}_0\in\mathbb{C}^d$, it follows that
	$$\varphi(\mathbf{z})=\varphi(\mathbf{z}_0)+\int_0^1\begin{bmatrix} \mathbf{z}-\mathbf{z}_0\\
		\overline{\mathbf{z}-\mathbf{z}_0} \end{bmatrix}^\mathrm{H}\nabla\varphi(\mathbf{z}_0+t(\mathbf{z}-\mathbf{z}_0))\mathrm{d}t.$$
\end{lemma}

\subsection{Subdifferential}
For a proper lower semicontinuous function $\varphi(\mathbf{u}): \mathbb{R}^d\to(-\infty, +\infty]$, the domain is 
$\mathrm{dom}(\varphi):=\{\mathbf{u}\in\mathbb{R}^d: \,\varphi(\mathbf{u})<+\infty\}.$
Then
$$\hat{\partial}\varphi(\mathbf{u}):=\{\mathbf{\varpi}\in\mathbb{R}^d: \,\liminf_{\mathbf{v}\neq\mathbf{u}, \mathbf{v}\to\mathbf{u}}\frac{1}{\|\mathbf{v}-\mathbf{u}\|}[\varphi(\mathbf{v})-\varphi(\mathbf{u})
-\langle\mathbf{\varpi},\mathbf{v}-\mathbf{u}\rangle]\geq0\}$$
is called the regular or Fr$\acute{e}$chet subdifferential of $\varphi(\cdot)$ at $\mathbf{u}\in\mathrm{dom}(\varphi)$, and
$$\partial\varphi(\mathbf{u}):=\{\mathbf{\varpi}\in\mathbb{R}^d: \,\exists \mathbf{u}^k\to\mathbf{u}, \varphi(\mathbf{u}^k)\to\varphi(\mathbf{u}), 
\hat{\partial}\varphi(\mathbf{u}^k)\ni\mathbf{\varpi}^k\to \textcolor[rgb]{0.00,0.00,1.00}{\mathbf{\varpi}}\}$$
is called the limiting or Mordukhovich subdifferential of $\varphi(\cdot)$ at $\mathbf{u}\in\mathrm{dom}(\varphi)$.
From Proposition 1.30 in \cite{mordukhovich2018variational} and Fermat's rule, the following lemma holds immediately.
\begin{lemma}\label{opt-diff}
Suppose $\varphi_1$ is continuously differentiable and $\varphi_2$ is properly lower semicontinuous. For any $\mathbf{u}\in\mathrm{dom}(\varphi_1+\varphi_2)$, it derives that
	\begin{equation}\label{diff-sumrule}
		\partial(\varphi_1(\mathbf{u})+\lambda\varphi_2(\mathbf{u}))=\nabla \varphi_1(\mathbf{u})+\lambda\partial\varphi_2(\mathbf{u}).\nonumber
	\end{equation}
	Furthermore, if $\hat{\mathbf{u}}$ is a local minimizer, then it should have $\mathbf{0}\in\nabla \varphi_1(\hat{\mathbf{u}})+\lambda\partial\varphi_2(\hat{\mathbf{u}}).$
\end{lemma}

\subsection{Kurdyka-Lojasiewicz Property}
For real-valued nonsmooth nonconvex optimization problems, Kurdyka-Lojasiewicz (KL) is a powerful tool for convergence analysis. A function $\varphi: \mathbb{R}^d\to(-\infty, +\infty]$ is said to satisfy the KL property with desingularizing function $\phi$ at
point $\bar{\mathbf{u}}\in\mathrm{dom}\partial\varphi$ if there exist $\eta>0$, a neighborhood $\mathcal{U}$ of $\bar{\mathbf{u}}$,
and a concave function $\phi:[0,\eta)\to\mathbb{R}_{+}$ such that 1) $\phi(0)=0$; 2) $\phi$ is continuous at 0 and continuously differentiable in $(0, \eta)$; 3) $\phi'(t)>0$ for all $t\in(0, \eta)$; 4) For all $\mathbf{u}\in\mathcal{U}\cap\{\mathbf{u}\in\mathbb{R}^d: \varphi(\bar{\mathbf{u}})<\varphi(\mathbf{u})<
	\varphi(\bar{\mathbf{u}})+\eta\}$, the following inequality holds
	\[
	\phi'(\varphi(\mathbf{u})-\varphi(\bar{\mathbf{u}}))\mathrm{dist}(\mathbf{0}, \partial\varphi(\mathbf{u}))\geq 1.
	\]
If $\varphi$ satisfies the KL property at each point of $\mathrm{dom}\partial\varphi$, then $\varphi$ is called a KL function.

Recall a subset of $\mathbb{R}^p$ is a  semialgebraic set if it  can  be written as a finite union of sets of the form
$\{\mathbf{u}\in\mathbb{R}^p: \theta_i(\mathbf{u})=0, \vartheta_i(\mathbf{u})<0\},$
where $\theta_i, \vartheta_i: \mathbb{R}^d\to\mathbb{R}$ are real polynomial functions. A function $\varphi(\mathbf{u}): \mathbb{R}^d\to(-\infty, +\infty]$ is called  semialgebraic if its graph
$$\mathrm{graph}_{\varphi}:=\{(\mathbf{u}^\top, t)\in\mathbb{R}^{d+1}: \,\varphi(\mathbf{u})=t\}$$
is a semialgebraic set of $\mathbb{R}^{d+1}$. It is known that the composition and finite sum of semialgebraic functions are also semialgebraic. For more details, see Subsection 4.3 of \cite{attouch2010proximal}.
Moreover, any semialgebraic function satisfies the KL property. For convenience, we state this fact as a lemma.
  
\begin{lemma}\label{semialg-prop}
	Any semialgebraic function satisfies the KL property with $\phi(t) = ct^{1 - \varrho}$ for some constant $c > 0$ and some rational number $\varrho \in [0,1)$, where $\varrho$ is referred to as the KL exponent.
\end{lemma}

\section{Statistical Guarantee}\label{Statistical}
In this section, we focus on the robustness for the regularized Huber model (\ref{min-HPR}) in the sense that the corresponding estimator should always be close to the underlying true parameters regardless of the distribution of the outliers. More precisely, we study the asymptotic error bound between them. To do that, we employ the sable condition introduced by \cite{eldar2014phase}  to achieve the stable recovery in real-valued PR problems when the data consists of random and noisy quadratic measurements of the target signal. It is stated in \cite{eldar2014phase} that  the measurement vector $\{\mathbf{a}_i\}\in\mathbb{R}^p$ are $\mu\geq0$ stable if 
		\begin{equation}\label{cond-ai-c1c2}
		\frac{1}{n}\sum_{i=1}^n|\langle \mathbf{u},\mathbf{a}_i\mathbf{a}_i^\top \mathbf{v}\rangle |\geq \mu\|\mathbf{u}\|\|\mathbf{v}\|, \quad \text{for all} ~\mathbf{u}, \mathbf{v}\in\mathbb{R}^p,
	\end{equation}
and the stability is established if the measurement vectors sample from sub-Gaussian distribution satisfying isotropic. Combing such stability and some other conditions, they further prove that based on the $\ell_q$-norm loss function, when $q\in(1,2]$, the error between the true signal and its estimator is bounded by a computable constant that depends only on the statistical properties of the noise. In addition, \cite{duchi2019solving} pointed out that the transformed data $\{\mathbf{a}_i/\|\mathbf{a}_i\|\}_{i=1}^n$ and $\{b_i/\|\mathbf{a}_i\|^2\}_{i=1}^n$ are likely to make the LAD problem better conditioned since the normalization can make the observations comparable to each other and yielding easier verification of their conditions. Indeed, \cite{candes2013phaselift} considered on the case that $\{\mathbf{a}_i\}_{i=1}^n$ are independently  sampled on the unit sphere because the normalization  provides the guarantee that their theorem can be equivalently stated in the case where $\{\mathbf{a}_i\}_{i=1}^n$ are
sampled on standard Gaussian vectors. Therefore, in this section, we assume that all $\mathbf{a}_i$ are unit vectors.

Note that \cite{eldar2014phase} assumed that the noise follows a symmetric sub-Gaussian distribution, which is a common assumption on the Huber method for high dimensional linear models, for example, \cite{fan2017estimation,lei2018asymptotics}. Moreover, given some conditions on $\alpha$ and $\lambda$ (depending on $n$), consistency results can be obtained.
Before continuing, we need the following lemma.

\begin{lemma}\label{prob-E} Define two random events $$E_1=\big\{\frac{1}{n}\sum_{i=1}^n|\varepsilon_i|\leq\mathbb{E}|\varepsilon_1|+t_n\big\},\quad E_2=\{\sup_{\mathbf{u},\mathbf{v}\in S^{p-1}}|\frac{1}{n}\sum_{i=1}^n\langle \mathbf{u}, \mathbf{a}_i\mathbf{a}_i^\top\mathbf{v}\rangle h'_{\alpha}(\varepsilon_i)|\leq t_n\alpha+\frac{\alpha}{n}\},$$
	where $S^{p-1}$ is the unit sphere defined by $S^{p-1}=\{\mathbf{u}\in\mathbb{R}^p: \|\mathbf{u}\|=1\}$. Under some mild conditions, it follows that
	\begin{equation}\label{prob-E1E2}
		\begin{aligned}\mathbb{P}\big( E_1\cap E_2\big)\to 1.\nonumber	\end{aligned}
	\end{equation}		
\end{lemma}
The proof of this lemma follows standard techniques, for completeness, we provide the detailed argument in the appendix.

\begin{theorem}\label{consistency}(Consistency) Let errors $\{\varepsilon_1,\ldots,\varepsilon_n\}$ be i.i.d. sub-exponential random variables with symmetric distribution around 0. Assume that the measurement vectors $\{\mathbf{a}_i\}\in\mathbb{R}^p$ is $\underline{C}_1$ stable with $\underline{C}_1>1/2$.  Assume that the parameter $\alpha$ satisfies \begin{equation}\label{cond-alpha-consistency}
		\alpha\geq\frac{6\mathbb{E}|\varepsilon_1|}{2(1-\rho_0)\underline{C}_1-1}.
	\end{equation}
	Assume there exists a positive constant $\underline{C}_2$ such that
	\begin{equation}\label{cond-ai-c2}
		\frac{1}{n}\sum_{i\in\mathcal{I}_{0}^{\mathrm{in}}}\langle \mathbf{u}, \mathbf{a}_i\mathbf{a}_i^\top\mathbf{v}\rangle^2\geq\underline{C}_2\|\mathbf{u}\| ^2\|\mathbf{v}\| ^2
		\quad \text{for all} ~\mathbf{u}, \mathbf{v}\in\mathbb{R}^p,
	\end{equation}
	where $\mathcal{I}_{0}^{\mathrm{in}}$ is defined by $\mathcal{I}_{0}^{\mathrm{in}}=\{i: |\varepsilon_i|\leq\rho_0\alpha\}$ and 
	$\rho_0$ is a positive constant satisfying $\rho_0<1-\frac{1}{2\underline{C}_1}$.
 Assume that \begin{equation}\label{cond-np-consistency} p\log n/n\to0 \quad\mbox{as}\quad n\to\infty,\end{equation}
 and \begin{equation}\label{cond-true-min}
 	|\mathbf{x}^\diamond|_{\mathrm{min}}\geq 2t_n^{1/6}.
 \end{equation} 
where $t_n=\sqrt{2(2p+1)\log(1+2n)/n}$. 
If the parameter $\lambda$ chooses \begin{equation}\label{cond-lambda-consistency}\lambda\big(\|\mathbf{x}^\diamond\|_{0}^{1/2}+\|\mathbf{x}^\diamond\|_{1/2}^{1/2}\big)
	\leq\frac{1}{4}\underline{C}_2t_n^{2/3}~\mbox{and}
	~\lambda\|\mathbf{x}^\diamond\|_0^{1/2}\geq \frac{\sqrt{2}}{2}\underline{C}_2|\mathbf{x}^\diamond|_{\mathrm{min}}^{1/2}\|\mathbf{x}^\diamond\|^2t_n,\end{equation} 
  then 
	$\min\{\|\hat{\mathbf{x}}-\mathbf{x}^\diamond\|, \|\hat{\mathbf{x}}+\mathbf{x}^\diamond\|\}=O_{\mathbb{P}}(r_n)$, i.e., $r_n^{-1}\min\{\|\hat{\mathbf{x}}-\mathbf{x}^\diamond\|, \hat{\mathbf{x}}+\mathbf{x}^\diamond\|\}$ is bounded in probability. Here $r_{n}=\frac{\sqrt{2}\lambda\|\mathbf{x}^\diamond\|_0^{1/2}}{\underline{C}_2|\mathbf{x}^\diamond|_{\mathrm{min}}^{1/2}\|\mathbf{x}^\diamond\|^2}$.
\end{theorem}

\begin{proof}
 This proof mainly consists of four steps, which will be described one by one below.

	\textit{Step 1)}. We start the proof by $F(\mathbf{x})-F(\mathbf{x}^\diamond)>0$ for any $\mathbf{x}$ satisfying $$\|\mathbf{x}-\mathbf{x}^\diamond\|\|\mathbf{x}+\mathbf{x}^\diamond\|\geq(1-\rho_0)\alpha$$ if the event $E_1$ occurs.
	It is not hard to check that \begin{equation}\label{huber-ulb}
		\alpha|u|-\frac{\alpha^2}{2}\leq	{h_{\alpha}(u)} \leq \alpha|u|.
	\end{equation}
	According to the relation (\ref{huber-ulb}) and the condition (\ref{cond-ai-c1c2}), we have
	\begin{equation}\label{Fx-lb}
		\begin{aligned}
			~F(\mathbf{x})-F(\mathbf{x}^\diamond)&=\frac{1}{n}\sum_{i=1}^nh_{\alpha}(\langle \mathbf{a}_i,\mathbf{x}\rangle^2-b_i)-\frac{1}{n}\sum_{i=1}^nh_{\alpha}(\varepsilon_i)+\lambda\|\mathbf{x}\|_{1/2}^{1/2}-\lambda\|\mathbf{x}^\diamond\|_{1/2}^{1/2}\\
			&\geq\frac{1}{n}\sum_{i=1}^n[\alpha|\langle \mathbf{a}_i,\mathbf{x}\rangle|^2-b_i|-\frac{\alpha^2}{2}]-\frac{\alpha}{n}\sum_{i=1}^n|\varepsilon_i|+\lambda\|\mathbf{x}\|_{1/2}^{1/2}-\lambda\|\mathbf{x}^\diamond\|_{1/2}^{1/2}\\
			&\geq\frac{1}{n}\sum_{i=1}^n[\alpha|\langle \mathbf{a}_i,\mathbf{x}\rangle^2-\langle \mathbf{a}_i,\mathbf{x}^\diamond\rangle^2|-\alpha|\varepsilon_i|]
			-\frac{\alpha}{n}\sum_{i=1}^n|\varepsilon_i|+\lambda\|\mathbf{x}\|_{1/2}^{1/2}-\lambda\|\mathbf{x}^\diamond\|_{1/2}^{1/2}-\frac{\alpha^2}{2}\\
			&\geq\alpha \underline{C}_1\|\mathbf{x}-\mathbf{x}^\diamond\|\|\mathbf{x}+\mathbf{x}^\diamond\|-\frac{2\alpha}{n}\sum_{i=1}^n|\varepsilon_i|+\lambda\|\mathbf{x}\|_{1/2}^{1/2}-\lambda\|\mathbf{x}^\diamond\|_{1/2}^{1/2}-\frac{\alpha^2}{2}\\	
			&\geq\alpha \underline{C}_1\|\mathbf{x}-\mathbf{x}^\diamond\|\|\mathbf{x}+\mathbf{x}^\diamond\|-\frac{2\alpha}{n}\sum_{i=1}^n|\varepsilon_i|-\lambda\|\mathbf{x}^\diamond\|_{1/2}^{1/2}-\frac{\alpha^2}{2}.			
		\end{aligned}
	\end{equation}
	Note that the condition (\ref{cond-alpha-consistency})  leads to
	$$\mathbb{E}|\varepsilon_1|\leq\frac{1}{6}\alpha[2((1-\rho_0)\underline{C}_1-1],$$
	and the condition (\ref{cond-lambda-consistency}) implies that for sufficiently large $n$, $$t_n\leq\frac{1}{18}\alpha[2((1-\rho_0)\underline{C}_1-1]\quad \mbox{and}\quad\lambda\|\mathbf{x}^\diamond\|_{1/2}^{1/2}\leq\frac{1}{20}\alpha^2[2((1-\rho_0)\underline{C}_1-1],$$
	we have
	\begin{equation}\label{sum-ub}
		\begin{aligned}
			2\alpha(\mathbb{E}|\varepsilon_1|+t_n)+\lambda\|\mathbf{x}^\diamond\|_{1/2}^{1/2}
			&\leq\frac{89}{180}\alpha^2[2((1-\rho_0)\underline{C}_1-1].
		\end{aligned}
	\end{equation}
	From (\ref{Fx-lb}) and (\ref{sum-ub}), we conclude that for sufficiently large $n$,
	\begin{equation}
		\begin{aligned}
			~F(\mathbf{x})-F(\mathbf{x}^\diamond)
			&\geq\alpha^2 \underline{C}_1(1-\rho_0)-\frac{2\alpha}{n}\sum_{i=1}^n|\varepsilon_i|-\lambda\|\mathbf{x}^\diamond\|_{1/2}^{1/2}-\frac{\alpha^2}{2}\\
			&\geq\frac{1}{2}\alpha^2[2(1-\rho_0)\underline{C}_1-1]-2\alpha(\mathbb{E}|\varepsilon_1|+t_n)-\lambda\|\mathbf{x}^\diamond\|_{1/2}^{1/2}\\
			&>0.		
		\end{aligned}
	\end{equation}

	\textit{Step 2)}. Under the event $E_2$, we proceed to prove $F(\mathbf{x})-F(\mathbf{x}^\diamond)>0$ for any $\mathbf{x}$ satisfying
	$$\|\mathbf{x}-\mathbf{x}^\diamond\|\|\mathbf{x}+\mathbf{x}^\diamond\|<  (1-\rho_0)\alpha$$
	and 
	$$\|\mathbf{x}-\mathbf{x}^\diamond\| \|\mathbf{x}+\mathbf{x}^\diamond\|> t_n^{1/3}.
	$$
	For each $i\in\mathcal{I}_{0}^{\mathrm{in}}=\{i:|\varepsilon_i|\leq\rho_0\alpha\}$, it follows that
	\begin{equation}
		\begin{aligned}\label{index_in}
			|\langle \mathbf{a}_i,\mathbf{x}\rangle^2-b_i|&=|\langle \mathbf{x}-\mathbf{x}^\diamond, \mathbf{a}_i\mathbf{a}_i^\top(\mathbf{x}+\mathbf{x}^\diamond)\rangle-\varepsilon_i|\\
			&\leq|\varepsilon_i|+|\langle \mathbf{x}-\mathbf{x}^\diamond, \mathbf{a}_i\mathbf{a}_i^\top(\mathbf{x}+\mathbf{x}^\diamond)\rangle|\\
			&\leq\rho_0\alpha+\|\mathbf{a}_i\|^2\|\mathbf{x}-\mathbf{x}^\diamond\|\|\mathbf{x}+\mathbf{x}^\diamond\|\\
			&\leq\alpha.
		\end{aligned}
	\end{equation}
	Thus, we have
	\begin{equation}\label{FFr1}
		\begin{aligned}
			&F(\mathbf{x})-F(\mathbf{x}^\diamond)\\
			&=\frac{1}{2n}\sum_{i\in\mathcal{I}_{0}^{\mathrm{in}}}(\langle \mathbf{a}_i,\mathbf{x}\rangle^2-b_i)^2+\frac{1}{n}\sum_{i\notin\mathcal{I}_{0}^{\mathrm{in}}}
			h_{\alpha}(\langle \mathbf{a}_i,\mathbf{x}\rangle^2-b_i)-\frac{1}{n}\sum_{i=1}^n		h_{\alpha}(\varepsilon_i)+\lambda\|\mathbf{x}\|_{1/2}^{1/2}-\lambda\|\mathbf{x}^\diamond\|_{1/2}^{1/2}\\
			&=\frac{1}{2n}\sum_{i\in\mathcal{I}_{0}^{\mathrm{in}}}[(\langle \mathbf{a}_i,\mathbf{x}\rangle^2-b_i)^2-	\varepsilon_i^2]+\frac{1}{n}\sum_{i\notin\mathcal{I}_{0}^{\mathrm{in}}}
			[h_{\alpha}(\langle \mathbf{a}_i,\mathbf{x}\rangle^2-b_i)-	h_{\alpha}(\varepsilon_i)]
			+\lambda\|\mathbf{x}\|_{1/2}^{1/2}-\lambda\|\mathbf{x}^\diamond\|_{1/2}^{1/2}\\
			&=\frac{1}{2n}\sum_{i\in\mathcal{I}_{0}^{\mathrm{in}}}\langle \mathbf{x}-\mathbf{x}^\diamond, \mathbf{a}_i\mathbf{a}_i^\top(\mathbf{x}+\mathbf{x}^\diamond)\rangle^2
			-\frac{1}{n}\sum_{i\in\mathcal{I}_{0}^{\mathrm{in}}}\langle \mathbf{x}-\mathbf{x}^\diamond, \mathbf{a}_i\mathbf{a}_i^\top(\mathbf{x}+\mathbf{x}^\diamond)\rangle\varepsilon_i\\
			&\qquad+\frac{1}{n}\sum_{i\notin\mathcal{I}_{0}^{\mathrm{in}}}[h_{\alpha}(\varepsilon_i-\langle \mathbf{x}-\mathbf{x}^\diamond, \mathbf{a}_i\mathbf{a}_i^\top(\mathbf{x}+\mathbf{x}^\diamond)\rangle)-h_{\alpha}(\varepsilon_i)]+\lambda\|\mathbf{x}\|_{1/2}^{1/2}-\lambda\|\mathbf{x}^\diamond\|_{1/2}^{1/2}\\
			&\geq\frac{1}{2n}\sum_{i\in\mathcal{I}_{0}^{\mathrm{in}}}\langle \mathbf{x}-\mathbf{x}^\diamond, \mathbf{a}_i\mathbf{a}_i^\top(\mathbf{x}+\mathbf{x}^\diamond)\rangle^2
			-\frac{1}{n}\sum_{i\in\mathcal{I}_{0}^{\mathrm{in}}}\langle \mathbf{x}-\mathbf{x}^\diamond, \mathbf{a}_i\mathbf{a}_i^\top(\mathbf{x}+\mathbf{x}^\diamond)\rangle h'_{\alpha}(\varepsilon_i)\\
			&\qquad-\frac{1}{n}\sum_{i\notin\mathcal{I}_{0}^{\mathrm{in}}}\langle \mathbf{x}-\mathbf{x}^\diamond, \mathbf{a}_i\mathbf{a}_i^\top(\mathbf{x}+\mathbf{x}^\diamond)\rangle h'_{\alpha}(\varepsilon_i)	+\lambda\|\mathbf{x}\|_{1/2}^{1/2}-\lambda\|\mathbf{x}^\diamond\|_{1/2}^{1/2}\\
			&=\frac{1}{2n}\sum_{i\in\mathcal{I}_{0}^{\mathrm{in}}}\langle \mathbf{x}-\mathbf{x}^\diamond, \mathbf{a}_i\mathbf{a}_i^\top(\mathbf{x}+\mathbf{x}^\diamond)\rangle^2
			-\frac{1}{n}\sum_{i=1}^n\langle \mathbf{x}-\mathbf{x}^\diamond, \mathbf{a}_i\mathbf{a}_i^\top(\mathbf{x}+\mathbf{x}^\diamond)\rangle h'_{\alpha}(\varepsilon_i)
			+\lambda\|\mathbf{x}\|_{1/2}^{1/2}-\lambda\|\mathbf{x}^\diamond\|_{1/2}^{1/2}.
		\end{aligned}
	\end{equation}
	For large enough $n$, using the conditions (\ref{cond-np-consistency}) and (\ref{cond-lambda-consistency}) leads to $$\alpha t_n^{4/3}\leq\frac{1}{5}\underline{C}_2t_n^{2/3},\quad \frac{1}{n} t_n^{1/3}\alpha\leq\frac{1}{20}\underline{C}_2t_n^{2/3}, \mbox{and}\quad\lambda\|\mathbf{x}^\diamond\|_{1/2}^{1/2}\leq\frac{1}{4}\underline{C}_2t_n^{2/3},$$
	we conclude from (\ref{FFr1}) and $t_n\leq t_1$ that \begin{equation}
		\begin{aligned}
			F(\mathbf{x})-F(\mathbf{x}^\diamond)
			&\geq\frac{1}{2}\underline{C}_2\|\mathbf{x}-\mathbf{x}^\diamond\|^2\|\mathbf{x}+\mathbf{x}^\diamond\|^2-\|\mathbf{x}-\mathbf{x}^\diamond\|\|\mathbf{x}+\mathbf{x}^\diamond\|(t_n+\frac{1}{n})\alpha
			-\lambda\|\mathbf{x}^\diamond\|_{1/2}^{1/2}\\
			&>\frac{1}{2}\underline{C}_2t_n^{2/3}-t_n^{4/3}\alpha-\frac{1}{n} t_n^{1/3}\alpha
			-\frac{1}{4}\underline{C}_2t_n^{2/3}\\
			&\geq0
		\end{aligned}
	\end{equation}
	 when $E_2$ occurs.

	\textit{Step 3)}. Under the event $E_2$, we proceed to prove $F(\mathbf{x})-F(\mathbf{x}^\diamond)>0$ for any $\mathbf{x}$ satisfying $$\|\mathbf{x}-\mathbf{x}^\diamond\|\|\mathbf{x}+\mathbf{x}^\diamond\|\leq  t_n^{1/3}$$
	and 
	\begin{equation}\label{nonasyp-ub}
		\min\{\|\mathbf{x}-\mathbf{x}^\diamond\|, \|\mathbf{x}+\mathbf{x}^\diamond\|\}\geq\frac{2t_n\alpha}{\underline{C}_2\|\mathbf{x}^\diamond\|}+\frac{2\alpha}{\underline{C}_2\|\mathbf{x}^\diamond\|n}
	+r_n.
\end{equation}
Notice that $\mathrm{supp}(\mathbf{x}^\diamond)\subseteq\mathrm{supp}(\mathbf{x})$ when $\|\mathbf{x}-\mathbf{x}^\diamond\|\|\mathbf{x}+\mathbf{x}^\diamond\|\leq  t_n^{1/3}$. Otherwise, using the condition $|\mathbf{x}^\diamond|_{\mathrm{min}}\geq2t_n^{1/6}$, we have
\begin{equation}
	\begin{aligned}
	4t_n^{1/3}\leq|\mathbf{x}^\diamond|_{\mathrm{min}}^2\leq\|\mathbf{x}-\mathbf{x}^\diamond\|\|\mathbf{x}+\mathbf{x}^\diamond\|\leq  t_n^{1/3},
	\end{aligned}\nonumber
\end{equation}
which is absurd.
	Without loss of generality, we assume $$\|\mathbf{x}-\mathbf{x}^\diamond\|\leq\|\mathbf{x}+\mathbf{x}^\diamond\|.$$
	Noting for each $j\in\mathrm{supp}(\mathbf{x}^\diamond)$, we can get an upper bound of the $j$th element of such $\mathbf{x}$ as follows
	$$|\mathbf{x}_j|\geq|\mathbf{x}_j^\diamond|-\|\mathbf{x}-\mathbf{x}^\diamond\|>|\mathbf{x}^\diamond|_{\mathrm{min}}-t_n^{1/6}
	\geq|\mathbf{x}^\diamond|_{\mathrm{min}}/2,$$
	which together with the concaveness of the function $t^{1/2}$ for $t>0$ yields that
	\begin{equation}
		\begin{aligned}
			\sum_{j\in\mathrm{supp}(\mathbf{x}^\diamond)}(|x_j|^{1/2}-|x_j^\diamond|^{1/2})
			&\geq\frac{1}{2}\sum_{j\in\mathrm{supp}(\mathbf{x}^\diamond)}|x_j|^{-1/2}(|x_j|-|x_j^\diamond|)\\
			&\geq-\frac{1}{2}\sum_{j\in\mathrm{supp}(\mathbf{x}^\diamond)}|x_j|^{-1/2}|x_j-x_j^\diamond|\\
			&>-\frac{\sqrt{2}}{2}|\mathbf{x}^\diamond|_{\mathrm{min}}^{-1/2}\|\mathbf{x}^\diamond\|_0^{\frac{1}{2}}\|\mathbf{x}-\mathbf{x}^\diamond\|.
		\end{aligned}\nonumber
	\end{equation}

	If $E_2$ occurs, we conclude from the inequality (\ref{FFr1}), the condition (\ref{cond-ai-c2}) and the above inequality that
	\begin{equation}
		\begin{aligned}
			&F(\mathbf{x})-F(\mathbf{x}^\diamond)\\
			&>\frac{1}{2}\underline{C}_2\|\mathbf{x}-\mathbf{x}^\diamond\|^2\|\mathbf{x}+\mathbf{x}^\diamond\|^2-\|\mathbf{x}-\mathbf{x}^\diamond\|\|\mathbf{x}+\mathbf{x}^\diamond\|(t_n+\frac{1}{n})\alpha
			-\frac{\sqrt{2}\lambda}{2}|\mathbf{x}^\diamond|_{\mathrm{min}}^{-1/2}\|\mathbf{x}^\diamond\|_0^{\frac{1}{2}}\|\mathbf{x}-\mathbf{x}^\diamond\|\\
			&\geq\frac{1}{2}\|\mathbf{x}-\mathbf{x}^\diamond\|\|\mathbf{x}+\mathbf{x}^\diamond\|\Big(\underline{C}_2\|\mathbf{x}-\mathbf{x}^\diamond\|\|\mathbf{x}+\mathbf{x}^\diamond\|
			-2(t_n+\frac{1}{n})\alpha
			-\sqrt{2}\lambda\|\mathbf{x}^\diamond\|_0^{\frac{1}{2}}|\mathbf{x}^\diamond|_{\mathrm{min}}^{-1/2}
			\|\mathbf{x}+\mathbf{x}^\diamond\|^{-1}\Big)\\
			&\geq\frac{1}{2}\|\mathbf{x}-\mathbf{x}^\diamond\|\|\mathbf{x}^\diamond\|\Big(\underline{C}_2\|\mathbf{x}-\mathbf{x}^\diamond\|\|\mathbf{x}^\diamond\|-2(t_n+\frac{1}{n})\alpha
			-\sqrt{2}\lambda\|\mathbf{x}^\diamond\|_0^{\frac{1}{2}}|\mathbf{x}^\diamond|_{\mathrm{min}}^{-1/2}\|\mathbf{x}^\diamond\|^{-1}
			\Big)\\
			&\geq 0,
		\end{aligned}\nonumber
	\end{equation}
	where the third inequality derives from the inequality $\|\mathbf{x}^\diamond\|\leq\|\mathbf{x}+\mathbf{x}^\diamond\|$ following from the assumption $\|\mathbf{x}-\mathbf{x}^\diamond\|\leq\|\mathbf{x}+\mathbf{x}^\diamond\|$, and the last is due to  $\mathbf{x}$ satisfying (\ref{nonasyp-ub}).

		\textit{Step 4)}. According to  \textit{Steps 1)-3)},  we obtain  
	$$\min\{\|\hat{\mathbf{x}}-\mathbf{x}^\diamond\|, \|\hat{\mathbf{x}}+\mathbf{x}^\diamond\|\}\leq\frac{2t_n\alpha}{\underline{C}_2\|\mathbf{x}^\diamond\|}+\frac{2\alpha}{\underline{C}_2\|\mathbf{x}^\diamond\|n}
	+r_n$$
	if the random events $E_1$ and $E_2$ occur simultaneously. 
This, together with the conditions (\ref{cond-np-consistency}) and (\ref{cond-lambda-consistency}) derives that $t_n\leq r_n$ and $nr_n\geq1$ for large enough $n$. As a direct result, we get 
	$$\frac{2t_n\alpha}{\underline{C}_2\|\mathbf{x}^\diamond\|}+\frac{2\alpha}{\underline{C}_2\|\mathbf{x}^\diamond\|n}
	+r_n\leq(1+\frac{3\alpha}{\underline{C}_2\|\mathbf{x}^\diamond\|})r_n.$$
Therefore, combining the above inequality and Lemma \ref{prob-E}, the desired result is obtained.
\end{proof}

 \begin{remark}\label{explanation-con}  
 The work \cite{duchi2019solving} considered the following corruption model
\begin{equation}\nonumber
	b_i =
	\begin{cases}
		\langle \mathbf{a}_i, \mathbf{x}^\diamond\rangle^2,& i\in\mathbb{I}_1,\\
		{\varepsilon_i,}&{i\in\mathbb{I}_2},
	\end{cases}
\end{equation}	
where the index set $\mathbb{I}_1$ and $\mathbb{I}_2$ satisfy $\mathbb{I}_1\cap\mathbb{I}_2=\emptyset$ and $\mathbb{I}_1\cup\mathbb{I}_2=\{1,\cdots,n\}$, and the indices $i\in\mathbb{I}_2$ are chosen randomly. If the proportion of corrupted measurements is relatively small as stated in \cite{duchi2019solving}, the condition (\ref{cond-ai-c2}) is also reasonable because $ \mathbb{I}_1\subseteq\mathcal{I}_{0}^{\mathrm{in}}$.  
For the extreme case that (\ref{pr}) is not corrupted by any noise, it follows that $\mathcal{I}_{0}^{\mathrm{in}}=\{1,\cdots,n\}$.  In such scenario, the condition (\ref{cond-ai-c2}) combing with Cauchy's inequality and $\|\mathbf{a}_i\|=1$ enable us to derive that for all $\mathbf{u}, \mathbf{v}\in\mathbb{R}^p$,
$$\underline{C}_2\|\mathbf{u}\| ^2\|\mathbf{v}\| ^2\leq\frac{1}{n}\sum_{i=1}^n\langle \mathbf{u}, \mathbf{a}_i\mathbf{a}_i^\top\mathbf{v}\rangle^2\leq\|\mathbf{u}\| ^2\|\mathbf{v}\| ^2,$$
or equivalently,
$$\underline{C}_2\|\mathbf{u}\mathbf{v}^\top\| ^2\leq\frac{1}{n}\sum_{i=1}^n\langle \mathbf{a}_i\mathbf{a}_i^\top, \mathbf{u}\mathbf{v}^\top\rangle^2\leq\|\mathbf{u}\mathbf{v}^\top\| ^2.$$
This is a type of restricted isometry property condition for low-rank matrix recovery which is also applied to PR problems with the LS formulation by SDP. Many random sensing designs are known to satisfy such conditions with high probability. See the overview  \cite{chi2019nonconvex} for more details.
For the linear model,  the work \cite{fan2017estimation} states that the parameter $\alpha$ plays a key role in adapting to various noise types. Particularly, the authors recommend selecting a sufficiently large  $\alpha$ to robustify the estimation when dealing with heavy-tailed symmetric noise distributions. While the condition (\ref{cond-alpha-consistency}) indicates that the lower bound of $\alpha$ does not need to be excessively large and is instead related to $\mathbb{E}|\varepsilon_1|$, this theoretical quantity is typically unknown in practical applications. Therefore, in implementation, the parameter 
$\alpha$ is commonly determined through cross-validation to achieve optimal robustness.
\end{remark}

\begin{remark}\label{explanation-con-lambda}  In \cite{huang2008asymptotic}, the authors investigated the asymptotic properties of bridge estimators for the linear model when
$p$ may increase to infinity with $n$ but $p=\mathrm{o}(n/\log n)$ ,i.e., $p\log n/n\to0$. To establish consistency and sparsity, their conditions on $\lambda$ can be simplified to $\lambda\sqrt{ns}\to0$ and $\lambda(n/p)^{3/4}\to\infty$ as $n\to\infty$ for the case where the bridge parameter takes the value $1/2$ and $s$ stand for the number of nonzero coefficients.
Moreover, they assume that the nonzero elements of $\mathbf{x}^\diamond$ in magnitude are bounded from below and above by positive constants $c_1$ and $c_2$, i.e.,
\begin{equation}\label{cond-mintrue2}
c_1\leq	|\mathbf{x}^\diamond|_{\mathrm{min}} \leq\max_{|x_j^\diamond|>0}|x_j^\diamond|\leq c_2.\nonumber
\end{equation}
This condition on $\mathbf{x}^\diamond$ is also employed  by \cite{fan2018variable} to study the asymptotic properties of regularized LS estimators in sparse quadratic regression, which includes sparse PR as a special case. Although our condition (\ref{cond-np-consistency}) is stronger than  $p/n\to0$, it still allows the dimension $p$  to be either fixed or diverge to infinity at a rate of $o(n/\log n)$. In such scenario, the condition (\ref{cond-true-min}) permits that $|\mathbf{x}^\diamond|_{\mathrm{min}}$ can tend to zero as $n\to \infty$. Suppose that $\lambda=(p\log n/n)^\varrho/\sqrt{\|\mathbf{x}^\diamond\|_0}$ with $\varrho\in(1/3, 1/2)$. Then our condition (\ref{cond-lambda-consistency}) holds when  the number of nonzero elements of $\mathbf{x}^\diamond$  is fixed and $\mathbf{x}^\diamond$ is bounded as mentioned above. It is clear that such $\lambda$ results in $\lambda\sqrt{ns}\to\infty$ and $\lambda(n/p)^{3/4}\to\infty$. Compared to the penalized parameter used in \cite{huang2008asymptotic},  the penalized parameter $\lambda$ used here is larger, which is reasonable and aligns with our intuition, as (\ref{pr}) is more complex than linear models.\end{remark}

\begin{remark}\label{explanation-con-rate}  	  Theorem \ref{consistency} yields the rate of convergence $$O_{\mathbb{P}}(r_n)=O_{\mathbb{P}}(\lambda\|\mathbf{x}^\diamond\|_0^{\frac{1}{2}}/(|\mathbf{x}^\diamond|_{\mathrm{min}}^{1/2}\|\mathbf{x}^\diamond\|^2)),$$ which can be simplified to $O_{\mathbb{P}}(\lambda/\|\mathbf{x}^\diamond\|_0^{\frac{1}{2}})$ under the assumption that the nonzero elements of $\mathbf{x}^\diamond$ in magnitude is bounded from below and above by positive constants like in \cite{huang2008asymptotic}. Notice that $r_n$ can be chosen as $t_n$. Therefore, the rate of convergence is at most $O_{\mathbb{P}}(r_n)=O_{\mathbb{P}}(\sqrt{p\log n/n})$ which is slower than $O_{\mathbb{P}}(\sqrt{p/n})$ for  bridge estimators for linear model provided by \cite{huang2008asymptotic}, and $O_{\mathbb{P}}(\sqrt{s\log (p/s)/n})$ for PR provided by  \cite{eldar2014phase}. Here $s$ is an upper bound of $\|\mathbf{x}^\diamond\|_0$. According to the proof of Lemma \ref{prob-E}, one can see that the factor $p\log n$ derives from the covering number of the unit sphere $\{\mathbf{u}\in\mathbb{R}^p: \|\mathbf{u}\|=1\}$ and the quadratic relationship in (\ref{pr}). Consequently, when $p$ is fixed,  the convergence rate simplifies  $O_{\mathbb{P}}(\sqrt{\log n /n})$, and when $p$ is diverging at the rate  $p=o(n/\log n)$ as $n\to\infty$, it improves to $O_{\mathbb{P}}(\sqrt{p\log p/n})$. 
 It is worth noting that the convergence rate in \cite{eldar2014phase}  is based on the $\ell_q$-norm loss with bounded constraints and requires that the sparsity of the optimal solution does not exceed $s$. Additionally, the upper bound of the $\ell_q$-norm empirical risk of the optimal solution must not exceed a certain constant, which is related to
			$\mathbb{E}|\varepsilon_1|^q$. In other words, their results critically depend on the condition that the estimator itself satisfies certain properties to achieve the claimed convergence rate. However, it remains theoretically unverified whether their estimation approach can guarantee that these required conditions hold for the estimator. In contrast, our results impose no specific requirements on the estimator, relying solely on the measurement vectors $\{\mathbf{a}_i\}$ and some controllable parameters such as $\lambda$ and $\alpha$. For the case $p>n$, we conjecture that, with the aid of additional conditions, we can improve our result to $$O_{\mathbb{P}}(\sqrt{\|\mathbf{x}^\diamond\|_0\log (p/\|\mathbf{x}^\diamond\|_0)/n})$$ for a locally optimal solution to (\ref{min-HPR}). Concurrently, based on  this rate, it is determined that the sampling size can be significantly reduced.\end{remark}

\section{Optimality Theory}\label{FPE}
Recall that  (\ref{min-HPR}) minimizes the sum of two functions where the first term is smooth as $\mathbf{x}\in\mathbb{R}^p$ and the second term is $\ell_{1/2}$-norm regularization. 
This naturally renders us to employ PGM to solve it. Remember  that PGM can be interpreted as a fixed point iteration. In this section, we will prove a global minimizer of (\ref{min-HPR}) satisfies a fixed point inclusion.  To achieve this aim, this section is organized into three parts. First, we revisit the half-thresholding operator used for addressing $\ell_{1/2}$-norm regularization problems. Subsequently, we introduce the gradient of the function $f(\cdot)$, with particular emphasis on the Wirtinger gradient for the case of complex variables. Finally, by employing the Majorization-Minimization (MM) method, we derive the fixed point inclusion.

\cite{2012L} and \cite{zeng2012sparse} have provided real- and complex-valued iterative half-thresholding algorithms for solving $\ell_{1/2}$-norm regularized minimization problems as follows.
\begin{lemma}\label{half-solution} Any solution for the problem
	\begin{equation}
		\min\limits_{\mathbf{v}\in \mathbb{H}^p}~\|\mathbf{v}-\mathbf{\xi}\|^2+\mu\|\mathbf{v}\|_{1/2}^{1/2}\nonumber
	\end{equation}
 can be expressed by the half-thresholding operator
	\begin{equation}
		\mathcal{H}_{\mu}(\mathbf{\xi})=(\chi_\mu(\xi_1),\chi_\mu(\xi_2),...,\chi_\mu(\xi_p)),\nonumber
	\end{equation}
	where
	\begin{equation}
		\chi_\mu(t)=
		\begin{cases}
			{\frac{2}{3}t(1+\cos(\frac{2\pi}{3}-\frac{2}{3}\mathrm{arccos}(\frac{\mu}{8}(\frac{|t|}{3})^{-3/2}))),}&{|t|>\frac{\sqrt[3]{54}}{4}\mu^{2/3}},\nonumber\\
			{(\mu/2)^{2/3}~\mathrm{or}~0,} &{|t|=\frac{\sqrt[3]{54}}{4}\mu^{2/3}},\\
			{0,}&\textrm{otherwise}.\nonumber
		\end{cases}
	\end{equation}
\end{lemma}

\begin{remark}\label{equi-half}
		It is worth mentioning that the following two minimization problems with complex variables 
		$$\min_{v\in\mathbb{C}}|v-t|^2+\mu|v|^{1/2}, \quad \min_{v\in\mathbb{C}}|v-t|^2
		+\mu(\mathcal{R}(v)^2+\mathcal{I}(v)^2)^{1/4}$$
		are equivalent, but the following problems
		$$\min_{v\in\mathbb{C}}|v-t|^2+\mu|v|^{1/2}, \quad \min_{v\in\mathbb{C}}|v-t|^2+\mu(|\mathcal{R}(v)|^{1/2}+|\mathcal{I}(v)|^{1/2})$$
		are not equivalent.
		Indeed, it is easy to check that if $\mathcal{R}(v)\mathcal{I}(v)\neq0$ then
		$$|\mathcal{R}(v)|^{1/2}+|\mathcal{I}(v)|^{1/2}>(\mathcal{R}(v)^2+\mathcal{I}(v)^2)^{1/4}=|v|^{1/2}.$$
		Therefore, one can conclude from the fact that $(\varphi_\mu(\mathcal{R}(t)),\varphi_\mu(\mathcal{R}(t)))^\top$ and
		$\varphi_\mu(t)$ are the optimal solutions of the corresponding minimization problems that
		$$\begin{aligned}&(\varphi_\mu(\mathcal{R}(t))-\mathcal{R}(t))^2+(\varphi_\mu(\mathcal{I}(t))-\mathcal{I}(t))^2+
			\mu(|\varphi_\mu(\mathcal{I}(t))|^{1/2}+\varphi_\mu(\mathcal{I}(t))|^{1/2})\\
			&>(\varphi_\mu(\mathcal{R}(t))-\mathcal{R}(t))^2+(\varphi_\mu(\mathcal{I}(t))-\mathcal{I}(t))^2+
			\mu(\varphi_\mu(\mathcal{I}(t))^2+\varphi_\mu(\mathcal{I}(t))^2)^{1/4}\\
			&\geq |\varphi_\mu(t)-t|^2+\mu|\varphi_\mu(t)|^{1/2}
		\end{aligned}$$
		when $|\mathcal{R}(t)|>\frac{\sqrt[3]{54}}{4}\mu^{2/3}$ and $|\mathcal{I}(t)|>\frac{\sqrt[3]{54}}{4}\mu^{2/3}$.
		This means that the
		complex variable setting can not be treated in the fashion that identifying $\mathbf{x}\in\mathbb{C}^p$ with $\tilde{\mathbf{x}}=[\mathcal{R}(\mathbf{x})^\top, \mathcal{I}(\mathbf{x})^\top]^\top\in\mathbb{R}^{2p}$ minimizes $F$ via the following problem
		$$\min_{\tilde{\mathbf{x}}}f+\lambda\|\tilde{\mathbf{x}}\|_{1/2}^{1/2},$$
		even though $f$ can be regarded as a function in the real domain. By the way, there exists a similar relation for $|v|^\varrho$ with $\varrho\in[0,1]$, i.e., $$|\mathcal{R}(v)|^{\varrho}+|\mathcal{I}(v)|^{\varrho}>(\mathcal{R}(v)^2+\mathcal{I}(v)^2)^{\frac{\varrho}{2}}=|v|^\varrho$$
		if $\mathcal{R}(v)\mathcal{I}(v)\neq0$. As a direct result, one cannot get the optimal solution of $\min_{\tilde{\mathbf{x}}}f+\lambda\|\mathbf{x}\|_\varrho^\varrho$ via 
		$\min_{\tilde{\mathbf{x}}}f+\lambda\|\tilde{\mathbf{x}}\|_\varrho^\varrho$ (Especially, $\|\tilde{\mathbf{x}}\|_\varrho^\varrho$ stands for $\|\tilde{\mathbf{x}}\|_0$ as $\varrho=0$)
		even that the Huber function in $f$ is replaced by the LS function.
\end{remark}

We now provide the gradient of $f(\mathbf{x})$. For the real-valued case, i.e., $\mathbb{H}=\mathbb{R}$, one can see that the gradient can be computed by
\begin{equation}
	\nabla f(\mathbf{x})=\frac{2}{n}\sum\limits_{i=1}^n h_{\alpha}'(|\langle \mathbf{a}_i, \mathbf{x}\rangle|^2-b_i)\langle \mathbf{a}_i, \mathbf{x}\rangle \mathbf{a}_i.\nonumber
\end{equation} 
For the complex-valued case, i.e., $\mathbb{H}=\mathbb{C}$, it is clear that $f(\mathbf{x})$ is not holomorphic and thus not differentiable in the standard complex variables sense. In fact, a nonconstant real-valued function of a complex variable is not complex analytic. 
Although $f(\mathbf{x})$ has a real gradient when partial derivatives are taken w.r.t the real and imaginary parts of complex variables, one must transfer from the mapping between $\mathbb{C}^n $ and $\mathbb{R}$ to $\mathbb{R}^{2n}$ and $\mathbb{R}$.
This means that calculating $f(\mathbf{x})$ directly in the real domain leads to  producing awkward expressions and cumbersome computation. 
Recall that the LS problem for PR is also not complex-differentiable, and \cite{candesWF} introduced Wirtinger derivatives to deal with the non-differentiability. 
Following a similar idea, we can study the gradient of $f(\mathbf{x})$ in the case $\mathbb{H}=\mathbb{C}$ with the help of Wirtinger calculus.
Notice that
$$f(\mathbf{x})=\frac{1}{n}\sum_{i=1}^nh_{\alpha}(\langle \mathbf{a}_i, \mathbf{x}\rangle \langle \overline{\mathbf{a}}_i, \overline{\mathbf{x}}\rangle-b_i)\triangleq f(\mathbf{x},\overline{\mathbf{x}}).$$
By simple calculation, one can conclude from the chain rule that
$$\frac{\partial f}{\partial \mathbf{x}}=\frac{1}{n}\sum_{i=1}^nh_{\alpha}'(|\langle \mathbf{a}_i, \mathbf{x}\rangle|^2-b_i)\langle \overline{\mathbf{a}}_i, \overline{\mathbf{x}}\rangle \mathbf{a}_i^\mathrm{H},$$
and
$$\frac{\partial f}{\partial \bar{\mathbf{x}}}=\frac{1}{n}\sum_{i=1}^nh_{\alpha}'(|\langle \mathbf{a}_i, \mathbf{x}\rangle|^2-b_i)\langle \mathbf{a}_i, \mathbf{x}\rangle \mathbf{a}_i^\top.$$
The Wirtinger gradient is defined as
\begin{equation}
\nabla{}f(\mathbf{x})=
\begin{bmatrix} \nabla_{\mathbf{x}}f(\mathbf{x}) \\
\nabla_{\overline{\mathbf{x}}}f(\mathbf{x}) \end{bmatrix}\nonumber
\end{equation}
with $\nabla_{\mathbf{x}}f(\mathbf{x})=(\frac{\partial f}{\partial \mathbf{x}})^\mathrm{H}$ and $\nabla_{\overline{\mathbf{x}}}f(\mathbf{x})=(\frac{\partial f}{\partial \overline{\mathbf{x}}})^\mathrm{H}$. It is easy to check that $\nabla_{\overline{\mathbf{x}}}f(\mathbf{x})=\overline{\nabla_{\mathbf{x}}f(\mathbf{x})}$.
Then, the gradient w.r.t. $\mathbf{x}$ can be derived as the form of
\begin{equation}
\nabla_{\mathbf{x}} f(\mathbf{x})=\frac{1}{n}\sum\limits_{i=1}^n h'_{\alpha}(|\langle \mathbf{a}_i, \mathbf{x}\rangle|^2-b_i)\langle \mathbf{a}_i, \mathbf{x}\rangle \overline{\mathbf{a}}_i.\nonumber
\end{equation}
For ease of expression, we use the uniform notation
\begin{equation}
g(\mathbf{x}):=\frac{1}{n}\sum\limits_{i=1}^n h'_{\alpha}(|\langle \mathbf{a}_i, \mathbf{x}\rangle|^2-b_i)\langle \mathbf{a}_i, \mathbf{x}\rangle \overline{\mathbf{a}}_i,\nonumber
\end{equation}
which implies that $\nabla f(\mathbf{x})=2g(\mathbf{x})$ for the case $\mathbb{H}=\mathbb{R}$ and $\nabla_{\mathbf{x}} f(\mathbf{x})=g(\mathbf{x})$
for the case $\mathbb{H}=\mathbb{C}$.

Below, we study the existence of the optimal solution $\hat{\mathbf{x}}$ of (\ref{min-HPR}) and employ the MM  techniques to solve it. To do that, we introduce the auxiliary function
\begin{equation}
F_\tau(\mathbf{x},\mathbf{y}):=f(\mathbf{y})+2\mathcal{R}(\langle g(\mathbf{y}),\mathbf{x}-\mathbf{y}\rangle) +\frac{1}{2\tau}\|\mathbf{x}-\mathbf{y}\|^2+\lambda\|\mathbf{x}\|_{1/2}^{1/2}\nonumber
\end{equation}
for any $\mathbf{x},\mathbf{y}\in \mathbb{H}^p$. Here, $\tau>0$ is a tuning parameter. It is evident that $F(\mathbf{x})=F_{\tau}(\mathbf{x},\mathbf{x})$. However, it is unclear whether the function $F_{\tau}(\mathbf{x},\mathbf{x})$ can serve as an upper bound for $F(\mathbf{x})$ for any $\mathbf{x},\mathbf{y}\in\mathbb{H}^p$, as the gradient $g(\cdot)$ lacks of the global Lipschitz continuity. The following two lemmas indicate that $g(\cdot)$ is Lipschitz on a level set of $F(\cdot)$, and the minimizer of $\min_{\mathbf{x}\in\mathbb{H}^p} F_\tau(\mathbf{x},\mathbf{y})$  is bounded for any $\mathbf{y}$ within that level set. These results are crucial for proving that $F_\tau(\cdot,\cdot)$ serves as a majorizer for $F(\cdot)$.

	\begin{lemma}\label{g-prop}  Given a bounded vector $\mathbf{x}^0\in\mathbb{H}^p$, the level set $S=\{\mathbf{x}\in\mathbb{H}^p: F(\mathbf{x})\leq F(\mathbf{x}^0)\}$ is nonempty and  compact. Further, it follows that
\begin{equation}\label{g-bound}
		\|g(\mathbf{x})\|\leq\frac{\alpha}{n}\sum_{i=1}^n\|\mathbf{a}_i\|^2\|\mathbf{x}\|
	\end{equation}
	and
	\begin{equation}\label{g-Lip}
		\|g(\mathbf{x})-g(\mathbf{y})\|\leq\frac{1}{n}\sum\limits_{i=1}^n (\alpha\|\mathbf{a}_i\|^2+2r^2\|\mathbf{a}_i\|^4)\|\mathbf{x}-\mathbf{y}\|
	\end{equation}
	for any $\mathbf{x},\mathbf{y}\in B_r$ where the ball
	$B_r$ is defined by $B_r=\{\mathbf{x}\in\mathbb{H}^p: \|\mathbf{x}\|\leq r\}$ with $r=\frac{4}{3}(\frac{2\alpha}{n}\sum_{i=1}^n\|\mathbf{a}_i\|^2+1)
	\sup_{\mathbf{x}\in S}\|\mathbf{x}\|$.
\end{lemma}

\begin{proof}
It is clear to see that $\mathbf{x}^0\in S$, thus $S$ is not empty.  Note that the regularized term $\|\mathbf{x}\|_{1/2}^{1/2}$ implies that $\lim_{\|\mathbf{x}\|\rightarrow\infty}~F(\mathbf{x})=\infty$. Hence, $S$ is bounded. The continuity of $F(\cdot)$ leads to the closeness of $S$. Then, the set $S$ is compact.
According to the definition of $g(\cdot)$, the first inequality of (\ref{hd-Lip}) and Cauchy's inequality, we get (\ref{g-bound}) immediately.

Since the compactness of $S$ leads to $\sup_{\mathbf{x}\in S}\|\mathbf{x}\|<\infty$, the ball $B_r$ is well defined. For any $\mathbf{x},\mathbf{y}\in B_r$, it is relatively easy for us to obtain that
\begin{equation}
	\begin{aligned}
		\|g(\mathbf{x})-g(\mathbf{y})\|
		&\leq \frac{1}{n}\sum\limits_{i=1}^n \|h'_{\alpha}(|\langle \mathbf{a}_i, \mathbf{x}\rangle|^2-b_i)\langle \mathbf{a}_i, \mathbf{x}\rangle \overline{\mathbf{a}}_i-h'_{\alpha}(|\langle \mathbf{a}_i, \mathbf{y}\rangle|^2-b_i)\langle \mathbf{a}_i, \mathbf{y}\rangle \overline{\mathbf{a}}_i\|\\
		&\leq \frac{1}{n}\sum\limits_{i=1}^n \|h'_{\alpha}(|\langle \mathbf{a}_i, \mathbf{x}\rangle|^2-b_i)\langle \mathbf{a}_i, \mathbf{x}-\mathbf{y}\rangle \overline{\mathbf{a}}_i\|\\
		&\quad+\frac{1}{n}\sum\limits_{i=1}^n \|(h'_{\alpha}(|\langle \mathbf{a}_i, \mathbf{x}\rangle|^2-b_i)-h'_{\alpha}(|\langle \mathbf{a}_i, \mathbf{y}\rangle|^2-b_i))\langle \mathbf{a}_i, \mathbf{y}\rangle \overline{\mathbf{a}}_i\|\\
		&\leq \frac{1}{n}\sum\limits_{i=1}^n |h'_{\alpha}(|\langle \mathbf{a}_i, \mathbf{x}\rangle|^2-b_i)|\|\langle \mathbf{a}_i, \mathbf{x}-\mathbf{y}\rangle \overline{\mathbf{a}}_i\|\\
		&\quad+\frac{1}{n}\sum\limits_{i=1}^n |h'_{\alpha}(|\langle \mathbf{a}_i, \mathbf{x}\rangle|^2-b_i)-h'_{\alpha}(|\langle \mathbf{a}_i, \mathbf{y}\rangle|^2-b_i)|\|\langle \mathbf{a}_i, \mathbf{y}\rangle \overline{\mathbf{a}}_i\|\\
		&\leq\frac{1}{n}\sum\limits_{i=1}^n \alpha\|\mathbf{a}_i\|^2\|\mathbf{x}-\mathbf{y}\|+\frac{1}{n}\sum\limits_{i=1}^n |\langle \mathbf{a}_i, \mathbf{x}\rangle|^2-|\langle \mathbf{a}_i, \mathbf{y}\rangle|^2|\| \mathbf{a}_i\|^2\|\mathbf{y}\|\\
		&=\frac{1}{n}\sum\limits_{i=1}^n (\alpha\|\mathbf{a}_i\|^2\|\mathbf{x}-\mathbf{y}\|+|(\mathbf{x}-\mathbf{y})^H\mathbf{a}_i\mathbf{a}_i^H
		(\mathbf{x}+ \mathbf{y})|\| \mathbf{a}_i\|^2\|\mathbf{y}\|)\\
		&\leq\frac{1}{n}\sum\limits_{i=1}^n (\alpha\|\mathbf{a}_i\|^2+2r^2\|\mathbf{a}_i\|^4)\|\mathbf{x}-\mathbf{y}\|,
	\end{aligned}\nonumber
\end{equation}
where the second inequality derives from the triangle inequality, the third follows from Cauchy's inequality, the fourth follows from (\ref{hd-Lip}), and the last is due to the fact $\mathbf{x},\mathbf{y}\in B_r$ and $\|\mathbf{x}\|+\|\mathbf{y}\|\leq 2r$.
Therefore, the proof is completed. \end{proof}

\begin{lemma}\label{Ftau-opt}
	For any $\mathbf{y}\in S$, the minimizer $\hat{\mathbf{x}}_{\mathbf{y}}$ of
	$\min_{\mathbf{x}\in\mathbb{H}^p} F_\tau(\mathbf{x},\mathbf{y})$ with $\tau\in(0,1]$ exists and satisfies
	$$\hat{\mathbf{x}}_{\mathbf{y}}\in B_r.$$
\end{lemma}
\begin{proof}
Since the regularized term $\|\mathbf{x}\|_{1/2}^{1/2}$ implies that $\lim_{\|\mathbf{x}\|\rightarrow\infty}~F(\mathbf{x})=\infty$,
we obtain the boundedness of the set $S'=\{\mathbf{x}\in\mathbb{H}^p: F_\tau(\mathbf{x},\mathbf{y})\leq F(\mathbf{x}^0)\}$, which together with
the continuity of $F(\mathbf{x})$ implies that the minimizer $\hat{\mathbf{x}}_{\mathbf{y}}$ exists and lies in $S'$ since the continuity of $F(\mathbf{x})$ also results in the closeness of $S'$.
Notice that $$F_\tau(\mathbf{x},\mathbf{y})=f(\mathbf{y})+\frac{1}{2\tau}\|\mathbf{x}-\mathbf{y}+2\tau g(\mathbf{y})\|^2+\lambda\|\mathbf{x}\|_{1/2}^{1/2}-2\tau\|g(\mathbf{y})\|^2$$
due to the fact $\|\mathbf{u}+\mathbf{v}\|^2=\|\mathbf{u}\|^2+2\mathcal{R}(\mathbf{u}^H\mathbf{v})+\|\mathbf{v}\|^2$ for any $\mathbf{u},\mathbf{v}\in\mathbb{H}^p$.
Since $\hat{\mathbf{x}}_{\mathbf{y}}$ is the minimizer of
$\min_{\mathbf{x}\in\mathbb{H}^p} F_\tau(\mathbf{x},\mathbf{y})$ with $\tau\in(0,1]$, we conclude from Lemma \ref{half-solution} that
$$\begin{aligned}
	\|\hat{\mathbf{x}}_{\mathbf{y}}\|\leq\frac{4}{3}\|\mathbf{y}-2\tau g(\mathbf{y})\|
	\leq\frac{4}{3}(\|\mathbf{y}\|+2 \|g(\mathbf{y})\|)
	\leq\frac{4}{3}(1+ \frac{2\alpha}{n}\sum_{i=1}^n\|\mathbf{a}_i\|^2)\sup_{\mathbf{x}\in S}\|\mathbf{x}\|,
\end{aligned}$$
where the second inequality derives from the triangle inequality and $\tau\in(0,1]$, the last thanks to the inequality (\ref{g-bound}) and $\mathbf{y}\in S$.
Recalling the definition of $B_r$ in Lemma \ref{g-prop}, we complete the proof.  \end{proof}

With the help of the MM techniques and Lemmas \ref{half-solution}, \ref{g-prop} and \ref{Ftau-opt}, we derive the existence and further provide a necessary condition for the optimal solution to  (\ref{min-HPR}).

\begin{theorem}\label{the1}(Fixed point inclusion)
There exists a global minimizer $\hat{\mathbf{x}}$ which lies in the level set $S$, and further satisfies the fixed point inclusion
 \begin{equation}\label{fpe-hpr}
 \mathbf{x}\in\mathcal{H}_{\lambda\tau}(\mathbf{x}-2\tau g(\mathbf{x}))
\end{equation}
for any $\tau\in(0, \min\{1, 1/L\}]$ with $L=4\sum_{i=1}^n(\alpha+2r^2\|\mathbf{a}_i\|^2)\|\mathbf{a}_i\|^2/n$ for the real-valued case and $L=2\sum_{i=1}^n(\alpha+2r^2\|\mathbf{a}_i\|^2)\|\mathbf{a}_i\|^2/n$ for the complex-valued case.
\end{theorem}

\begin{proof}
Due to the continuity of $F(\mathbf{x})$, we obtain that the global minimizer of (\ref{min-HPR}) must exist and lie in the level set $S$. Noting that the ball $B_r$ is defined by $B_r=\{\mathbf{x}\in\mathbb{H}^p: \|\mathbf{x}\|\leq r\}$ with $r=\frac{4}{3}(1+ \frac{2\alpha}{n}\sum_{i=1}^n\|\mathbf{a}_i\|^2)\|)\sup_{\mathbf{x}\in S}\|\mathbf{x}\|$, we have $S\subseteq B_r$ and $\hat{\mathbf{x}}\in B_r$.

It is claimed that
\begin{equation}\label{MM-HPR}
F(\mathbf{x})\leq F_{\tau}(\mathbf{x},\mathbf{y})\quad \mbox{for any}\,\, \mathbf{x},\mathbf{y}\in B_r
\end{equation}
for any $\tau\in(0, \min\{1, 1/L\}]$.
The equality holds clearly. We now prove the inequality for the case $\mathbb{H}=\mathbb{R}$.
From the mean value theorem, it implies that
\begin{equation}\label{tildef-mvt}
f(\mathbf{x})=f(\mathbf{y})
+\langle\nabla f(\mathbf{y}),\mathbf{x}-\mathbf{y}\rangle
+\langle\nabla f(\mathbf{y}+\theta(\mathbf{x}-\mathbf{y})),\mathbf{x}-\mathbf{y}\rangle,\nonumber
\end{equation}
where $\theta\in(0,1)$. It then follows from the definitions of $F_\tau$ and $g$, Cauchy's inequality and Lemma \ref{g-prop} that
\begin{equation}
\begin{aligned}
	F(\mathbf{x})&=F_\tau(\mathbf{x},\mathbf{y})+\langle\nabla f(\mathbf{y}+\theta(\mathbf{x}-\mathbf{y}))-\nabla f(\mathbf{y}),\mathbf{x}-\mathbf{y}\rangle
	-\frac{1}{2\tau}\|\mathbf{x}-\mathbf{y}\|^2\\
	&\leq F_\tau(\mathbf{x},\mathbf{y})+\|\nabla f(\mathbf{y}+\theta(\mathbf{x}-\mathbf{y}))-\nabla f(\mathbf{y})\|\|\mathbf{x}-\mathbf{y}\|
	-\frac{1}{2\tau}\|\mathbf{x}-\mathbf{y}\|^2\\
	&\leq F_\tau(\mathbf{x},\mathbf{y})+\frac{2}{n}\sum_{i=1}^n (\alpha+2r^2\|\mathbf{a}_i\|^2)\|\mathbf{a}_i\|^2\|\mathbf{x}-\mathbf{y}\|^2
	-\frac{1}{2\tau}\|\mathbf{x}-\mathbf{y}\|^2\\
	&= F_\tau(\mathbf{x},\mathbf{y})
	-\frac{1}{2}\|\mathbf{x}-\mathbf{y}\|^2(\frac{1}{\tau}-\frac{4}{n}\sum_{i=1}^n (\alpha+2r^2\|\mathbf{a}_i\|^2)\|\mathbf{a}_i\|^2)\\
	&\leq F_\tau(\mathbf{x},\mathbf{y}),\nonumber
\end{aligned}
\end{equation}
where the last inequality follows from the condition $\tau\in(0, \min\{1, 1/L\}]$ with $L=4\sum_{i=1}^n(\alpha+2r^2\|\mathbf{a}_i\|^2)\|\mathbf{a}_i\|^2/n$ for the real-valued case.

Next, we prove the inequality for the case  $\mathbb{H}=\mathbb{C}$.
By Lemma A.2 in \cite{sun2018geometric}, we have
\begin{equation}\label{FMM2}
\begin{aligned}
	F(\mathbf{x})
	&=f(\mathbf{x})+\lambda\|\mathbf{x}\|_{1/2}^{1/2}\\
	&=f(\mathbf{y})+\int_0^1\begin{bmatrix} \mathbf{x}-\mathbf{y} \\
		\overline{\mathbf{x}}-\overline{\mathbf{y}} \end{bmatrix}^H\nabla{}f(\mathbf{y}+t(\mathbf{x}-\mathbf{y}))\mathrm{d}t+\lambda\|\mathbf{x}\|_{1/2}^{1/2}\\
	&=f(\mathbf{y})+2\int_0^1\mathcal{R}(\langle \mathbf{x}-\mathbf{y}, g(\mathbf{y}+t(\mathbf{x}-\mathbf{y}))\rangle) \mathrm{d}t+\lambda\|\mathbf{x}\|_{1/2}^{1/2}\\
	&=f(\mathbf{y})+2\mathcal{R}(\langle g(\mathbf{y}),\mathbf{x}-\mathbf{y}\rangle)+\lambda\|\mathbf{x}\|_{1/2}^{1/2}\\
	&\qquad\qquad+2\int_0^1\mathcal{R}(\langle \mathbf{x}-\mathbf{y}, g(\mathbf{y}+t(\mathbf{x}-\mathbf{y}))-g(\mathbf{y})\rangle) \mathrm{d}t\\
	&=F_{\tau}(\mathbf{x},\mathbf{y})+2\int_0^1\mathcal{R}(\langle \mathbf{x}-\mathbf{y}, g(\mathbf{y}+t(\mathbf{x}-\mathbf{y}))-g(\mathbf{y})\rangle) \mathrm{d}t-\frac{1}{2\tau}\|\mathbf{x}-\mathbf{y}\|^2,
\end{aligned}
\end{equation}
where the second equality follows from $\nabla_{\overline{\mathbf{x}}}f(\mathbf{x})=\overline{\nabla_{\mathbf{x}}f(\mathbf{x})}$.
Combing (\ref{FMM2}) and Lemma \ref{g-prop}, we conclude that
\begin{equation}
\begin{aligned}
	F(\mathbf{x})
	&=F_{\tau}(\mathbf{x},\mathbf{y})+2\int_0^1\mathcal{R}(\langle \mathbf{x}-\mathbf{y}, g(\mathbf{y}+t(\mathbf{x}-\mathbf{y}))-g(\mathbf{y})\rangle \mathrm{d}t-\frac{1}{2\tau}\|\mathbf{x}-\mathbf{y}\|^2\\
	&\leq F_{\tau}(\mathbf{x},\mathbf{y})+2\int_0^1\|\mathbf{x}-\mathbf{y}\|\|g(\mathbf{y}+t(\mathbf{x}-\mathbf{y}))-g(\mathbf{y})\| \mathrm{d}t-\frac{1}{2\tau}\|\mathbf{x}-\mathbf{y}\|^2\\
	&\leq F_{\tau}(\mathbf{x},\mathbf{y})+\frac{1}{n}\sum\limits_{i=1}^n (\alpha+2r^2\|\mathbf{a}_i\|^2)\|\mathbf{a}_i\|^2\|\mathbf{x}-\mathbf{y}\|^2\int_0^1 2t \mathrm{d}t-\frac{1}{2\tau}\|\mathbf{x}-\mathbf{y}\|^2\\
	&= F_{\tau}(\mathbf{x},\mathbf{y})-\frac{1}{2}\|\mathbf{x}-\mathbf{y}\|^2(\frac{1}{\tau}-\frac{2}{n}\sum\limits_{i=1}^n (\alpha+2r^2\|\mathbf{a}_i\|^2)\|\mathbf{a}_i\|^2)\\
	&\leq F_{\tau}(\mathbf{x},\mathbf{y}),\nonumber
\end{aligned}
\end{equation}
where the last inequality is due to $\tau\in(0, \min\{1, 1/L\}]$ with  $L=2\sum_{i=1}^n(\alpha+2r^2\|\mathbf{a}_i\|^2)\|\mathbf{a}_i\|^2/n$ for the complex-valued case. Then we get the inequality of (\ref{MM-HPR}).

Let $\hat{\mathbf{x}}_{\hat{\mathbf{x}}}\in\arg\min\limits_{\mathbf{x}\in\mathbb{H}^p}F_\tau(\mathbf{x},\hat{\mathbf{x}})$. Lemma \ref{Ftau-opt} leads to $\hat{\mathbf{x}}_{\hat{\mathbf{x}}}\in B_r$.  According to (\ref{MM-HPR}), we can draw the following conclusion
\begin{equation}
F(\hat{\mathbf{x}})=F_{\tau}(\hat{\mathbf{x}},\hat{\mathbf{x}})\geq F_{\tau}(\hat{\mathbf{x}}_{\hat{\mathbf{x}}},\hat{\mathbf{x}})
\geq F(\hat{\mathbf{x}}_{\hat{\mathbf{x}}})\geq F(\hat{\mathbf{x}}),\nonumber
\end{equation}
which yields $\hat{\mathbf{x}}\in\arg\min\limits_{\mathbf{x}\in\mathbb{H}^p}F_{\tau}(\mathbf{x},\hat{\mathbf{x}})$. Obviously, minimizing  $F_{\tau}(\mathbf{x},\mathbf{y})$ w.r.t. $\mathbf{x}\in\mathbb{H}^p$ is equivalent to
\begin{equation}
\min\limits_{\mathbf{x}\in\mathbb{H}^p}~\|\mathbf{x}-\mathbf{y}+2\tau g(\mathbf{y})\|^2
+2\lambda\tau\|\mathbf{x}\|_{1/2}^{1/2}.\nonumber
\end{equation} 
It then follows from Lemma \ref{half-solution} that (\ref{fpe-hpr}) is satisfied. Thus, the proof is completed.
\end{proof}

\section{Convergent Algorithm}\label{convergence}
\subsection{Majorization-Minimization}\label{s5.1}

Theorem \ref{the1} has provided a globally necessary optimality condition that may not hold at the local minimizers. It is worth mentioning that $\chi_{\mu}(t)$ does not have a unique value when $|t|=\frac{\sqrt[3]{54}}{4}\mu^{2/3}$. In the implementation, we take $\chi_{\mu}(t)=0$ for such case, i.e., 
\begin{equation}
\chi_{\mu}(t)=
\begin{cases}
{\frac{2}{3}t(1+\cos(\frac{2\pi}{3}-\frac{2}{3}\mathrm{arccos}(\frac{\mu}{8}(\frac{|t|}{3})^{-3/2}))),}&{\mbox{if}~|t|>\frac{\sqrt[3]{54}}{4}\mu^{2/3}},\nonumber\\
{0,}&\textrm{otherwise.}\nonumber
\end{cases}
\end{equation}
According to (\ref{fpe-hpr}), the iterative scheme for solving (\ref{min-HPR}) can be given in Algorithm \ref{MM}.

\begin{algorithm}[ht!]
	\caption{MM for solving (\ref{min-HPR})} \label{MM}
	
	\KwIn{Data $\{\mathbf{a}_i,b_i\}$, parameters $\lambda, \alpha,~\gamma,~\beta,~\delta,~\epsilon=10^{-6}$.}
	
     \textbf{Initialize:} Spectral point $\mathbf{x}^0$, let $k=0,~j=0,~\tau_0=\beta$.

	\While{``not converged''}
	{

			(\textit{Step 1}) Do
\begin{equation}\label{iter}\nonumber
\mathbf{x}^{k+1}=\mathcal{H}_{\lambda\tau_k}(\mathbf{x}^k-2\tau_k\nabla f(\mathbf{x}^k)),
\end{equation}
where $\tau_k=\gamma\beta^{j_k}$ and $j_k$ is the smallest nonnegative integer such that
\begin{equation}\label{armijo}
F(\mathbf{x}^{k})-F(\mathbf{x}^{k+1})\geq\delta\|\mathbf{x}^{k+1}-\mathbf{x}^k\|^2.
\end{equation}

			(\textit{Step 2}) Check convergence: if 
$$\|\mathbf{x}^{k+1}-\mathbf{x}^k\|\leq\epsilon \max\{1,\|\mathbf{x}^k\|\},$$
otherwise, set $k\leftarrow k+1$, and go back to \textit{Step 1}.

	}
	
	\KwOut{$\mathbf{x}$.}
	
\end{algorithm}

Note that an important step  is the computation of $\tau_k$, which balances the speed of reduction of the objective function $F$ and the search time for the optimal length. According to Theorem \ref{the1}, the ideal choice of $\tau_k$ depends on $L$, however, it is hard to calculate since $r$ is unknown. A practical strategy is to perform an inexact line search to identify a step length that achieves adequate reduction in $L$. 
Therefore, we adopt the Armijo-type line search in Algorithm \ref{MM}. 
It requires finding the smallest nonnegative integer $j_k$ such that (\ref{armijo}) holds. 
Analog to the proof of Lemma B.3 in \cite{fanss}, we can prove that the set of such integers is nonempty and further find such integer $j_k$ successfully.

\begin{lemma}\label{armi1}
Let $\delta>0, \gamma>0, \beta\in(0,1)$ be any given parameters. For any given bounded initial vector $\mathbf{x}^0\in\mathbb{H}^p$, there is an integer $j_k\in[\underline{j}, 0]$ such that (\ref{armijo}) holds
where $$\underline{j}=\left\{
             \begin{array}{lll}
             0, & \mbox{if}~\gamma(L+\delta)\leq 1,\\
             -\mbox{[}\log_\alpha\gamma(L+\delta)\mbox{]}+1, & \mbox{otherwise}.
             \end{array}
        \right.$$
\end{lemma}
For the sake of completeness, we provide the detailed proof in the appendix.

\subsection{Convergence Analysis}

In this subsection, we will analyze the convergence of the sequence $\{\mathbf{x}^k\}$ in detail.  To begin with, we first provide the following lemma.

\begin{lemma}\label{xktaukbound}
Assume that $\{\mathbf{x}^{k}\}$ and $\{\tau_k\}$ are the sequence generated by Algorithm \ref{MM}. Then,
$\{\mathbf{x}^k\}$ is bounded and $\tau_k\in[\gamma\beta^{\bar{j}},\gamma]$.
\end{lemma}

\begin{proof}
It follows from Lemma \ref{armi1} that $\{F(\mathbf{x}^{k})\}$ is monotonically decreasing.
Then, we have $\{\mathbf{x}^k\}\subseteq D\subseteq B_r$, thereby proving the first part.
Using the definition of $\tau_k$ and Lemma \ref{armi1}, we can further prove the second part. \end{proof}

Using the above lemma and the definition of $\chi_{\mu}(\cdot)$, we can show subsequence convergence of the proposed algorithm.

\begin{theorem}\label{conv-sub}(Subsequence convergence)
Assume that $\{\mathbf{x}^k\}$ is the sequence generated by Algorithm \ref{MM}. Then the following conclusions hold:
\begin{enumerate}
	\item $\{F(\mathbf{x}^{k})\}$  decreasingly converges to $F(\mathbf{x}^*)$, where $\mathbf{x}^*$ is any accumulation point of $\{\mathbf{x}^{k}\}$;
	\item $\lim_{k\rightarrow\infty}~\|\mathbf{x}^{k+1}-\mathbf{x}^k\|=0$;
	\item Every accumulation point of $\{\mathbf{x}^{k}\}$ satisfies the fixed point inclusion
(\ref{fpe-hpr})
for a positive constant $\tau\in[\gamma\beta^{\bar{j}},\gamma]$.
\end{enumerate}
\end{theorem}

\begin{proof}
(i) Since  $\{\mathbf{x}^{k}\}$ is bounded, it has at least one accumulation point. 
Recall that $F(\cdot)\geq0$ and  $\{F(\mathbf{x}^{k})\}$ is monotonically decreasing which is due to (\ref{armijo}), it indicates that $\{F(\mathbf{x}^{k})\}$ converges to a constant $F^*(\geq0)$. 
According to the fact that $F(\cdot)$ is continuous, we have $\{F(\mathbf{x}^{k})\}\to{}F^*=F(\mathbf{x}^*)$, where $\mathbf{x}^*$ is an accumulation point of $\{\mathbf{x}^{k}\}$ as $k\to\infty.$

(ii) From the definition of $\mathbf{x}^{k+1}$ and (\ref{armijo}), we have
$$\begin{aligned}
	\sum_{k=0}^n\|\mathbf{x}^{k+1}-\mathbf{x}^{k}\|^2\leq\frac{2}{\delta}\sum_{k=0}^n[F(\mathbf{x}^{k})-F(\mathbf{x}^{k+1})]
	=\frac{2}{\delta}[F(\mathbf{x}^{0})-F(\mathbf{x}^{n+1})]
	\leq\frac{2}{\delta}F(\mathbf{x}^{0}).
\end{aligned}$$
Hence, $\sum_{k=0}^\infty\|\mathbf{x}^{k+1}-\mathbf{x}^{k}\|^2<\infty$ and $\|\mathbf{x}^{k+1}-\mathbf{x}^{k}\|\to0$ as $k\to\infty$.
Therefore, (b) is derived from the second result of Lemma \ref{xktaukbound}.

(iii) Because  $\{\mathbf{x}^{k}\}$ and $\{\tau_k\}$ have convergent sequences, without loss of generality, assume that
\begin{equation}\label{kjconv}
	\mathbf{x}^{k}\to\mathbf{x}^*~\mathrm{and}~ {\tau_k}\to\tau^*~\mbox{as}~k\to\infty.\end{equation}
It suffices to prove that
$\mathbf{x}^*$ and $\tau^*$ satisfy (\ref{fpe-hpr}). Note that combing the iteration $\mathbf{x}^{k+1}=\mathcal{H}_{\lambda\tau_k}(\mathbf{x}^k-2\tau_k\nabla f(\mathbf{x}^k))$ and Lemma \ref{half-solution} enables us to derive that for any $\mathbf{x}\in\mathbb{H}^p$,
\begin{equation}\label{iii}
	\begin{aligned}
		F_{\tau_k}(\mathbf{x}^{k+1},\mathbf{x}^k)\leq F_{\tau_k}(\mathbf{x},\mathbf{x}^k).\nonumber
	\end{aligned}
\end{equation}
This, together with the continuousness of the function $ F_{\tau}(\cdot,\cdot)$ and the limit (\ref{kjconv}), implies that for any $\mathbf{x}\in\mathbb{H}^p$,
\begin{equation}\label{iii}
	\begin{aligned}
		F_{\tau^*}(\mathbf{x}^{*},\mathbf{x}^{*})=\lim_{k\to\infty}F_{\tau_k}(\mathbf{x}^{k+1},\mathbf{x}^k)\leq \lim_{k\to\infty}F_{\tau_k}(\mathbf{x},\mathbf{x}^k)=F_{\tau^*}(\mathbf{x},\mathbf{x}^{*}).\nonumber
	\end{aligned}
\end{equation}
As a direct result of the above inequality, one can see that $\mathbf{x}^{*}\in\arg\min_{\mathbf{x}\in\mathbb{H}^p}F_{\tau^*}(\mathbf{x},\mathbf{x}^{*})$. Again using Lemma \ref{half-solution}, we complete the proof.   \end{proof}

We next discuss the whole sequence
convergence and rate of the proposed algorithm. To use the KL technique, we first study the subdifferential of the objective function for the real-valued PR. According to Lemma 2.3 in \cite{wu2023regularized}, it has
$$\partial(\|\mathbf{x}\|_{1/2}^{1/2})=\sum_{j=1}^p\mathbf{e}_{p,j} T_j(\mathbf{x})~\mbox{with}~T_j(\mathbf{x})=
\begin{cases}
	\frac{\mathbf{e}_{p,j}^\top \mathbf{x}}{2|\mathbf{e}_{p,j}^\top \mathbf{x}|^{3/2}},&{\mbox{if}\, \mathbf{e}_{p,j}^\top \mathbf{x}\neq0},\\
	{\mathbb{R},}&{\mbox{otherwise}}.
\end{cases}$$
Then, the equality in Lemma \ref{opt-diff} leads to
$$\partial F(\mathbf{x})=\nabla f(\mathbf{x})+\lambda\sum_j^p\mathbf{e}_{p,j}T_j(\mathbf{x}).$$

It is worth pointing out that the KL property is not applied to complex variables directly, then we will transform complex variables into real-valued cases. For any $\mathbf{x}\in \mathbb{C}^p$, denote $\tilde{\mathbf{x}}=( \mathcal{R}(\mathbf{x})^\top, \mathcal{I}(\mathbf{x})^\top)^\top $ and 
\begin{equation}
\mathbf{A}_i=\begin{bmatrix} \mathcal{R}(\mathbf{a}_i)\mathcal{R}(\mathbf{a}_i)^\top+\mathcal{I}(\mathbf{a}_i)\mathcal{I}(\mathbf{a}_i)^\top & \mathcal{I}(\mathbf{a}_i)\mathcal{R}(\mathbf{a}_i)^\top-\mathcal{R}(\mathbf{a}_i)\mathcal{I}(\mathbf{a}_i)^\top \\
\mathcal{R}(\mathbf{a}_i)\mathcal{I}(\mathbf{a}_i)^\top-\mathcal{I}(\mathbf{a}_i)\mathcal{R}(\mathbf{a}_i)^\top & \mathcal{R}(\mathbf{a}_i)\mathcal{R}(\mathbf{a}_i)^\top+\mathcal{I}(\mathbf{a}_i)\mathcal{I}(\mathbf{a}_i)^\top \end{bmatrix}.\nonumber
\end{equation}
For each $j=1,\cdots,p$, define
\begin{equation}
	\mathbf{W}_j=
	\begin{bmatrix} \mathbf{e}_{2p,j}^\top \\
		\mathbf{e}_{2p,p+j}^\top \end{bmatrix}.\nonumber
\end{equation}
Through simple calculation, one can get
\begin{equation}
|\langle \mathbf{a}_i,\mathbf{x}\rangle|^2=\tilde{\mathbf{x}}^\top \mathbf{A}_i\tilde{\mathbf{x}},~i=1,2,\cdots,n,\nonumber
\end{equation}
and
$$\|\mathbf{x}\|_{1/2}^{1/2}=\sum_{j=1}^p\|\mathbf{W}_j\tilde{\mathbf{x}}\|^{1/2},$$
which yields that
\begin{equation}
	f(\mathbf{x})=\tilde{f}(\tilde{\mathbf{x}})~\mbox{and}~ F(\mathbf{x})=\tilde{F}(\tilde{\mathbf{x}}).\nonumber
\end{equation}
Here, $\tilde{f}(\tilde{\mathbf{x}}):=\frac{1}{n}\sum_{i=1}^nh_{\alpha}(\tilde{\mathbf{x}}^T\mathbf{A}_i\tilde{\mathbf{x}}-b_i)$  and $\tilde{F}(\tilde{\mathbf{x}}):=\tilde{f}(\tilde{\mathbf{x}})+\lambda\sum_{j=1}^p\|\mathbf{W}_j\tilde{\mathbf{x}}\|^{1/2}.$
It is not hard to check that
\begin{equation}
\nabla\tilde{f}(\tilde{\mathbf{x}}) = 2
\begin{pmatrix} \mathcal{R}(\nabla_{\bar{\mathbf{x}}}f(\mathbf{x})) \\
\mathcal{I}(\nabla_{\bar{\mathbf{x}}}f(\mathbf{x})) \end{pmatrix}.\nonumber
\end{equation}
Moreover, it follows that
$$\begin{aligned}
F_\tau(\mathbf{x},\mathbf{y})\nonumber
&=f(\mathbf{y})+\frac{1}{2\tau}\|\mathbf{x}-\mathbf{y}+2\tau g(\mathbf{y})\|^2+\lambda\|\mathbf{x}\|_{1/2}^{1/2}-2\tau\|g(\mathbf{y})\|^2\nonumber\\
&=\tilde{f}(\tilde{\mathbf{y}})+\frac{1}{2\tau}\|\tilde{\mathbf{x}}-\tilde{\mathbf{y}}+\tau \nabla \tilde{f}(\tilde{\mathbf{y}})\|^2 +\lambda\|\mathbf{x}\|_{1/2}^{1/2}
-\frac{\tau}{2}\|\nabla\tilde{f}(\tilde{\mathbf{y}})\|^2\nonumber\\
&=\tilde{f}(\tilde{\mathbf{y}})+\langle\nabla \tilde{f}(\tilde{\mathbf{y}}),\tilde{\mathbf{x}}-\tilde{\mathbf{y}}\rangle
+\frac{1}{2\tau}\|\tilde{\mathbf{x}}-\tilde{\mathbf{y}}\|^2+\lambda\sum_{j=1}^p\|\mathbf{W}_j\tilde{\mathbf{x}}\|_{1/2}^{1/2},\nonumber
\end{aligned}
$$
where $\tilde{\mathbf{y}}=( \mathcal{R}(\mathbf{y})^\top, \mathcal{I}(\mathbf{y})^\top)^\top$.

We turn to consider the differentiability of $\sum_{j=1}^p\|\mathbf{W}_j\tilde{\mathbf{x}}\|^{1/2}$ based on the $2p$-dimension variable $\tilde{\mathbf{x}}$. One can easily get
$$\partial(\sum_{j=1}^p\|\mathbf{W}_j\tilde{\mathbf{x}}\|^{1/2})=\sum_{j=1}^p\mathbf{W}_j^\top\tilde{T}_j(\tilde{\mathbf{x}})~\mbox{with}~ \tilde{T}_j(\tilde{\mathbf{x}})=\begin{cases}
\frac{\mathbf{W}_j\tilde{\mathbf{x}}}{2\|\mathbf{W}_j\tilde{\mathbf{x}}\|^{3/2}},&{ \mbox{if}\,\|\mathbf{W}_j\tilde{\mathbf{x}}\|\neq0},\\
{\mathbb{R}^{2},}&{\mbox{otherwise}},
\end{cases}$$
and then
$$\partial\tilde{F}(\tilde{\mathbf{x}})=\nabla \tilde{f}(\tilde{\mathbf{x}})+\lambda\sum_{j=1}^p\mathbf{W}_j^\top\tilde{T}_j(\tilde{\mathbf{x}}).$$

Denote $\psi(\mathbf{u})= \psi_1(\mathbf{u})+\lambda\psi_2(\mathbf{u})$,
where $\psi_1(\mathbf{u})=\sum_{i=1}^nh_\alpha(\mathbf{u}^\top\mathbf{B}_i\mathbf{u}-b_i)$ and $\psi_2(\mathbf{u})=\sum_{j=1}^p\|\mathbf{D}_j\mathbf{u}\|^{1/2}$.
For $\tau>0$, let $$\psi_{\tau}(\mathbf{u}, \mathbf{v})=\psi_1(\mathbf{v})+\langle\nabla \psi_1(\mathbf{v}),\mathbf{u}-\mathbf{v}\rangle
+\frac{1}{2\tau}\|\mathbf{u}-\mathbf{v}\|^2+\lambda\psi_2(\mathbf{u}).$$
Based on these notations, we introduce the following two minimization problems
\begin{equation}\label{uni-frame} \nonumber
\min_{\mathbf{u}\in\mathbb{R}^d} \psi(\mathbf{u})
~\mbox{and}~\min_{\mathbf{u}\in\mathbb{R}^d} \psi_\tau(\mathbf{u}, \mathbf{v}).
\end{equation}
It is not hard to check out that the functions $\psi$ and $\psi_\tau$ coincide with $F$ and $F_\tau$
by taking $\mathbf{u}=\mathbf{x}, \mathbf{B}_i=\mathbf{a}_i\mathbf{a}_i^\top$,  $\mathbf{D}_j=\mathbf{e}_{p,j}^\top$ and $d=p$ for the real-valued PR case, and by $\mathbf{u}=\tilde{\mathbf{x}}, \mathbf{B}_i=\mathbf{A}_i$,  $\mathbf{D}_j=\mathbf{W}_j$ and $d=2p$ as illustrated above for the complex-valued PR case, respectively. We also  see that the above two minimization problems are equivalent to (\ref{min-HPR})
and $\min_{\mathbf{x}\in\mathbb{H}^p}F_\tau(\mathbf{x},\mathbf{y})$, respectively.

We also need to discuss whether the objective function $\psi(\mathbf{u})$ satisfies the KL property. The following result combined  with Lemma \ref{semialg-prop} provides an affirmative answer.

\begin{lemma}\label{huber-semialg}
Both $h_\alpha(u)$ for $u\in\mathbb{R}$ and $\|\mathbf{D}_j\mathbf{u}\|^{1/2}$ for $\mathbf{u}\in\mathbb{R}^d$ are semialgebraic.
\end{lemma}

\begin{proof}
By the definition of semialgebraic functions, it suffices to prove that the graph of $h_\alpha(u)$ denoted by $\mathrm{graph}_{h_\alpha}$ is
a semialgebraic set. In fact,
$$\begin{aligned}
	\mathrm{graph}_{h_\alpha}=&\{(u, t)\in\mathbb{R}^{2}: h_\alpha(u)=t\}\\
	&=\{(u, t)\in\mathbb{R}^{2}: t-\frac{1}{2}u^2=0, u^2-\alpha^2<0\}\\
	&\cup\{(u, t)\in\mathbb{R}^{2}: t-\alpha u+\frac{1}{2}\alpha^2=0, -u+\alpha<0\}\\
	&\cup\{(u, t)\in\mathbb{R}^{2}: t+\alpha u+\frac{1}{2}\alpha^2=0, u+\alpha<0\}\\
	&\cup\{(u, t)\in\mathbb{R}^{2}: (t-\frac{1}{2}\alpha^2)^2+(u-\alpha)^2=0, t+u-\frac{1}{2}\alpha^2-\alpha-1<0\}\\
	&\cup\{(u, t)\in\mathbb{R}^{2}: (t-\frac{1}{2}\alpha^2)^2+(u+\alpha)^2=0, t+u-\frac{1}{2}\alpha^2+\alpha-1<0\}.
\end{aligned}$$
It indicates that $h_\alpha(u)$ is semialgebraic. Moreover, the following statements hold,
$$\begin{aligned}
		\{(\mathbf{u}, t)\in\mathbb{R}^{p+1}:  \|\mathbf{D}_j\mathbf{u}\|^{1/2}=t\}
		&=\{(\mathbf{u}, t)\in\mathbb{R}^{p+1}:  \|\mathbf{D}_j\mathbf{u}\|^2-t^4=0, -t<0\}\\
		&\cup\{(\mathbf{u}, t)\in\mathbb{R}^{p+1}: \|\mathbf{D}_j\mathbf{u}\|^2+t^4=0, \|\mathbf{D}_j\mathbf{u}\|^2+t^4-1<0\},
	\end{aligned}$$
	thus we obtain that $\|\mathbf{D}_j\mathbf{u}\|^{1/2}$ is also semialgebraic.  \end{proof} 

Now, we are ready to discuss the whole convergence with the help of the KL technique. 

\begin{theorem}\label{who-conv-u}(Whole sequence convergence)
Assume that $\{\mathbf{x}^k\}$ is the sequence generated by Algorithm \ref{MM}. Then the whole sequence $\{\mathbf{x}^k\}$ converges to a vector satisfying (\ref{fpe-hpr}) for a positive number $\tau$.
\end{theorem}

\begin{proof}
Recall that the iteration formula (\ref{iter}) can be regarded as a solution of $\min_{\mathbf{x}\in\mathbb{H}^p}F_{\tau_k}(\mathbf{x}, \mathbf{x}^k)$, i.e.,
$$\mathbf{x}^{k+1}\in\arg\min_{\mathbf{x}\in\mathbb{H}^p}F_{\tau_k}(\mathbf{x}, \mathbf{x}^k).$$
Denote $\mathbf{u}^k=\mathbf{x}^k$ for $\mathbf{x}^k\in\mathbb{R}^p$, and $\mathbf{u}^k=(\mathcal{R}(\mathbf{x}^k)^\top, \mathcal{I}(\mathbf{x}^k)^\top)^\top$ for $\mathbf{x}^k\in\mathbb{C}^p$. We can rewrite the generated sequence $\{\mathbf{x}^k\}$ via the new notation as $\{\mathbf{u}^k\}$ which satisfies
\begin{equation}\label{iter-u}\nonumber
	\mathbf{u}^{k+1}\in\arg\min_{\mathbf{u}\in\mathbb{R}^d}\psi_{\tau_k}(\mathbf{u}, \mathbf{u}^k),\end{equation}
and
\begin{equation}\label{armijo-u}
	\psi(\mathbf{u}^{k})-\psi(\mathbf{u}^{k+1})\geq\delta\|\mathbf{u}^{k+1}-\mathbf{u}^k\|^2.
\end{equation}

For each $k$, we conclude from Lemma \ref{opt-diff}  that
$$\mathbf{0}\in2(\mathbf{u}^{k+1}-\mathbf{u}^k+\tau_k\nabla \psi_1(\mathbf{u}^k))+2\lambda\tau_k\partial\psi_2(\mathbf{u}^{k+1}).$$
Combing this and (\ref{diff-sumrule}), we obtain that for each $k$,
\begin{equation}\label{diff-sum}
	\tau_k^{-1}(\mathbf{u}^{k}-\mathbf{u}^{k+1})+\nabla \psi_1(\mathbf{u}^{k+1})-\nabla \psi_1(\mathbf{u}^k)\in \nabla \psi_1(\mathbf{u}^{k+1})+\lambda\partial\psi_2(\mathbf{u}^{k+1})=\partial \psi(\mathbf{u}^{k+1}).
\end{equation}
Noting that $\mathbf{u}^{k},\mathbf{u}^{k+1}\in \{\mathbf{u}\in\mathbb{R}^d: \|\mathbf{u}\|\leq r\}$, we conclude from (\ref{g-Lip}) that
$$\|\nabla \psi_1(\mathbf{u}^{k+1})-\nabla \psi_1(\mathbf{u}^k)\|\leq L\|\mathbf{u}^{k}-\mathbf{u}^{k+1}\|,$$
which together with (\ref{diff-sum}) and Lemma \ref{xktaukbound} results in that
\begin{equation}\label{dist-upb}\nonumber
	\begin{aligned}
		\mathrm{dist}(\mathbf{0},\partial \psi(\mathbf{u}^{k+1}))&\leq \|\tau_k^{-1}(\mathbf{u}^{k}-\mathbf{u}^{k+1})+\nabla \psi_1(\mathbf{u}^{k+1})-\nabla \psi_1(\mathbf{u}^k)\|\\
		&\leq \tau_k^{-1}\|\mathbf{u}^{k}-\mathbf{u}^{k+1}\|+L\|\mathbf{u}^{k+1}-\mathbf{u}^k\|\\
		&\leq (\gamma^{-1}\beta^{-\bar{j}}+L)\|\mathbf{u}^{k}-\mathbf{u}^{k+1}\|.
\end{aligned}\end{equation}
From Theorem \ref{conv-sub}, we know that there exists a subsequence $\{\mathbf{u}^{k_l}\}$ and a cluster point $\mathbf{u}^*$ such that
\begin{equation}\label{u-conv} \mathbf{u}^{k_l}\to \mathbf{u}^*~\mbox{as}~j\to\infty~
	\mbox{and}~\psi(\mathbf{u}^{k})\to\psi(\mathbf{u}^*)~\mbox{as}~k\to\infty.
\end{equation}
This leads to $\psi(\mathbf{u}^k)\geq \psi(\mathbf{u}^*)$ for all $k$.

For the case that there exists an index $\bar{k}$ such that $\psi(\mathbf{u}^{\bar{k}})= \psi(\tilde{\mathbf{u}})$. Then, the monotonicity  implies that
$\psi(\mathbf{u}^{\bar{k}+1})= \psi(\tilde{\mathbf{u}})$. As a direct result the equality and (\ref{armijo-u}), we have
$$0\leq\delta\|\mathbf{u}^{\bar{k}}-\mathbf{u}^{\bar{k}+1}\|^2\leq \psi(\mathbf{u}^{\bar{k}})-\psi(\mathbf{u}^{\bar{k}+1})=0,$$
and then $\mathbf{u}^{k}=\mathbf{u}^{\bar{k}}$ for any $k>\bar{k}$. This means that the whole sequence $\{\mathbf{u}^k\}$ converges to the point $\mathbf{u}^*$ with $\mathbf{u}^*=\mathbf{u}^{\bar{k}}$.

We next prove the desired result for the case that $\psi(\mathbf{u}^k)> \psi(\mathbf{u}^*)$ for any $k$. The second limit in (\ref{u-conv}) implies that given $\eta>0$ there exists an index $\hat{k}_1$ such that $\psi(\mathbf{u}^k)<\psi(\mathbf{u}^*)$ for any $k\geq \hat{k}_1$. Therefore, it follows that
\begin{equation}\label{psi-eta}\psi(\mathbf{u}^*)<\psi(\mathbf{u}^k)<\psi(\mathbf{u}^*)+\eta\end{equation}
for all $k\geq \hat{k}_1$. Denote $\Delta_k=\phi(\psi(\mathbf{u}^k)-\psi(\mathbf{u}^*))$.  From Lemmas \ref{semialg-prop} and \ref{huber-semialg}, we conclude  that
$\psi$ is a KL function with $\phi(t)=ct^{1-\varrho}$ for some $c>0$ and rational number $\varrho\in[0,1).$ Combing this, $\psi(\mathbf{u}^k)>\psi(\mathbf{u}^*)$ and  (\ref{u-conv}), we obtain that $\Delta_k\geq0$ and $\Delta_k$ is monotonically nondecreasing for all $k\geq \hat{k}_1$. Further, it follows that for any $k\geq \hat{k}_1$,
\begin{equation}
	\begin{aligned}
		\Delta_k-\Delta_{k+1}&\geq\phi'(\psi(\mathbf{u}^k)<\psi(\mathbf{u}^*))(\psi(\mathbf{u}^k)-\psi(\mathbf{u}^{k+1}))\\
		&\geq\frac{1}{\mathrm{dist}(\mathbf{0},\psi(\mathbf{u}^k)}(\psi(\mathbf{u}^k)-\psi(\mathbf{u}^{k+1}))\\
		&\geq\frac{\delta\|\mathbf{u}^k-\mathbf{u}^{k+1}\|^2}{\mathrm{dist}(\mathbf{0},\psi(\mathbf{u}^k)}\\
		&\geq\frac{\delta\|\mathbf{u}^k-\mathbf{u}^{k+1}\|^2}{(\gamma^{-1}\beta^{-\bar{j}}+L)\|\mathbf{u}^k-\mathbf{u}^{k-1}\|},
	\end{aligned}\nonumber\end{equation}
which together with the elementary inequality $2\sqrt{st}\leq s+t\geq$ for any $s,t\geq0$ yields that
\begin{equation}\label{psi-diff}
	\begin{aligned}
		2\|\mathbf{u}^k-\mathbf{u}^{k+1}\|&\leq
		2\sqrt{\frac{\gamma^{-1}\beta^{-\bar{j}}+L}{\delta}(\Delta_k-\Delta_{k+1})\|\mathbf{u}^k-\mathbf{u}^{k-1}\|}\\
		&\leq\frac{(\gamma\beta^{\bar{j}})^{-1}+L}{\delta}(\Delta_k-\Delta_{k+1})+\|\mathbf{u}^k-\mathbf{u}^{k-1}\|.
\end{aligned}\end{equation}

We now prove the sequence $\{\mathbf{u}^k\}$ is a Cauchy sequence. Given any $\epsilon>0$, we also conclude from (\ref{u-conv}) that there exists an index $\hat{k}_2$ such that $\psi(\mathbf{u}^k)-\psi(\mathbf{u}^*)<\frac{\varsigma\epsilon_1}{\delta}$ where $\varsigma$ is a positive number satisfying
$$\varsigma<\min\{\frac{\delta}{4}, (\delta(2c(\gamma\beta^{\bar{j}})^{-1}+L)^{-1})^{\frac{1}{1-\varrho}}, \frac{1}{2}\}, \quad \epsilon_1=\min\{\epsilon^2, \epsilon^{\frac{1}{1-\varrho}}\}.$$
Denote $\Delta_k=\phi(\psi(\mathbf{u}^k)-\psi(\mathbf{u}^*))$ and $\hat{k}_3=\max\{\hat{k}_1, \hat{k}_2\}$. Since $\phi(t)=ct^{1-\varrho}$ for some $\varrho\in(0,1)$, $\psi(\mathbf{u}^k)>\psi(\mathbf{u}^*)$ and  (\ref{u-conv}), we obtain that $\Delta_k$ is monotonically nonincreasing and $\Delta_k\in(0, c(\frac{\varsigma\epsilon_1}{\delta})^{1-\varrho})$ for all $k\geq \hat{k}_3$.
Take two arbitrary indexes $K_1$ and $K_2$ with $K_2>K_1>\hat{k}_3$. According to (\ref{psi-diff}), we have
\begin{equation}
	\begin{aligned}
		2\sum_{k=K_1}^{K_2-1}\|\mathbf{u}^k-\mathbf{u}^{k+1}\|& \leq\frac{(\gamma\beta^{\bar{j}})^{-1}+L}{\delta}\sum_{k=K_1}^{K_2-1}(\Delta_k-\Delta_{k+1})+\sum_{k=K_1}^{K_2-1}\|\mathbf{u}^k-\mathbf{u}^{k-1}\|\\
&=\frac{(\gamma\beta^{\bar{j}})^{-1}+L}{\delta}(\Delta_{K_1}-\Delta_{K_2})+\sum_{k=K_1}^{K_2-1}\|\mathbf{u}^k-\mathbf{u}^{k+1}\|
		+\|\mathbf{u}^{K_1}-\mathbf{u}^{K_1-1}\|\\
		&\leq \frac{(\gamma\beta^{\bar{j}})^{-1}+L}{\delta}\Delta_{K_1}+\sum_{k=K_1}^{K_2-1}\|\mathbf{u}^k-\mathbf{u}^{k+1}\|
		+\|\mathbf{u}^{K_1}-\mathbf{u}^{K_1-1}\|,
	\end{aligned}\nonumber\end{equation}
where the last inequality derives from $\Delta_{K_2}\geq0$.
Subtracting the second summand from the right-hand side, we get
\begin{equation}
	\begin{aligned}
		\sum_{k=K_1}^{K_2-1}\|\mathbf{u}^k-\mathbf{u}^{k+1}\|
		&\leq\frac{(\gamma\beta^{\bar{j}})^{-1}+L}{\delta}\Delta_{K_1}
		+\|\mathbf{u}^{K_1}-\mathbf{u}^{K_1-1}\|\\
		&\leq\frac{(\gamma\beta^{\bar{j}})^{-1}+L}{\delta}\Delta_{\hat{k}_3}
		+\|\mathbf{u}^{K_1}-\mathbf{u}^{K_1-1}\|,
	\end{aligned}\nonumber\end{equation}
where the last inequality derives from that $K_1>\hat{k}_3\geq\hat{k}_1$ and $\Delta_{k}$ is monotonically nonincreasing as $k\geq\hat{k}_1$.
Combing this, (\ref{armijo-u}) and the nonincreasing monotonicity of $\psi(\mathbf{u}^k)$, we get
\begin{equation}
	\begin{aligned}
		\|\mathbf{u}^{K_2}-\mathbf{u}^{K_1}\|&\leq\sum_{k=K_1}^{K_2-1}\|\mathbf{u}^k-\mathbf{u}^{k+1}\|\\
		&\leq
		\frac{(\gamma\beta^{\bar{j}})^{-1}+L}{\delta}\phi(\psi(\mathbf{u}^{\hat{k}_3})-\psi(\mathbf{u}^*))
		+\sqrt{\frac{\psi(\mathbf{u}^{K_1-1})-\psi(\mathbf{u}^{K_1})}{\delta}}\\
		&\leq
		\frac{(\gamma\beta^{\bar{j}})^{-1}+L}{\delta}\phi(\psi(\mathbf{u}^{\hat{k}_3})-\psi(\mathbf{u}^*))
		+\sqrt{\frac{\psi(\mathbf{u}^{\hat{k}_3})-\psi(\mathbf{u}^*)}{\delta}}\\
		&\leq\frac{(\gamma\beta^{\bar{j}})^{-1}+L}{\delta}c(\frac{\varsigma\epsilon_1}{\delta})^{1-\varrho}+\frac{\varsigma\epsilon_1}{\delta}\\
		&<\epsilon,
	\end{aligned}\nonumber\end{equation}
where the third inequality derives from $\psi(\mathbf{u}^{K_1-1})\leq\psi(\mathbf{u}^{\hat{k}_3})$ and $\psi(\mathbf{u}^{K_1})>\psi(\mathbf{u}^*)$, the fourth follows from upper bounds of $\psi(\mathbf{u}^k)-\psi(\mathbf{u}^*)$ and $\Delta_k$ as analyzed above, and the last inequality is due to the definitions of $\varsigma$ and $\epsilon_1$. Thus, the sequence $\{\mathbf{u}^k\}$ is Cauchy sequence. Therefore, we prove that  $\{\mathbf{u}^k\}$ converges to $\mathbf{u}^*$. 
\end{proof}

\subsection{Convergence Rate}\label{s5.3}
When it comes to estimating the convergence rate, it is well known that the KL exponent plays an important role. As stated in \cite{attouch2009convergence}, the KL exponent $\varrho\in(0,1/2]$ corresponds to a linear convergence rate, and $\varrho\in(1/2,1)$   corresponds to a sublinear convergence rate. Nevertheless, it is not an easy task to estimate the KL exponent, especially to verify whether it is strictly smaller than 1/2. Under mild conditions, we check that $\psi$ may have KL property of exponent 1/2 and then prove the proposed algorithm converges with a linear rate.

\begin{theorem}\label{conv-who} (Convergence rate)
Assume that $\{\mathbf{x}^k\}$ is the sequence generated by Algorithm \ref{MM}. Then the whole sequence $\{\mathbf{x}^k\}$ is convergent and converges at least sublinearly to a vector $\mathbf{x}^*$ satisfying (\ref{fpe-hpr}) for a positive number $\tau$. Further, if there exists a positive constant $\epsilon_1<\alpha$ such that
{\small\begin{equation}\label{gene-jaciban-pos-real}
		\begin{aligned}
		\lambda_{\mathrm{min}}(\sum_{|\langle\mathbf{x}^*,\mathbf{a}_i\mathbf{a}_i^{\top}\mathbf{x}^*\rangle-b_i|\leq\alpha-\epsilon_1}\mathbf{H}_i)\geq	3\sum_{\big||\langle\mathbf{x}^*,\mathbf{a}_i\mathbf{a}_i^{\top}\mathbf{x}^*\rangle-b_i|-\alpha\big|<\epsilon_1}\|\mathbf{H}_i\|_2+
			\frac{3\lambda}{4}\max_{j^\in\Gamma^*}|x_j^*|^{-3/2}
\end{aligned}\end{equation}}
for the real-valued PR, and
{\small\begin{equation}\label{gene-jaciban-pos-com}
\begin{aligned}
\lambda_{\mathrm{min}}(\sum_{|\langle\tilde{\mathbf{x}}^*,\mathbf{A}_i\tilde{\mathbf{x}}^*\rangle-b_i|\leq\alpha-\epsilon_1}
\widetilde{\mathbf{H}}_i)\geq 3\sum_{\big||\langle\tilde{\mathbf{x}}^*,\mathbf{A}_i\tilde{\mathbf{x}}^*\rangle-b_i|-\alpha\big|<\epsilon_1}
\|\widetilde{\mathbf{H}}_i\|_2 +\frac{3\lambda}{2}\max_{\|\mathbf{D}_j\tilde{x}^*\|\neq0}\|\mathbf{D}_j\tilde{x}^*\|^{-3/2}
\end{aligned}
\end{equation}}
for the complex-valued PR, then the whole sequence $\{\mathbf{x}^k\}$ converges to the point $\mathbf{x}^*$ with a linear rate.
Here, $\mathbf{H}_i=\frac{1}{n}(3\langle\mathbf{a}_i,\mathbf{x}^*\rangle^2-b_i)
\mathbf{a}_{i\Gamma^*}\mathbf{a}_{i\Gamma^*}^\top$ and $\widetilde{\mathbf{H}}_i=\frac{1}{n}(2\mathbf{A}_i^{\tilde{\Gamma}^*\tilde{\Gamma}^*}\tilde{x}_{\tilde{\Gamma}^*}^{*}(\tilde{x}_{\tilde{\Gamma}^*}^{*})^\top
\mathbf{A}_i^{\tilde{\Gamma}^*\tilde{\Gamma}^*}+\frac{1}{n}(\langle\tilde{x}^*,\mathbf{A}_i\tilde{\mathbf{x}}^*\rangle-b_i)
\mathbf{A}_i^{\tilde{\Gamma}^*\tilde{\Gamma}^*})$ with $\Gamma^*=\{j=1,\cdots,p: \mathbf{e}_{p,j}^\top \mathbf{x}^*\neq0\}$,  $\tilde{\mathbf{x}}^*=( \mathcal{R}(\mathbf{x}^*)^\top$, $\mathcal{I}(\mathbf{x}^*)^\top)^\top$ and $\tilde{\Gamma}^*=\Gamma^*\cup\{p+j: j\in\Gamma^*\}$.
\end{theorem}

\begin{proof}
According to the proof of Theorem \ref{who-conv-u} and notation, both the sequences of $\{\mathbf{x}^k\}$ and $\{\mathbf{u}^k\}$ have the same convergence property. So, we only need to prove the convergence of the sequence $\{\mathbf{u}^k\}$. Recall that $\psi$ is a KL function with $\phi(t)=ct^{1-\varrho}$ for some $c>0$ and rational number $\varrho\in[0,1).$ The sublinear convergence rate of the sequence $\{\mathbf{u}^k\}$ can be derived by applying the generic rate of convergence result in \cite{attouch2009convergence}.

We proceed to prove the liner convergence rate under (\ref{gene-jaciban-pos-real}) or (\ref{gene-jaciban-pos-com}).
Based on the definition of the operator $\mathcal{H}$ and Theorem \ref{who-conv-u}, the sequence $\{\mathbf{x}^k\}$ converging to $\mathbf{x}^*$ implies that there exists an index $\hat{j}_4$ such that $\{j=1,\cdots,p: \mathbf{e}_{p,j}^\top\mathbf{x}^k\neq0\}=\Gamma^*$ for all $k\geq\hat{j}_4$.
So, it suffices to prove that \begin{equation}\label{x_subvector}\mathbf{x}_{\Gamma^*}^k\to\mathbf{x}_{\Gamma^*}^*\end{equation} with a linear rate.

Let $\mathbf{u}^*=\mathbf{x}^*$ for $\mathbf{x}^*\in\mathbb{R}^p$ and $\mathbf{u}^*=(\mathcal{R}(\mathbf{x}^*)^\top, \mathcal{I}(\mathbf{x}^*)^\top)^\top$ for $\mathbf{x}^*\in\mathbb{C}^p$. We also make use of the unified notations $\Upsilon^*$ and $\mathbf{B}_i$ to respectively represent $\Gamma^*$ and $\mathbf{a}_i\mathbf{a}_i^\top$ for the real-valued case, and $\tilde{\Gamma}^*=\Gamma^*\cup\{p+j: j\in\Gamma^*\}$ and $\mathbf{A}_i$ for the complex-valued case, respectively.
For any $\mathbf{w}\in\mathbb{R}^{|\Upsilon^*|}$, define $\psi_{\Upsilon^*}(\mathbf{w})=\psi_{1\Upsilon^*}(\mathbf{w})+\lambda\psi_{2\Upsilon^*}(\mathbf{w})$ with
$$\psi_{1\Upsilon^*}(\mathbf{w})=\frac{1}{n}\sum_{i=1}^nh_\alpha(\mathbf{w}^\top\mathbf{B}_i^{\Upsilon^*\Upsilon^*}\mathbf{w}-b_i)~\mbox{and}~
\psi_{2\Upsilon^*}(\mathbf{w})=\sum_{j\in\Gamma^*}\|\mathbf{D}_{j\Upsilon^*}\mathbf{w}\|^{1/2},$$
where $|\Upsilon^*|$ represents the cardinality of the set $\Upsilon^*$.  
For any $\mathbf{u}\in\mathbb{R}^d$ with $\{j=1,\cdots,p: \mathbf{e}_{d,j}^\top\mathbf{u}\neq0\}\subseteq\Upsilon^*$, it is clear to see that
$$\psi_1(\mathbf{u})=\psi_{1\Upsilon^*}(\mathbf{u}_{\Upsilon^*}),~\psi_2(\mathbf{u})=\psi_{2\Upsilon^*}(\mathbf{u}_{\Upsilon^*})
~\mbox{and}~
\psi(\mathbf{u})=\psi_{\Upsilon^*}(\mathbf{u}_{\Upsilon^*}).$$ Then, $\psi_{1\Upsilon^*}$ is continuous differentiable with the following gradient
$$\nabla\psi_{1\Upsilon^*}(\mathbf{w})=\frac{1}{n}\sum_{i=1}^nh_\alpha'(\mathbf{w}^\top\mathbf{B}_i^{\Upsilon^*\Upsilon^*}\mathbf{w}-b_i)\mathbf{B}_i^{\Upsilon^*\Upsilon^*}
\mathbf{w},$$
and $\psi_{2\Upsilon^*}$ is twice continuous differentiable with the following gradient
$$\nabla\psi_{2\Upsilon^*}(\mathbf{w})=\sum_{j\in\Gamma^*}\mathbf{D}_{j\Upsilon^*}^\top \frac{\mathbf{D}_{j\Upsilon^*}\mathbf{w}}{2\|\mathbf{D}_{j\Upsilon^*}\mathbf{w}\|^{3/2}}.$$
Beyond these, we have
$$(\nabla\psi_1(\mathbf{u}))_{\Upsilon^*}=\nabla\psi_{1\Upsilon^*}(\mathbf{u}_{\Upsilon^*})~\mbox{and}~ (\partial\psi_2(\mathbf{u}))_{\Upsilon^*}=\nabla\psi_{2\Upsilon^*}(\mathbf{u}_{\Upsilon^*}),$$
and then
$$(\partial\psi(\mathbf{u}))_{\Upsilon^*}=\nabla\psi_{1\Upsilon^*}(\mathbf{u}_{\Upsilon^*})+\lambda\nabla\psi_{2\Upsilon^*}(\mathbf{u}_{\Upsilon^*}).$$
Recall Theorem 2 in \cite{attouch2009convergence}, then the statement (\ref{x_subvector}) holds if $\psi_{\Upsilon^*}$ satisfies the KL property at $\mathbf{u}_{\Upsilon^*}^*$ with an exponent of $1/2$, i.e., the desingularizing function $\phi(\cdot)$ can be chosen as $\phi(t)=ct^{1/2}$ for some $c>0$.

We now prove the function $\psi_{\Upsilon^*}$ has the KL property at $\mathbf{u}_{\Upsilon^*}^*$ with an exponent of $1/2$. Denote  $\mathbf{w}^*=\mathbf{u}_{\Upsilon^*}^*$ and $\varepsilon_1=\frac{1}{2}\min\{1, \min_{j\in\Gamma^*}|\mathbf{e}_{p,j}^\top\mathbf{x}^*|\}$.
Since $\mathbf{x}^*$ satisfies (\ref{fpe-hpr}), we have
$$\mathbf{u}^*\in\arg\min_{\mathbf{u}\in\mathbb{R}^d}\psi_\tau(\mathbf{u}, \mathbf{u}^*)$$ for some $\tau>0$.
Combing this and Lemma \ref{opt-diff}, we get
$$\mathbf{0}\in\nabla\psi_1(\mathbf{u}^*)+\lambda\partial\psi_2(\mathbf{u}^*),$$
which leads to
\begin{equation}\label{psi-res-0}\nabla\psi_{\Upsilon^*}(\mathbf{w}^*)=\nabla\psi_{1\Upsilon^*}(\mathbf{w}^*)+\lambda\nabla\psi_{2\Upsilon^*}(\mathbf{w}^*)=
	\nabla\psi_{1\Upsilon^*}(\mathbf{u}_{\Upsilon^*})+\lambda\nabla\psi_{2\Upsilon^*}(\mathbf{u}_{\Upsilon^*})=\mathbf{0}.\end{equation}

From the similar line of the proof in Lemma \ref{g-prop}, we conclude that
\begin{equation}\label{Lpsi1}\|\nabla\psi_{1\Upsilon^*}(\mathbf{w})-\nabla\psi_{1\Upsilon^*}(\mathbf{w}^*)\|\leq
	\frac{1}{n}\sum_{i=1}^n(\alpha+(1+2\|\mathbf{B}_i\|_2\varepsilon_1^2))\|\mathbf{B}_i\|_2\|\mathbf{w}^*\|\|\mathbf{w}-\mathbf{w}^*\|,\end{equation}
for any $\mathbf{w}\in\{\mathbf{w}: \|\mathbf{w}-\mathbf{w}^*\|<\varepsilon_1\}.$
We now turn to the second-order differentiability of $\psi_{\Upsilon^*}$. By simple calculation, one can get   the Hessian matrix of $\psi_{2\Upsilon^*}$ at $\mathbf{w}$ as follows
$$\nabla^2\psi_{2\Upsilon^*}(\mathbf{w})=
\sum_{j\in\Gamma^*}(\frac{\mathbf{D}_{j\Upsilon^*}^\top\mathbf{D}_{j\Upsilon^*}}{2\|\mathbf{D}_{j\Upsilon^*}\mathbf{w}\|^{3/2}}-
\frac{3\mathbf{D}_{j\Upsilon^*}^\top\mathbf{D}_{j\Upsilon^*}\mathbf{w}_0\mathbf{w}^\top\mathbf{D}_{j\Upsilon^*}^\top
	\mathbf{D}_{j\Upsilon^*}}{4\|\mathbf{D}_{j\Upsilon^*}\mathbf{w}\|^{7/2}}),$$
which implies that
$$\begin{aligned}
	&\|\nabla^2\psi_{2\Upsilon^*}(\mathbf{w})\|_2\\
	&=\|\sum_{j\in\Gamma^*}(\frac{\mathbf{D}_{j\Upsilon^*}^\top\mathbf{D}_{j\Upsilon^*}}{2\|\mathbf{D}_{j\Upsilon^*}\mathbf{w}\|^{3/2}}-
	\frac{3\mathbf{D}_{j\Upsilon^*}^\top\mathbf{D}_{j\Upsilon^*}\mathbf{w}\mathbf{w}^\top\mathbf{D}_{j\Upsilon^*}^\top
		\mathbf{D}_{j\Upsilon^*}}{4\|\mathbf{D}_{j\Upsilon^*}\mathbf{w}\|^{7/2}})\|_2\\
	&=\|\sum_{j\in\Gamma^*}\frac{1}{2\|\mathbf{D}_{j\Upsilon^*}\mathbf{w}\|^{3/2}}
	(\mathbf{D}_{j\Upsilon^*}^\top\mathbf{D}_{j\Upsilon^*}-
	\frac{3\mathbf{D}_{j\Upsilon^*}^\top\mathbf{D}_{j\Upsilon^*}\mathbf{w}\mathbf{w}^\top\mathbf{D}_{j\Upsilon^*}^\top
		\mathbf{D}_{j\Upsilon^*}}{2\|\mathbf{D}_{j\Upsilon^*}\mathbf{w}\|^2})\|_2.\end{aligned}$$
Combing this and the notation of $\mathbf{D}_{j\Upsilon^*}$, we obtain that
$$\begin{aligned}
	&\|\nabla^2\psi_{2\Upsilon^*}(\mathbf{w})\|_2\\
	&=\|\sum_{j\in\Gamma^*}\frac{1}{2\|\mathbf{D}_{j\Upsilon^*}\mathbf{w}\|^{3/2}}(1-
	\frac{3\mathbf{D}_{j\Upsilon^*}\mathbf{w}\mathbf{w}^\top\mathbf{D}_{j\Upsilon^*}^\top}{2\|\mathbf{D}_{j\Upsilon^*}\mathbf{w}\|^2})\mathbf{D}_{j\Upsilon^*}^\top\mathbf{D}_{j\Upsilon^*}\|_2\\
	&=
	\frac{1}{4}\|\sum_{j\in\Gamma^*}\frac{1}{2\|\mathbf{D}_{j\Upsilon^*}\mathbf{w}\|^{3/2}}
	\mathbf{D}_{j\Upsilon^*}^\top\mathbf{D}_{j\Upsilon^*}\|_2\\
	&\leq\frac{1}{4}\max_{j^*\in\Gamma^*}\|\mathbf{D}_{j\Upsilon^*}\mathbf{w}\|^{-3/2}
	=\frac{1}{4}\max_{j^*\in\Gamma^*}\|{w}_j|^{-3/2}\end{aligned}$$
as $\mathbf{D}_{j\Upsilon^*}=(\mathbf{e}_{p,j}^\top)_{\Upsilon^*}$, and
$$\begin{aligned}\|\nabla^2\psi_{2\Upsilon^*}(\mathbf{w})\|_2&\leq
	\|\sum_{j\in\Gamma^*}\frac{1}{2\|\mathbf{D}_{j\Upsilon^*}\mathbf{w}\|^{3/2}}
	\mathbf{D}_{j\Upsilon^*}^\top(\mathbf{I}_2-
	\frac{3\mathbf{D}_{j\Upsilon^*}\mathbf{w}\mathbf{w}^\top\mathbf{D}_{j\Upsilon^*}^\top}{2\|\mathbf{D}_{j\Upsilon^*}\mathbf{w}\|^2}
	)\mathbf{D}_{j\Upsilon^*}\|_2\\
	&\leq\frac{1}{2}\max_{j\in\Gamma^*}\|\mathbf{D}_{j\Upsilon^*}\mathbf{w}\|^{-3/2}\end{aligned}$$
as $$\mathbf{D}_{j\Upsilon^*}=\mathbf{W}_{j\Upsilon^*}=
\begin{bmatrix} \mathbf{e}_{2p,j}^\top \\
	\mathbf{e}_{2p,p+j}^\top \end{bmatrix}_{\Upsilon^*}.$$
Combing the above upper bounds and the mean value theorem enables to us derive that
\begin{equation}\label{psi2-grad-Lip}
	\begin{aligned}
		\|\nabla\psi_{2\Upsilon^*}(\mathbf{w})-\nabla\psi_{2\Upsilon^*}(\mathbf{w}^*)\|&=
		\|\int_n^1\langle\nabla\psi_{2\Upsilon^*}^2(\mathbf{w}^*+t(\mathbf{w}-\mathbf{w}^*)), \mathbf{w}-\mathbf{w}^*\rangle\mathrm{d}_t\|\\
		&\leq\int_n^1\|\nabla\psi_{2\Upsilon^*}^2(\mathbf{w}^*+t(\mathbf{w}-\mathbf{w}^*))\|_2 \|\mathbf{w}-\mathbf{w}^*\|\mathrm{d}_t\\
		&\leq\frac{1}{2}\int_n^1\max_{j\in\Gamma^*}\|\mathbf{D}_{j\Upsilon^*}(\mathbf{w}^*+t(\mathbf{w}-\mathbf{w}^*))\|^{-3/2}
		\mathrm{d}_t\|\mathbf{w}-\mathbf{w}^*\|\\
		&\leq\frac{1}{2}\max_{j\in\Gamma^*}\sup_{\|\mathbf{w}-\mathbf{w}^*\|\leq\epsilon_2}
		\|\mathbf{D}_{j\Upsilon^*}\mathbf{w}\|^{-3/2}\|\mathbf{w}-\mathbf{w}^*\|\\
		&\leq 2\max_{j\in\Gamma^*}\|\mathbf{D}_{j\Upsilon^*}\mathbf{w}^*\|^{-3/2}\|\mathbf{w}-\mathbf{w}^*\|
	\end{aligned}
\end{equation}
for any $\mathbf{w}\in\{\mathbf{w}: \|\mathbf{w}-\mathbf{w}^*\|<\varepsilon_1\}.$
Recall that  the second inequality of (\ref{hd-Lip}) implies that $\psi_{1\Upsilon^*}^i:=h_\alpha'(\mathbf{w}^\top\mathbf{B}_i^{\Upsilon^*\Upsilon^*}\mathbf{w}-b_i)\mathbf{B}_i^{\Upsilon^*\Upsilon^*}
\mathbf{w}$ is locally Lipschitz continuous for each $i=1,\cdots,n$.
Then, the generalized Jacobian  matrix of $\psi_{1\Upsilon^*}^i$ at $\mathbf{w}$, denoted by $J(\psi_{1\Upsilon^*}^i;\mathbf{w})$,  is the set of matrices defined as the convex hull of
$$\{\mathbf{M}: \exists\,\mathbf{w}_l\to\mathbf{w} \,\mbox{with}\,\psi_{1\Upsilon^*}^i\, \mbox{differentiable at}\, \mathbf{w}_l\, \mbox{and}\, J(\psi_{1\Upsilon^*}^i;\mathbf{w}_l)\to\mathbf{M}\},$$
where $J(\psi_{1\Upsilon^*}^i;\mathbf{w}_l)$ is the Jacobian matrix of $\psi_{1\Upsilon^*}^i$ at $\mathbf{w}_l$.
It is easy to check that
$$J(\psi_{1\Upsilon^*}^i;\mathbf{w})=\begin{cases}
	\mathbf{M}^i(\mathbf{w}),&{\mbox{if}\, \langle\mathbf{w},\mathbf{B}_i^{\Upsilon^*\Upsilon^*}\mathbf{w}\rangle<\alpha},\\
	\{\mathbf{0}\},& {\mbox{if}\, \langle\mathbf{w},\mathbf{B}_i^{\Upsilon^*\Upsilon^*}\mathbf{w}\rangle>\alpha,}\\
	\{\kappa_i\mathbf{M}^i(\mathbf{w}): \kappa_i\in[0,1]\},&{\mbox{if}\, \langle\mathbf{w},\mathbf{B}_i^{\Upsilon^*\Upsilon^*}\mathbf{w}\rangle=\alpha},
\end{cases}$$
where $$\mathbf{M}^i(\mathbf{w})=\frac{1}{n}\Big(2\mathbf{B}_i^{\Upsilon^*\Upsilon^*}\mathbf{w}\mathbf{w}^\top\mathbf{B}_i^{\Upsilon^*\Upsilon^*}
+(\langle\mathbf{w}, \mathbf{B}_i^{\Upsilon^*\Upsilon^*}\mathbf{w}\rangle-b_i)\mathbf{B}_i^{\Upsilon^*\Upsilon^*}\Big).$$
As a direct result, we have
\begin{equation}\label{psii-bound}
	\|J(\psi_{1\Upsilon^*}^i;\mathbf{w})\|_2\leq\|\mathbf{M}^i(\mathbf{w})\|_2.
\end{equation}

Notice that there exists a small enough number $\varepsilon_2\in(0,\varepsilon_1]$ such that 
for any $\mathbf{w}\in\{\mathbf{w}: \|\mathbf{w}-\mathbf{w}^*\|<\varepsilon_2\}$,
$$\max_{1\leq i\leq n}|\langle\mathbf{w},\mathbf{B}_i^{\Upsilon^*\Upsilon^*}\mathbf{w}\rangle-\langle\mathbf{w}^*,\mathbf{B}_i^{\Upsilon^*\Upsilon^*}\mathbf{w}^*\rangle|<\epsilon_1$$
which together with triangle inequality results in 
$$|\langle\mathbf{w}^*,\mathbf{B}_i^{\Upsilon^*\Upsilon^*}\mathbf{w}^*\rangle-b_i|-\epsilon_1<|\langle\mathbf{w},\mathbf{B}_i^{\Upsilon^*\Upsilon^*}\mathbf{w}\rangle-b_i|<
|\langle\mathbf{w}^*,\mathbf{B}_i^{\Upsilon^*\Upsilon^*}\mathbf{w}^*\rangle-b_i|+\epsilon_1.$$
Define  $\Gamma_1^{\alpha}(\mathbf{w})=\{i=1,\cdots,n: |\langle\mathbf{w},\mathbf{B}_i^{\Upsilon^*\Upsilon^*}\mathbf{w}\rangle-b_i|\leq\alpha\}$,  $\Gamma_2^{\alpha}(\mathbf{w})=\{i=1,\cdots,n: |\langle\mathbf{w},\mathbf{B}_i^{\Upsilon^*\Upsilon^*}\mathbf{w}\rangle-b_i|>\alpha\}$, $\Gamma_1^{\alpha,\epsilon_1}=\{i=1,\cdots,n: |\langle\mathbf{w}^*,\mathbf{B}_i^{\Upsilon^*\Upsilon^*}\mathbf{w}^*\rangle-b_i|\leq\alpha-\epsilon_1\}$, and  $\Gamma_2^{\alpha,\epsilon_1}=\{i=1,\cdots,n: |\langle\mathbf{w}^*,\mathbf{B}_i^{\Upsilon^*\Upsilon^*}\mathbf{w}^*\rangle-b_i|-\alpha|<\epsilon_1\}$.
For any $\mathbf{w}\in\{\mathbf{w}: \|\mathbf{w}-\mathbf{w}^*\|<\varepsilon_2\}$, then, we have 
$$\Gamma_1^{\alpha,\epsilon_1}\subseteq\{i=1,\cdots,n: |\langle\mathbf{w},\mathbf{B}_i^{\Upsilon^*\Upsilon^*}\mathbf{w}\rangle-b_i|<\alpha\}\subseteq\Gamma_1^{\alpha}(\mathbf{w})
\subseteq\Gamma_1^{\alpha,\epsilon_1}\cup\Gamma_2^{\alpha,\epsilon_1},$$ and rewrite $\psi_{1\Upsilon^*}(\mathbf{w})$ as follows
\begin{equation}\nonumber
	\begin{aligned}\psi_{1\Upsilon^*}(\mathbf{w})&=\sum_{i\in\Gamma_1^{\alpha}(\mathbf{w})}\psi_{1\Upsilon^*}^i(\mathbf{w})
		+\sum_{i\in\Gamma_2^{\alpha}(\mathbf{w})}\psi_{1\Upsilon^*}^i(\mathbf{w})\\
		&=\sum_{i\in\Gamma_1^{\alpha,\epsilon_1}}\psi_{1\Upsilon^*}^i(\mathbf{w})+\sum_{i\in\Gamma_1^{\alpha}(\mathbf{w})\setminus\Gamma_1^{\alpha,\epsilon_1}}\psi_{1\Upsilon^*}^i(\mathbf{w})
		+\sum_{i\in\Gamma_2^{\alpha}(\mathbf{w})}\psi_{1\Upsilon^*}^i(\mathbf{w})\\
		&\triangleq \psi_{1\Upsilon^*}^{(1)}(\mathbf{w})+\psi_{1\Upsilon^*}^{(2)}(\mathbf{w})+\psi_{1\Upsilon^*}^{(3)}(\mathbf{w}).
\end{aligned}\end{equation}
From the above analysis, one can see that both $\psi_{1\Upsilon^*}^{(1)}$ and $\psi_{1\Upsilon^*}^{(3)}$ are continuous differentiable and the corresponding Jacobian matrices at point $\mathbf{w}$ coincide with the Hessian matrices, i.e.,
$$J(\psi_{1\Upsilon^*}^{(1)};\mathbf{w})=\nabla^2\psi_{1\Upsilon^*}^{(1)}(\mathbf{w})=
\sum_{i\in\Gamma_1^{\alpha,\epsilon_1}}
\mathbf{M}^i(\mathbf{w}),$$ and $$ J(\psi_{1\Upsilon^*}^{(3)};\mathbf{w})=\nabla^2\psi_{1\Upsilon^*}^{(3)}(\mathbf{w})=\mathbf{0}.$$
The generalized Jacobian  matrix of $\psi_{1\Upsilon^*}^{(2)}$ at $\mathbf{w}$ satisfies
$$J(\psi_{1\Upsilon^*}^{(2)};\mathbf{w})\subseteq
\sum_{i\in\Gamma_1^{\alpha}(\mathbf{w})\setminus\Gamma_1^{\alpha,\epsilon_1}}
J(\psi_{1\Upsilon^*}^{i};\mathbf{w})\subseteq
\sum_{i\in\Gamma_2^{\alpha,\epsilon_1}}
J(\psi_{1\Upsilon^*}^{i};\mathbf{w}),$$
which together with (\ref{psii-bound}) and the triangle inequality yields that
$$\|J(\psi_{1\Upsilon^*}^{(2)};\mathbf{w})\|_2\leq
\sum_{i\in\Gamma_2^{\alpha,\epsilon_1}}
\|\mathbf{M}^i(\mathbf{w})\|_2\leq
\sum_{i\in\Gamma_2^{\alpha,\epsilon_1}}
\Big(\|\mathbf{M}^i(\mathbf{w}^{*})\|_2+\|\mathbf{M}^i(\mathbf{w})-\mathbf{M}^i(\mathbf{w}^{*})\|_2\Big).$$
Therefore, the generalized Jacobian  matrix of $\psi_{\Upsilon^*}$ at $\mathbf{w}$ satisfies
$$\begin{aligned}J(\psi_{\Upsilon^*};\mathbf{w})& \subseteq J(\psi_{1\Upsilon^*};\mathbf{w})+\lambda\nabla^2\psi_{2\Upsilon^*}(\mathbf{w})\\
	&\subseteq J(\psi_{1\Upsilon^*}^{(1)};\mathbf{w})
	+J(\psi_{1\Upsilon^*}^{(2)};\mathbf{w})
	+J(\psi_{1\Upsilon^*}^{(3)};\mathbf{w})+\lambda\nabla^2\psi_{2\Upsilon^*}(\mathbf{w})\\
	&=
	\nabla^2\psi_{1\Upsilon^*}^{(1)}(\mathbf{w})+J(\psi_{1\Upsilon^*}^{(2)};\mathbf{w})
	+\lambda\nabla^2\psi_{2\Upsilon^*}(\mathbf{w}),\end{aligned}$$
which together with the above inequality yields that the smallest eigenvalue of $J(\psi_{\Upsilon^*};\mathbf{w})$, denoted by
$\lambda_{\mathrm{min}}(J(\psi_{\Upsilon^*};\mathbf{w}))$, satisfies
\begin{equation}\label{smallest-eig-bound}\begin{aligned}
		\lambda_{\mathrm{min}}(J(\psi_{\Upsilon^*};\mathbf{w}))&\geq
		\lambda_{\mathrm{min}}(\nabla^2\psi_{1\Upsilon^*}^{(1)}(\mathbf{w}))-
		\|J(\psi_{1\Upsilon^*}^{(2)};\mathbf{w})\|_2
		-\lambda\|\nabla^2\psi_{2\Upsilon^*}(\mathbf{w})\|_2\\
		&\geq\lambda_{\mathrm{min}}(\nabla^2\psi_{1\Upsilon^*}^{(1)}(\mathbf{w}))-
		\|\sum_{i\in\Gamma_2^{\alpha,\epsilon_1}}
		\mathbf{M}^i(\mathbf{w}^*)\|_2-\lambda\|\nabla^2\psi_{2\Upsilon^*}(\mathbf{w}^*)\|_2\\
		&\qquad\qquad-\sum_{i\in\Gamma_2^{\alpha,\epsilon_1}}
		\|\mathbf{M}^i(\mathbf{w})-\mathbf{M}^i(\mathbf{w}^{*})\|_2-\lambda\|\nabla^2\psi_{2\Upsilon^*}(\mathbf{w})-\nabla^2\psi_{2\Upsilon^*}(\mathbf{w}^*)\|_2\\
		&\geq \lambda_{\mathrm{min}}(\nabla^2\psi_{1\Upsilon^*}^{(1)}(\mathbf{w}^*))-
		\|\sum_{i\in\Gamma_2^{\alpha,\epsilon_1}}
		\mathbf{M}^i(\mathbf{w}^*)\|_2-\lambda\|\nabla^2\psi_{2\Upsilon^*}(\mathbf{w}^*)\|_2\\
		&\qquad-\|\nabla^2\psi_{1\Upsilon^*}^{(1)}(\mathbf{w})-\nabla^2\psi_{1\Upsilon^*}^{(1)}(\mathbf{w}^*)\|_2
		-\sum_{i\in\Gamma_2^{\alpha,\epsilon_1}}
		\|\mathbf{M}^i(\mathbf{w})-\mathbf{M}^i(\mathbf{w}^{*})\|_2\\
		&\qquad\qquad\qquad-\lambda\|\nabla^2\psi_{2\Upsilon^*}(\mathbf{w})-\nabla^2\psi_{2\Upsilon^*}(\mathbf{w}^*)\|_2.
\end{aligned}\end{equation}
Noting $\Gamma_1^{\alpha,\epsilon_1}=\Gamma^{\alpha,\epsilon_1}$ and taking a sufficiently small positive number $\varepsilon_3\leq\varepsilon_2$ such that 
$$\|\nabla^2\psi_{1\Upsilon^*}^{(1)}(\mathbf{w})-\nabla^2\psi_{1\Upsilon^*}^{(1)}(\mathbf{w}^*)\|_2\leq\frac{\epsilon_2}{9},$$
$$
\sum_{i\in\Gamma_2^{\alpha,\epsilon_1}}\|\mathbf{M}^i(\mathbf{w})-\mathbf{M}^i(\mathbf{w^*})\|_2\leq\frac{\epsilon_2}{9}$$
and 
$$\lambda\|\nabla^2\psi_{2\Upsilon^*}(\mathbf{w})-\nabla^2\psi_{2\Upsilon^*}(\mathbf{w}^*)\|_2\leq\frac{\epsilon_2}{9}$$
for any $\mathbf{w}\in\{\mathbf{w}: \|\mathbf{w}-\mathbf{w}^*\|<\varepsilon_3\}$, we conclude from the conditions (\ref{gene-jaciban-pos-real}) and (\ref{gene-jaciban-pos-com}) and the inequality (\ref{smallest-eig-bound}) that
\begin{equation}\label{Jlowerbound}\lambda_{\mathrm{min}}(J(\psi_{\Upsilon^*};\mathbf{w}))\geq \frac{\epsilon_2}{3}.\end{equation}
For any $\mathbf{w}\in\{\mathbf{w}: \|\mathbf{w}-\mathbf{w}^*\|<\varepsilon_3\}$, by Theorem 8 in \cite{hiriart1980mean}, we obtain that  there exist real numbers $\sigma_l$, vectors $\mathbf{w}_l,$ matrices $\mathbf{M}_l$, $l=1,\cdots,m$ such that ${\varpi}_l\geq0$, $\mathbf{w}_l$ lying in the segment $[\mathbf{w}, \mathbf{w}^*]$, $\mathbf{M}_l\in J(\psi_{\Upsilon^*},\mathbf{w}_l)$ for all $l$, $\sum_{l=1}^m{\varpi}_l=1$ and
$$\nabla\psi_{\Upsilon^*}(\mathbf{w})-\nabla\psi_{\Upsilon^*}(\mathbf{w}^*)=
\sum_{l=1}^m{\varpi}_l\mathbf{M}_l(\mathbf{w}-\mathbf{w}^*),
$$
which together with (\ref{Jlowerbound}) leads to
\begin{equation}\label{grad-diff-bound}
	\begin{aligned}
		\|\nabla\psi_{\Upsilon^*}(\mathbf{w})-\nabla\psi_{\Upsilon^*}(\mathbf{w}^*)\|&=
		\|\sum_{l=1}^m{\varpi}_l\mathbf{M}_l(\mathbf{w}-\mathbf{w}^*)\|\\
		&\geq\lambda_{\mathrm{min}}(\sum_{l=1}^m{\varpi}_l\mathbf{M}_l)\|\mathbf{w}-\mathbf{w}^*\|
		\geq\frac{\epsilon_2}{3}\|\mathbf{w}-\mathbf{w}^*\|.
	\end{aligned}
\end{equation}
It follows from (\ref{psi-eta}) and $\psi(\mathbf{u})=\psi_{\Upsilon^*}(\mathbf{u}_{\Upsilon^*})$ that the set $\{\mathbf{w}: \psi_{\Upsilon^*}(\mathbf{w}^*)<\psi_{\Upsilon^*}(\mathbf{w})<\psi_{\Upsilon^*}(\mathbf{w}^*)+\eta\}$ is nonempty. For any $\mathbf{w}\in\{\mathbf{w}: \psi_{\Upsilon^*}(\mathbf{w}^*)<\psi_{\Upsilon^*}(\mathbf{w})<\psi_{\Upsilon^*}(\mathbf{w}^*)+\eta\}\cap\{\mathbf{w}: \|\mathbf{w}-\mathbf{w}^*\|<\varepsilon_3\}$,
we conclude from the mean value theorem, (\ref{psi-res-0}), (\ref{Lpsi1}) and (\ref{psi2-grad-Lip}) that
\begin{equation}\label{mvt-psi}
	\begin{aligned}
		0&<\psi_{\Upsilon^*}(\mathbf{w})-\psi_{\Upsilon^*}(\mathbf{w}^*)\\
		&=
		\langle\nabla\psi_{\Upsilon^*}(\mathbf{w}^*+t(\mathbf{w}-\mathbf{w}^*)), \mathbf{w}-\mathbf{w}^*\rangle\\
		&=\langle\nabla\psi_{\Upsilon^*}(\mathbf{w}^*+t(\mathbf{w}-\mathbf{w}^*))-\nabla\psi_{\Upsilon^*}(\mathbf{w}^*), \mathbf{w}-\mathbf{w}^*\rangle
		+\langle\nabla\psi_{\Upsilon^*}(\mathbf{w}^*), \mathbf{w}-\mathbf{w}^*\rangle\\
		&=\langle\nabla\psi_{\Upsilon^*}(\mathbf{w}^*+t(\mathbf{w}-\mathbf{w}^*))-\nabla\psi_{\Upsilon^*}(\mathbf{w}^*), \mathbf{w}-\mathbf{w}^*\rangle\\
		&\leq\|\nabla\psi_{\Upsilon^*}(\mathbf{w}^*+t(\mathbf{w}-\mathbf{w}^*))-\nabla\psi_{\Upsilon^*}(\mathbf{w}^*)\|\|\mathbf{w}-\mathbf{w}^*\|\\
		&\leq (\|\nabla\psi_{1\Upsilon^*}(\mathbf{w}^*+t(\mathbf{w}-\mathbf{w}^*))-\nabla\psi_{1\Upsilon^*}(\mathbf{w}^*)\|\|\mathbf{w}-\mathbf{w}^*\|\\
		&\qquad+\lambda\|\nabla\psi_{2\Upsilon^*}(\mathbf{w}^*+t(\mathbf{w}-\mathbf{w}^*))-\nabla\psi_{2\Upsilon^*}(\mathbf{w}^*)\|\|\mathbf{w}-\mathbf{w}^*\|\\
		&\leq L_{\psi^*}\|\mathbf{w}-\mathbf{w}^*\|^2,
	\end{aligned}\nonumber
\end{equation}
where $L_{\psi^*}=\frac{1}{n}\sum_{i=1}^n(\alpha+(1+2\|\mathbf{B}_i\|_2\varepsilon_3^2))\|\mathbf{B}_i\|_2\|\mathbf{w}^*\|
+2\lambda\max_{j\in\Gamma^*}\|\mathbf{D}_{j\Upsilon^*}\mathbf{w}^*\|^{-3/2}$.
By the above inequality and (\ref{grad-diff-bound}), we then get
\begin{equation}\label{KL-half}
	\begin{aligned}
		\mathrm{dist}(\mathbf{0}, \nabla\psi_{\Upsilon^*}(\mathbf{w}))\geq\frac{\epsilon_2}{3}\|\mathbf{w}-\mathbf{w}^*\|
		\geq\frac{\epsilon_2}{3L_{\psi*}^{1/2}}(\psi_{\Upsilon^*}(\mathbf{w})-\psi_{\Upsilon^*}(\mathbf{w}^*))^{1/2}.
	\end{aligned}\nonumber
\end{equation}
This means the function $\psi_{\Upsilon^*}$ satisfies the KL property with $\phi(t)=\frac{3L_{\psi^*}^{1/2}}{\epsilon_2}t^{1/2}$ at the point $\mathbf{w}^*=\mathbf{u}_{\Upsilon^*}^*$. Thus, we complete the proof. 
\end{proof}

\begin{remark}\label{r5}
	To show the reasonableness of (\ref{gene-jaciban-pos-real}) and (\ref{gene-jaciban-pos-com}), we take the real-valued case as an example. Consider the matrix
	$\mathbf{H}_i(\mathbf{x})=\frac{1}{n}(3\langle\mathbf{a}_i,\mathbf{x}\rangle^2-b_i)
	\mathbf{a}_{i\Gamma^\diamond}\mathbf{a}_{i\Gamma^\diamond}^\top$
	which can be rewritten as $$\mathbf{H}_i(\mathbf{x})=\frac{2}{n}\langle\mathbf{a}_i,\mathbf{x}\rangle^2
	\mathbf{a}_{i\Gamma^\diamond}\mathbf{a}_{i\Gamma^\diamond}^\top-
	\frac{1}{n}\varepsilon_i\mathbf{a}_{i\Gamma^\diamond}\mathbf{a}_{i\Gamma^\diamond}^\top.$$ It is easy to check that
		$$\{i=1,\cdots,n: |\langle\mathbf{x}^\diamond,\mathbf{a}_i\mathbf{a}_i^{\top}\mathbf{x}^\diamond\rangle-b_i|\leq\alpha-\epsilon_1\}=\{i=1,\cdots,n: |\varepsilon_i|\leq\alpha-\epsilon_1\}$$
		and 
		$$\{i=1,\cdots,n: |\langle\mathbf{x}^\diamond,\mathbf{a}_i\mathbf{a}_i^{\top}\mathbf{x}^\diamond\rangle-b_i|-\alpha\big|<\epsilon_1\}=\{i=1,\cdots,n: \big||\varepsilon_i|-\alpha\big|<\epsilon_1\}$$
		Similar to the proof of Lemma \ref{prob-E}, one can see that  both $$\|\frac{1}{n}\sum_{|\varepsilon_i|\leq\alpha-\epsilon_1}\varepsilon_i
	\mathbf{a}_{i\Gamma^\diamond}\mathbf{a}_{i\Gamma^\diamond}^\top\|_2
	\quad \mbox{and}\quad\|\frac{1}{n}\sum_{\big||\varepsilon_i|-\alpha\big|<\epsilon_1}\varepsilon_i
	\mathbf{a}_{i\Gamma^\diamond}\mathbf{a}_{i\Gamma^\diamond}^\top\|_2$$ are very small with high probability. By simple calculation, we also have 
	$$\begin{aligned}
		\|\sum_{\big||\varepsilon_i|-\alpha\big|<\epsilon_1}\mathbf{H}_i(\mathbf{x}^\diamond)\|_2
		&\leq \|\frac{2}{n}\sum_{\big||\varepsilon_i|-\alpha\big|<\epsilon_1}\langle\mathbf{a}_i,\mathbf{x}^\diamond\rangle^2
		\mathbf{a}_{i\Gamma^\diamond}\mathbf{a}_{i\Gamma^\diamond}^\top\|_2+
		\frac{1}{n}\|\sum_{\big||\varepsilon_i|-\alpha\big|<\epsilon_1}\varepsilon_i\mathbf{a}_{i\Gamma^\diamond}\mathbf{a}_{i\Gamma^\diamond}^\top\|_2\\
		&\leq\frac{2}{n}\sum_{\big||\varepsilon_i|-\alpha\big|<\epsilon_1}\|\mathbf{a}_i\|^4\|\mathbf{x}^\diamond\|^2
		+
		\frac{1}{n}\|\sum_{\big||\varepsilon_i|-\alpha\big|<\epsilon_1}\varepsilon_i\mathbf{a}_{i\Gamma^\diamond}\mathbf{a}_{i\Gamma^\diamond}^\top\|_2.
	\end{aligned}$$
	Take $\epsilon_1=(1-\rho_0)\alpha$, and then $\{|\varepsilon_i|\leq\alpha-\epsilon_1\}=\mathcal{I}_{0}^{\mathrm{in}} $ and $\{\big||\varepsilon_i|-\alpha\big|<\epsilon_1\}\subseteq
	\{i: |\varepsilon_i|>\rho_0\alpha\}$. Given the conditions (\ref{cond-ai-c2}) and (\ref{cond-lambda-consistency}), it is not hard to check that
	$$\lambda_{\mathrm{min}}(\frac{2}{n}\sum_{|\varepsilon_i|\leq\alpha-\epsilon_1}(\mathbf{a}_i,\mathbf{x}^\diamond)^2
	\mathbf{a}_{i\Gamma^\diamond}\mathbf{a}_{i\Gamma^\diamond}^\top)\geq 2\underline{C}_2\|\mathbf{x}^\diamond\|^2$$
	and $\lambda\max_{j^\diamond\in\Gamma^\diamond}|x_j^\diamond|^{-3/2}$ is very small. 
	Therefore, in the corruption model mentioned in Remark \ref{explanation-con}, the above analysis implies that 
	$$\begin{aligned}&\lambda_{\mathrm{min}}(\sum_{|\varepsilon_i|\leq\alpha-\epsilon_1}\mathbf{H}_i(\mathbf{x}^\diamond))\\
		&\geq\lambda_{\mathrm{min}}(\frac{2}{n}\sum_{|\varepsilon_i|\leq\alpha-\epsilon_1}(\mathbf{a}_i,\mathbf{x}^\diamond)^2
		\mathbf{a}_{i\Gamma^\diamond}\mathbf{a}_{i\Gamma^\diamond}^\top)-\|\frac{1}{n}\sum_{|\varepsilon_i|\leq\alpha-\epsilon_1}\varepsilon_i
		\mathbf{a}_{i\Gamma^\diamond}\mathbf{a}_{i\Gamma^\diamond}^\top\|_2\\		
		& \geq 3\|\sum_{\big||\varepsilon_i|-\alpha\big|<\epsilon_1}\mathbf{H}_i(\mathbf{x}^\diamond)\|_2+\frac{3}{4}\lambda\max_{j^\diamond\in\Gamma^\diamond}|x_j^\diamond|^{-3/2}\end{aligned}$$
	when the proportion of corrupted measurements is relatively small and $\|\mathbf{a}_i\|^2=1$ for each $i=1,\cdots,n$. Further, since
 each \(\mathbf{H}_i(\mathbf{x})\) depends continuously on \(\mathbf{x}\),
	 the residuals $|\langle \mathbf{x},\mathbf{a}_i \mathbf{a}_i^\top \mathbf{x}\rangle - b_i|$ are continuous in \(\mathbf{x}\), and the factor \(\max_{j\in\Gamma^\diamond}|x_j|^{-3/2}\) is continuous in a neighborhood of \(\mathbf{x}^\diamond\) (because \(|x_j^\diamond|>0\) for \(j\in\Gamma^\diamond\)),
		the same spectral gap inequality remains valid throughout a sufficiently small neighborhood of \(\mathbf{x}^\diamond\). This means that the	condition 
 (\ref{gene-jaciban-pos-real}) is  reasonable.
\end{remark}

\section{Numerical Experiments}\label{numerical}
This section conducts numerical examples to illustrate the superiority of the proposed method with several popular methods, i.e., \verb"M-RWF" \cite{ZhangHuishuai2018Median}, \verb"M-MRAF" \cite{2021Median}, \verb"L0L1-PR" \cite{2016}, \verb"LAD-ADMM" \cite{2020Robust} and \verb"L1/2-LAD" \cite{kong20221}. All experiments are implemented in MATLAB(R2018b) on a laptop with 8GB memory. 

\subsection{Setup}

\begin{figure}[t]
	\centering
	\subfigure[Type-I (R)]{
		\includegraphics[width=4.8cm]{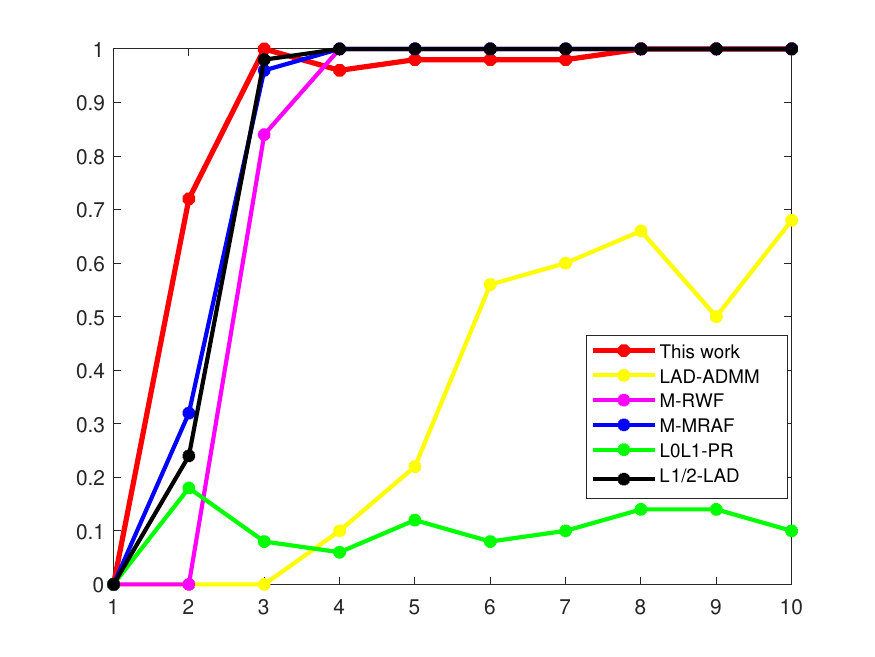}
	}
	\subfigure[Type-II (R)]{
		\includegraphics[width=4.8cm]{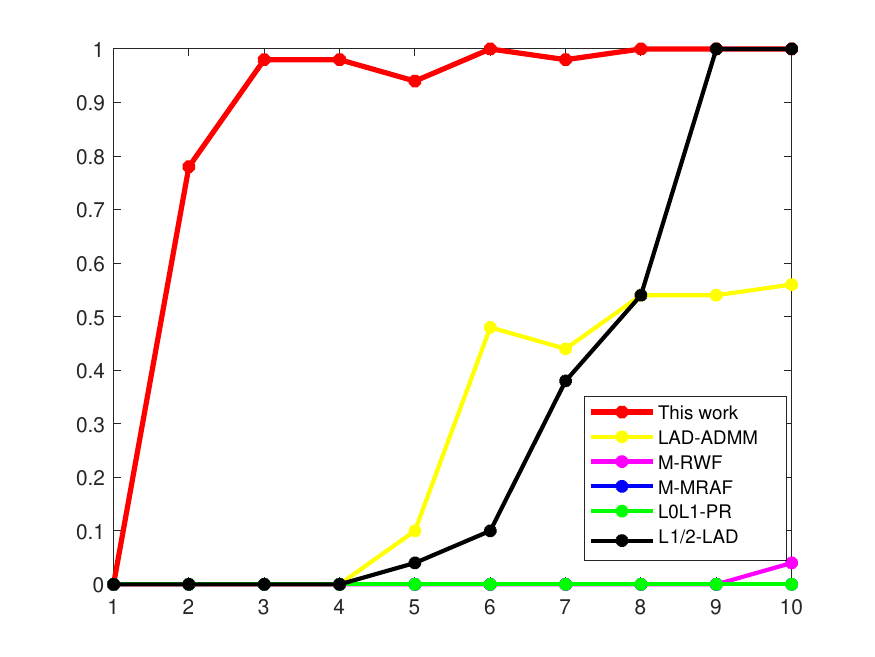}
	}
	\subfigure[Type-III (R)]{
		\includegraphics[width=4.8cm]{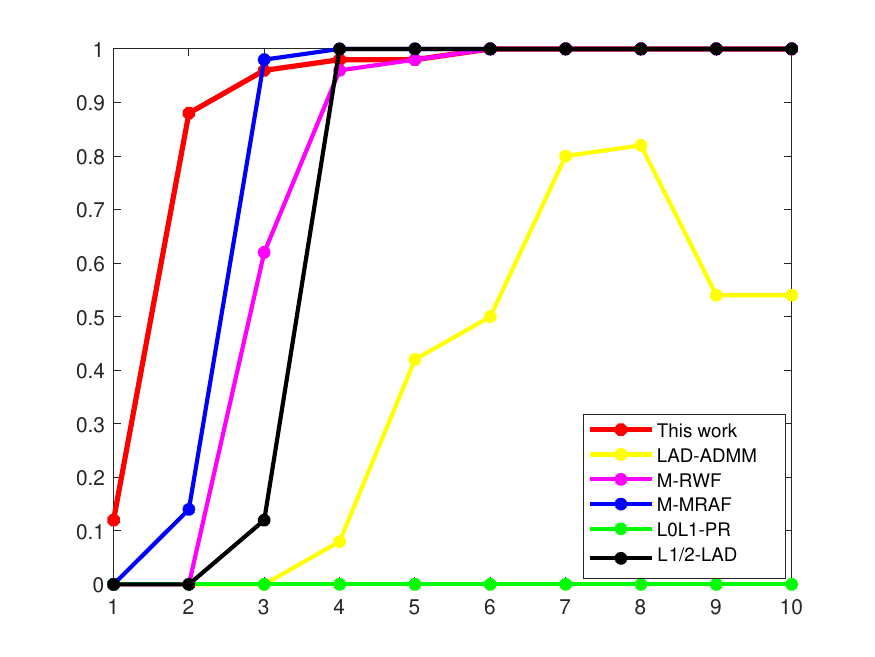}
	}
	\subfigure[Type-I (C)]{
		\includegraphics[width=4.8cm]{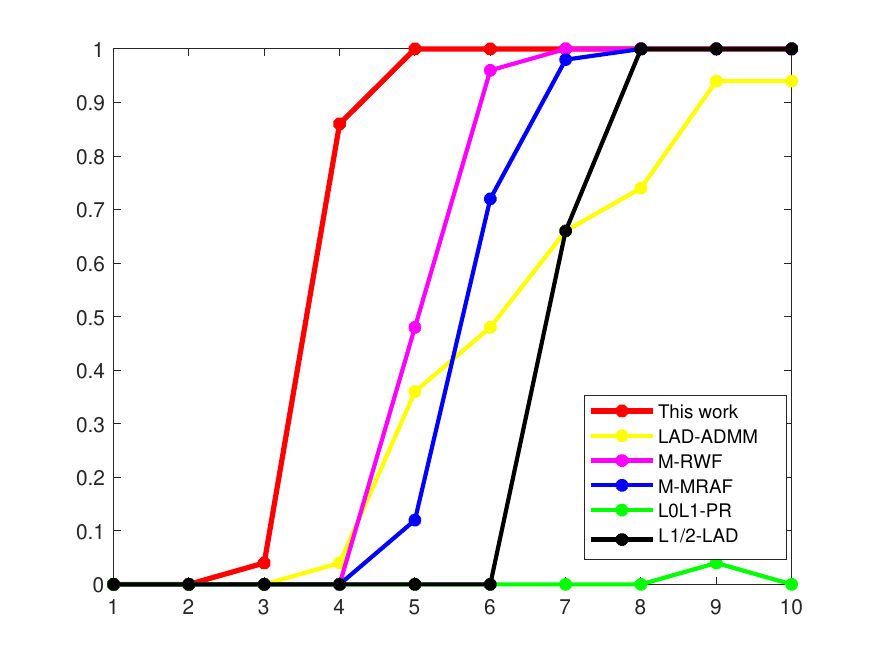}
	}
	\subfigure[Type-II (C)]{
		\includegraphics[width=4.8cm]{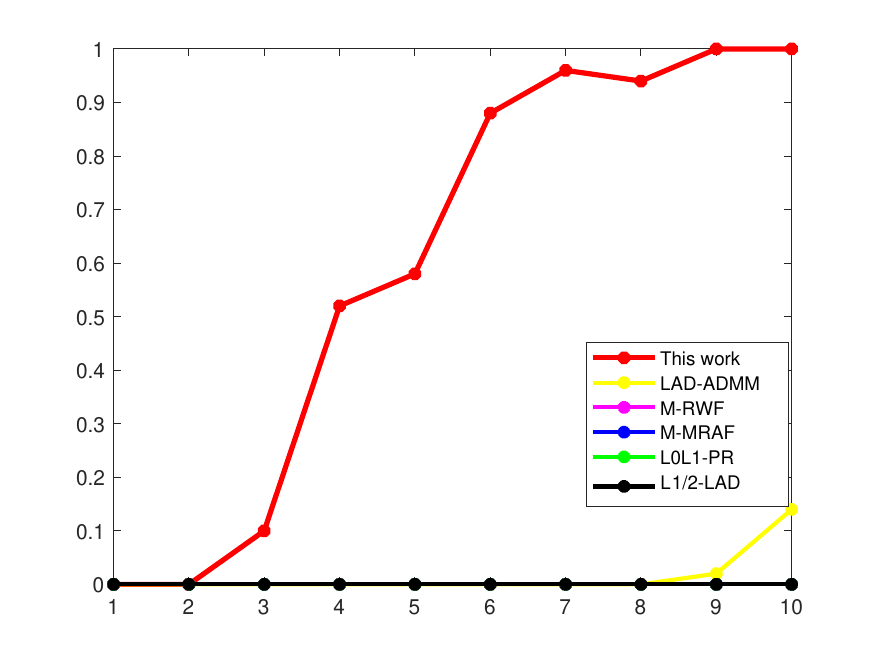}
	}
	\subfigure[Type-III (C)]{
		\includegraphics[width=4.8cm]{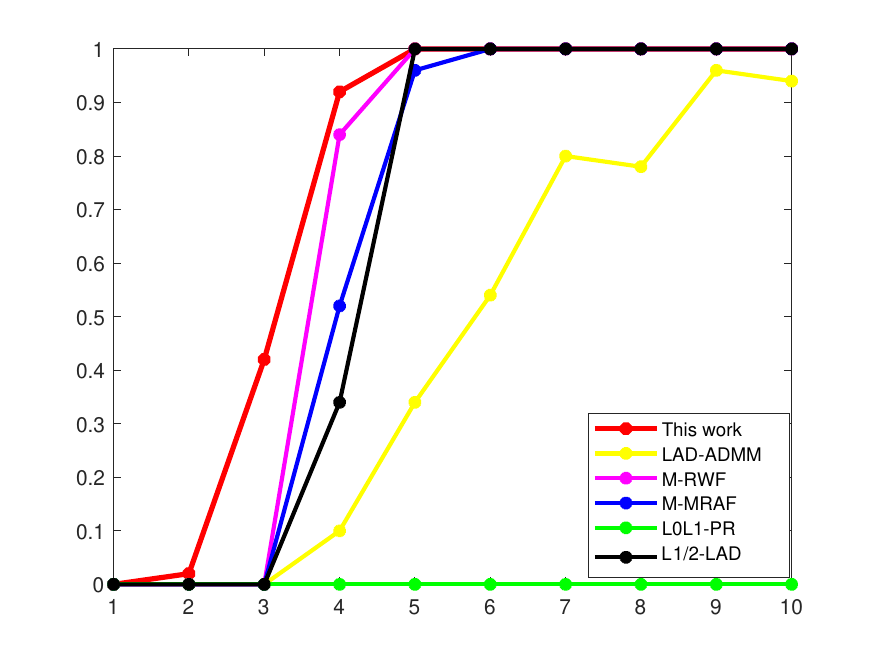}
	}
\vspace{-2mm}
	\caption{Success rates versus $ n/p$ for signal examples: (a)-(c) are the results for real-valued cases,  (d)-(f) are the results for complex-valued cases.}\label{rate}
\end{figure}

	\begin{table}[!t]
	\centering
	\caption{Running time (in seconds) under different $n/p$ for real-valued cases. The best results are marked in bold.}\label{time1}
	\begin{tabular}{l l|c|c|c|c|c|c}
			\toprule		
		~ &~ &\verb"M-RWF"&\verb"M-MRAF"&\verb"L0L1-PR" &\verb"LAD-ADMM" &\verb"L1/2-LAD" &\verb"This work" \\ \midrule	
		&$n=2p$         &\textbf{0.01}  &0.03   &9.74    &0.50      &66.16     &0.41\\ 								
		&$n=4p$         &\textbf{0.01} &0.03   &17.64   &10.70     &75.90     &0.22\\ 									
		Type-I&$n=6p$   &\textbf{0.01}  &0.04   &23.96   &68.28     &92.50     &0.18\\ 		
		&$n=8p$         &\textbf{0.02}  &0.04   &31.07   &308.14    &111.75    &0.26\\	
		&$n=10p$        &\textbf{0.02}  &0.05   &37.05   &356.20    &126.22    &0.22\\  		
		\midrule				
		&$n=2p$        &\textbf{0.07} &0.17   &15.64   &0.59     &95.26   &0.50\\ 								
		&$n=4p$        &0.32 &0.80   &25.74   &7.94     &110.99  &\textbf{0.26}\\ 									
		Type-II&$n=6p$ &0.50 &1.24   &37.21   &88.32    &150.70  &\textbf{0.33}\\ 		
		&$n=8p$        &0.72  &1.76   &54.02   &427.12   &209.31  &\textbf{0.32}\\	
		&$n=10p$       &0.89  &2.18   &64.11   &565.66   &224.93  &\textbf{0.41}\\  		
		\midrule		
		&$n=2p$         &\textbf{0.01}  &0.03   &13.66   &0.58      &92.60   &0.26\\ 								
		&$n=4p$         &0.02  &\textbf{0.01}   &23.46   &13.08     &34.91   &0.18\\ 									
		Type-III&$n=6p$ &0.02  &\textbf{0.01}  &33.96   &41.98     &20.14   &0.18\\ 		
		&$n=8p$         &0.03  &\textbf{0.01}   &44.34   &203.52    &30.84   &0.23\\	
		&$n=10p$        &0.04  &\textbf{0.01}  &56.42   &345.63    &21.03   &0.39\\  		
	\bottomrule
	\end{tabular}
\end{table}

	\begin{table}[t]
	\centering
	\caption{Running time (in seconds) under different $n/p$ for complex-valued cases. The best results are marked in bold.}\label{time2}
	\begin{tabular}{l l|c|c|c|c|c|c}
			\toprule		
		~ &~ &\verb"M-RWF"&\verb"M-MRAF"&\verb"L0L1-PR" &\verb"LAD-ADMM" &\verb"L1/2-LAD" &\verb"This work" \\ \midrule	
		&$n=2p$         &\textbf{0.02}  &0.05   &14.77    &3.14       &58.91      &0.29\\ 								
		&$n=4p$         &\textbf{0.02}  &0.06   &21.84    &223.51     &80.43      &0.42\\ 									
		Type-I&$n=6p$   &\textbf{0.03}  &0.07   &30.99    &711.59     &107.21     &0.26\\ 		
		&$n=8p$         &\textbf{0.04}  &0.09   &47.91    &1327.11    &155.58     &0.25\\	
		&$n=10p$        &\textbf{0.05}  &0.10   &66.125   &1459.33    &233.05     &0.28\\  		
		\midrule			
		&$n=2p$        &\textbf{0.03}  &0.06   &22.85   &1.67       &105.84   &0.60\\ 								
		&$n=4p$        &\textbf{0.04}  &0.08   &32.05   &261.59     &134.32   &0.89\\ 									
		Type-II&$n=6p$ &\textbf{0.05}  &0.10   &40.79   &773.32     &163.65   &0.58\\ 		
		&$n=8p$        &\textbf{0.06}  &0.13   &57.96   &1246.02    &230.07   &0.60\\	
		&$n=10p$       &\textbf{0.08}  &0.16   &70.50   &1565.18    &266.72   &0.62\\  		
	\midrule		
		&$n=2p$         &\textbf{0.02}  &0.06   &19.52   &3.69       &100.79   &0.54\\ 								
		&$n=4p$         &\textbf{0.04}  &0.08   &27.54   &157.62     &116.70   &0.36\\ 									
		Type-III&$n=6p$ &0.05  &\textbf{0.02}   &35.84   &440.57     &22.01    &0.25\\ 		
		&$n=8p$         &0.05  &\textbf{0.02}   &47.64   &553.92     &7.90     &0.28\\	
		&$n=10p$        &0.07  &\textbf{0.02}   &62.48   &528.91     &6.82     &0.23\\  		
		\bottomrule
	\end{tabular}
\end{table}

\begin{figure}[!t]
	\centering
	\subfigure[Type-I (R)]{
		\includegraphics[width=4.8cm]{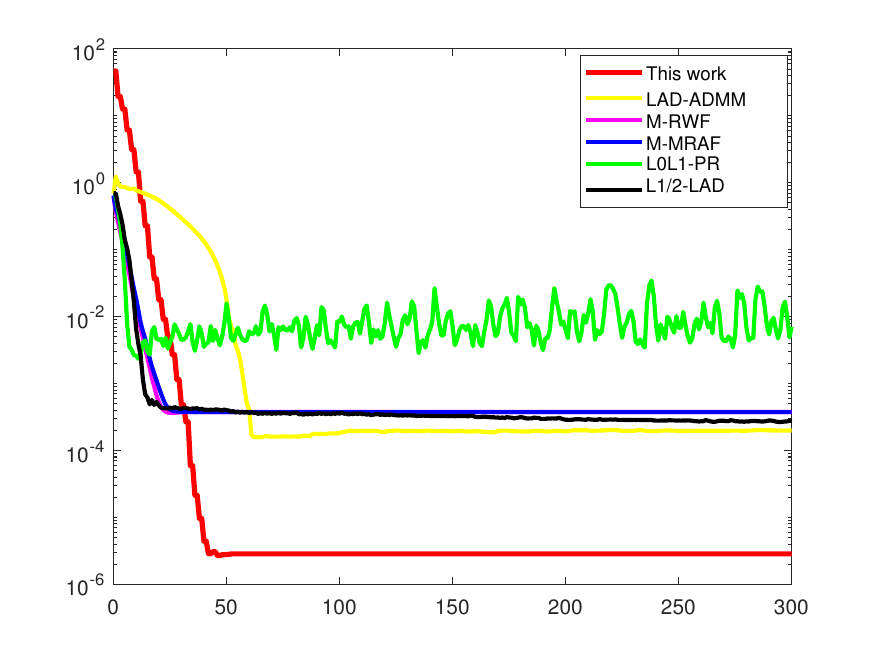}
	}
	\subfigure[Type-II (R)]{
		\includegraphics[width=4.8cm]{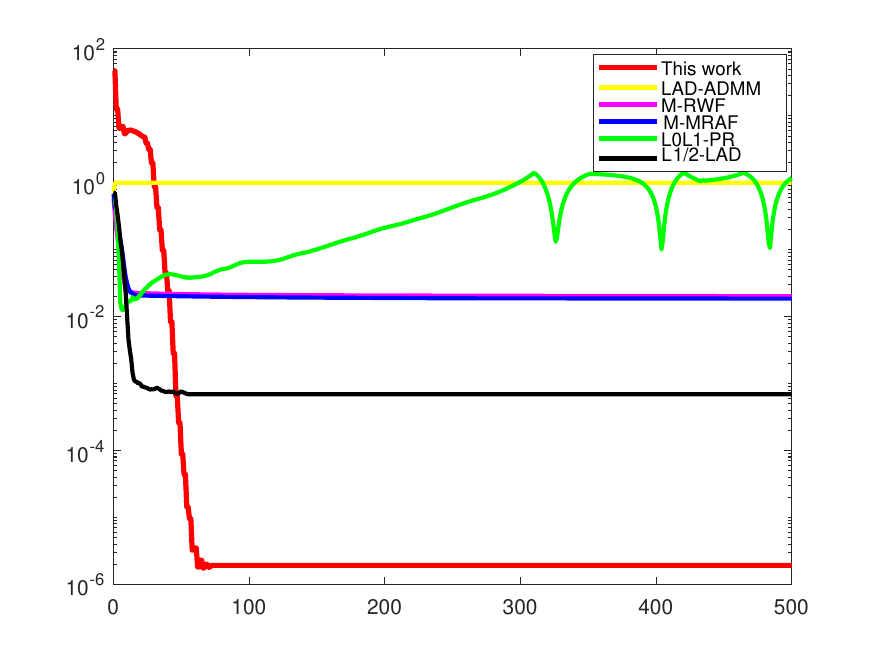}
	}
	\subfigure[Type-III (R)]{
		\includegraphics[width=4.8cm]{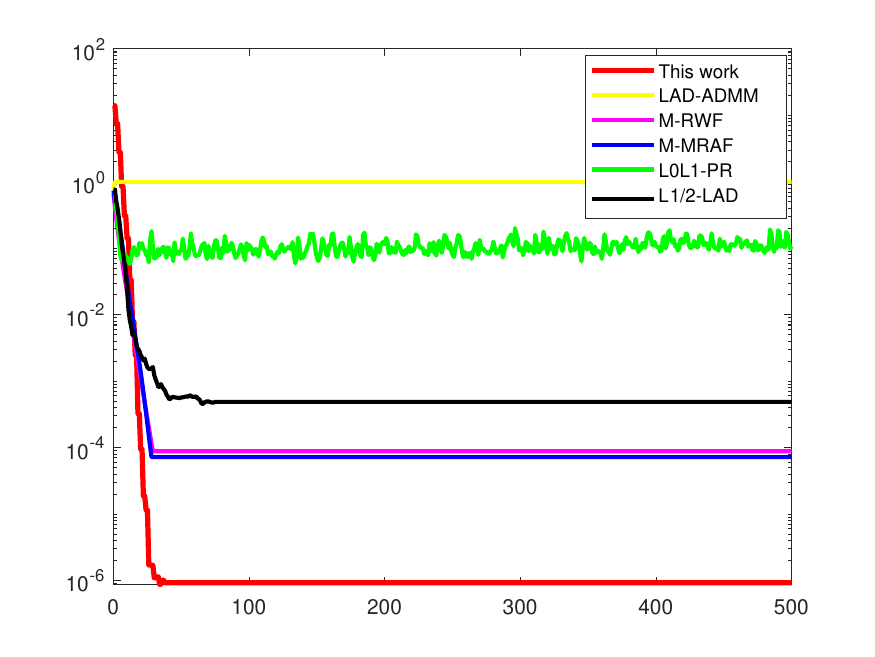}
	}
	\subfigure[Type-I (C)]{
		\includegraphics[width=4.8cm]{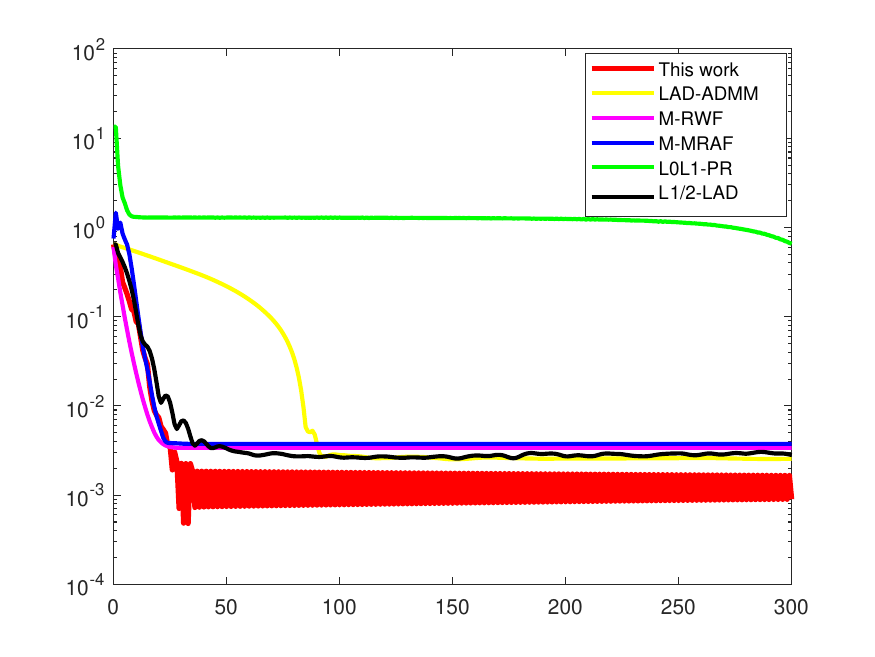}
	}
	\subfigure[Type-II (C)]{
		\includegraphics[width=4.8cm]{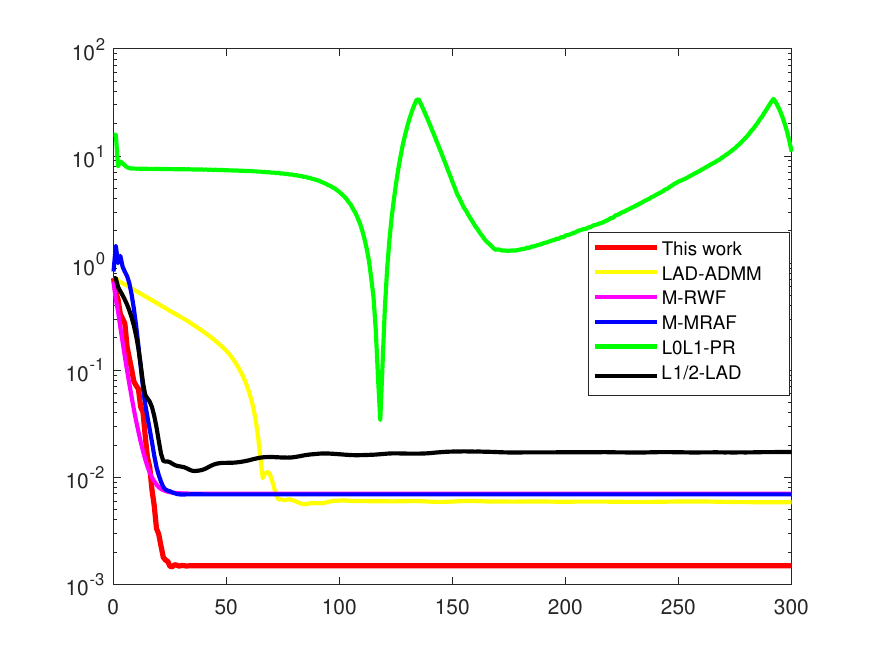}
	}
	\subfigure[Type-III (C)]{
		\includegraphics[width=4.8cm]{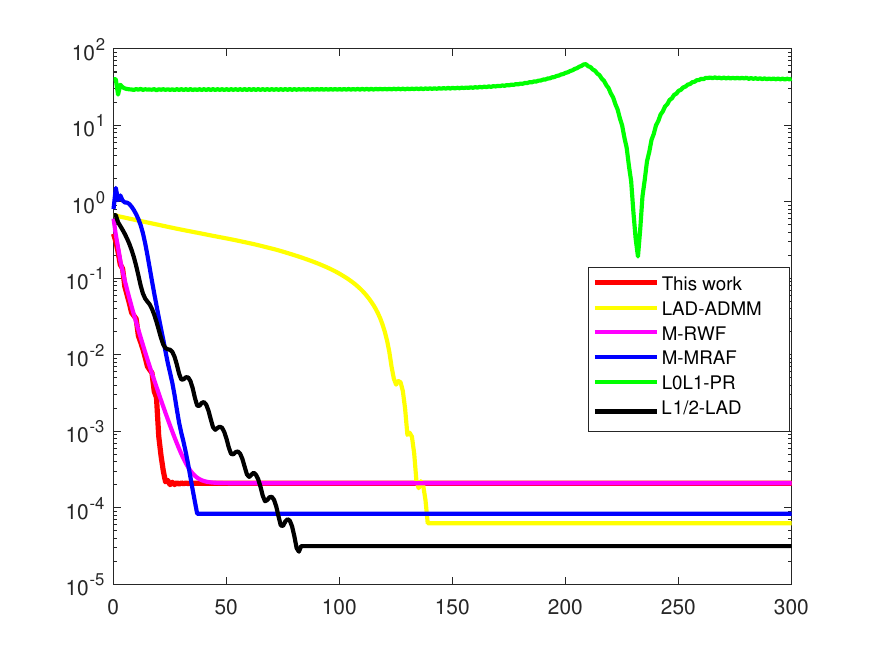}
	}
\vspace{-2mm}
	\caption{Relative errors versus iterations with $n=6p$ for signal examples: (a)-(c) are the results for real-valued cases,  (d)-(f) are the results for complex-valued cases.}\label{error}
\end{figure}

\begin{figure}[!t]
	\centering
	\subfigure[HI]{
		\includegraphics[width=4.8cm]{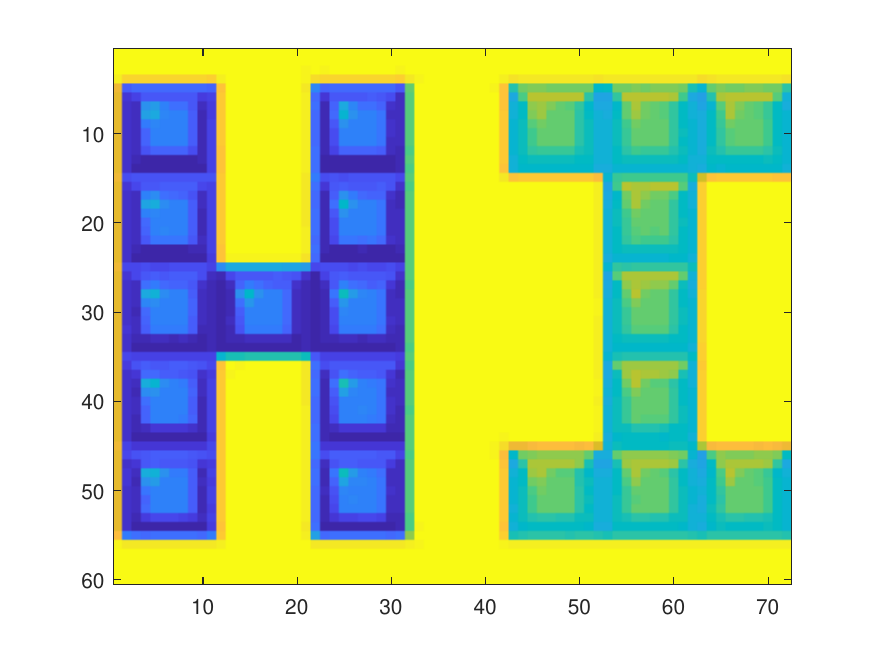}
	}
	\subfigure[Star]{
		\includegraphics[width=4.8cm]{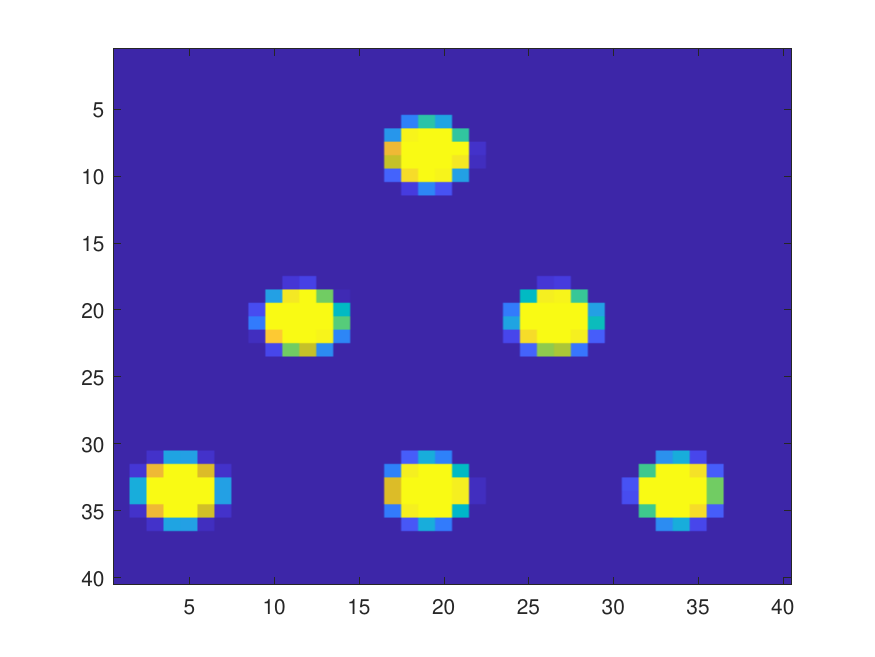}
	}
	\subfigure[Cell]{
		\includegraphics[width=4.8cm]{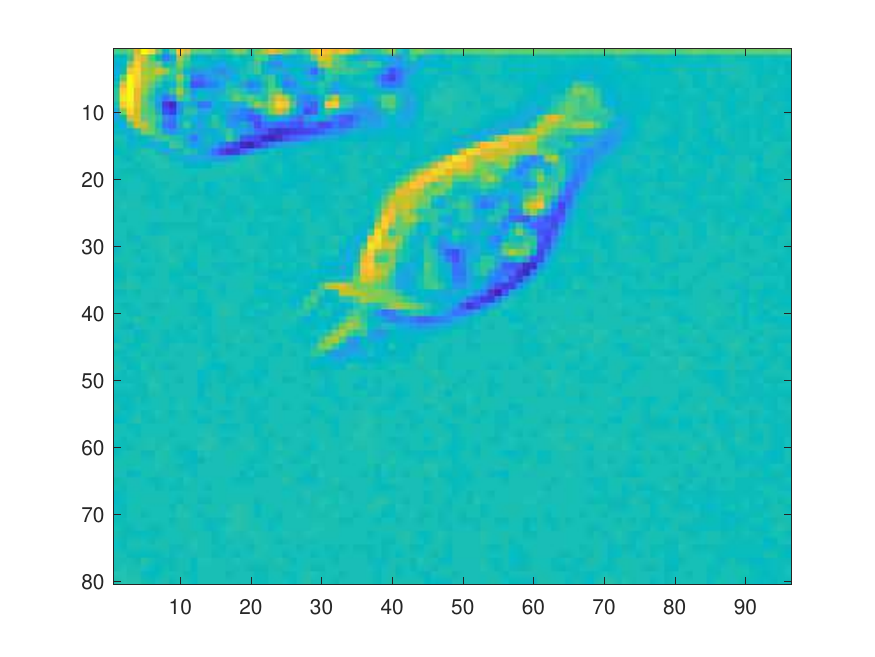}
	}
	\caption{Tested images: (a) HI with the pixel size $72\times60$, (b) Star with the pixel size $40\times40$, (c) Cell with the pixel size $96\times80$.}\label{images}
\end{figure}

For the parameters $\lambda$ and $\alpha$, they are both tuned by cross-validation techniques. Particularly, $\alpha$ is set as 1.345 for noise cases and 0.1345 for outlier cases, respectively.
The initial point is generated from the spectral initialization based on \cite{candesWF} like in \cite{ZhangHuishuai2018Median,2021Median,kong20221} for real-valued cases, and chooses the vector that sets all but the largest (in magnitude) $2s$ elements of the spectral initialization to zero for complex-valued cases. It is worth mentioning that the sparsity $s$ has a certain influence on simulation results. Considering the number $s$ is unknown for some cases, a rough estimate of $s$ from the priori information is taken. Of course, one can replace $2s$ with another number not far from $s$.
According to \cite{candesWF,2021Sample}, the relative error is determined by
\begin{equation}
 \text{{\normalsize{relative error}}}=\min_{\theta\in[0,2\pi]}\frac{\|\hat{\mathbf{x}}-e^{i\theta}\mathbf{x}_{true}\|}{\|\mathbf{x}_{true}\|},\nonumber
\end{equation}
where $\hat{\mathbf{x}}$ denotes the solution obtained by optimization solvers and $\mathbf{x}_{true}$ denotes the true signal. We say that the recovery is successful if the relative error is less than 5e-3. The success rate is the percentage of the successful numbers over 50 Monte Carlo runs.

In order to evaluate the performance of all compared methods, we consider three different types of noise $\{\varepsilon_i\}_{i=1}^n$ as follows.
\begin{itemize}
  \item Type-I: Dense bounded noise \cite{ZhangHuishuai2018Median}, which is generated independently from the uniform distribution $\mathcal{U}(0,\mu)$, where $\mu=\eta\|\mathbf{x}_{true}\|^2$ and $\eta=0.1$.
  \item Type-II: Laplace noise \cite{weller15}, which is generated independently by $\mathrm{Laplace}(0,\mu/\sqrt{2})$, where $\mu=\eta\sqrt{\sum_{i=1}^nb_i^2/n}$ with $\eta=0.1$.
  \item Type-III: Outliers \cite{2021Median}, each measurement $b_i$ is corrupted by $\mathcal{U}(0,\|\mathbf{x}_{true}\|^2)$ with $Bernoulli(\eta)$ and $\eta=0.1$.
\end{itemize}

\begin{table}[!t]
	\centering
	\caption{Relative errors (running time) under the Type-I noise for tested images. The best results with the corresponding time are marked in bold.}\label{rel-dense} 
	\begin{tabular}{l  l|c|c|c|c|c }
		\toprule	
	~ &~ &~~\verb"M-RWF"~~&~~\verb"M-MRAF"~~&~~\verb"L0L1-PR"~~ &~~\verb"L1/2-LAD"~~ &~~\verb"This work"~~ \\ 	\midrule
		&$\eta=0.001$        &\tabincell{c}{3.88e-4\\(5.35)}  &\tabincell{c}{3.85e-2\\(10.99)}   &\tabincell{c}{4.08e-4\\(180.55)}     &\tabincell{c}{2.67e-3\\(180.48)}      &\tabincell{c}{\textbf{3.53e-4}\\\textbf{(4.47)}} \\ 								
		HI&$\eta=0.01$         &\tabincell{c}{3.85e-3\\(6.67)}  &\tabincell{c}{3.83e-3\\(14.30)}    &\tabincell{c}{3.72e-3\\(180.39)}      &\tabincell{c}{4.72e-2\\(180.37)}      &\tabincell{c}{\textbf{3.20e-3}\\\textbf{(4.90)}}\\ 	
		      &$\eta=0.1$        &\tabincell{c}{3.83e-2\\(8.03)}  &\tabincell{c}{3.20e-2\\(13.52)}   &\tabincell{c}{3.66e-1\\(180.78)}     &\tabincell{c}{1.17e-1\\(180.92)}   		&\tabincell{c}{\textbf{2.82e-2}\\\textbf{(6.49)}} \\				
			\midrule		
		&$\eta=0.001$        &\tabincell{c}{3.74e-4\\(2.29)}  &\tabincell{c}{3.73e-4\\(4.93)}   &\tabincell{c}{1.00e-3\\(145.67)}        &\tabincell{c}{2.74e-3\\(61.56)}      &\tabincell{c}{\textbf{3.54e-4}\\\textbf{(2.07)}} \\ 								
		Star&$\eta=0.01$         &\tabincell{c}{4.03e-3\\(2.39)}  &\tabincell{c}{4.02e-3\\(6.15)}    &\tabincell{c}{4.50e-3\\(159.19)}      &\tabincell{c}{5.08e-2\\(66.49)}      &\tabincell{c}{\textbf{3.31e-3}\\\textbf{(2.27)}}\\ 
		&$\eta=0.1$        &\tabincell{c}{3.96e-2\\(2.66)}  &\tabincell{c}{3.68e-2\\(4.87)}   &\tabincell{c}{3.92e-2\\(164.16)}     &\tabincell{c}{1.20e-1\\(65.58)}   		&\tabincell{c}{\textbf{2.74e-2}\\\textbf{(2.54)}} \\										
		\midrule
		&$\eta=0.001$        &\tabincell{c}{3.90e-4\\(6.96)}  &\tabincell{c}{3.90e-4\\(12.72)}   &\tabincell{c}{4.35e-3\\(181.78)}          &\tabincell{c}{1.81e-3\\(180.79)}      &\tabincell{c}{\textbf{3.62e-4}\\\textbf{(5.90)}} \\ 								
		Cell&$\eta=0.01$         &\tabincell{c}{3.96e-3\\(6.46)}  &\tabincell{c}{3.94e-3\\(14.07)}    &\tabincell{c}{4.01e-3\\(181.33)}      &\tabincell{c}{1.33e-2\\(181.36)}      &\tabincell{c}{\textbf{3.28e-3}\\\textbf{(6.60)}}\\ 	
		&$\eta=0.1$        &\tabincell{c}{4.01e-2\\(10.36)}  &\tabincell{c}{3.71e-2\\(19.22)}   &\tabincell{c}{1.16e-1\\(181.22)}          &\tabincell{c}{1.13e-1\\(181.89)}      &\tabincell{c}{\textbf{2.75e-2}\\\textbf{(7.75)}} \\ 															
		\bottomrule	
	\end{tabular}
\end{table}

\begin{table}[!t]
	\centering
	\caption{Relative errors (running time) under the Type-II noise for tested images. The best results with the corresponding time are marked in bold.}\label{rel-lap} 
	\begin{tabular}{l l|c|c|c|c|c}
	\toprule	
	~ &~ &~~\verb"M-RWF"~~&~~\verb"M-MRAF"~~&~~\verb"L0L1-PR"~~ &~~\verb"L1/2-LAD"~~ &~~\verb"This work"~~ \\ 	\midrule
		&$\eta=0.001$        &\tabincell{c}{9.78e-4\\(7.14)}  &\tabincell{c}{9.63e-4\\(12.43)}   &\tabincell{c}{8.97e-4\\(180.82)}     &\tabincell{c}{6.58e-3\\(180.47)}      &\tabincell{c}{\textbf{8.50e-4}\\\textbf{(3.58)}} \\ 								
		HI&$\eta=0.01$         &\tabincell{c}{9.85e-3\\(7.61)}  &\tabincell{c}{9.85e-3\\(12.62)}    &\tabincell{c}{9.09e-3\\(180.60)}      &\tabincell{c}{4.07e-2\\(180.53)}      &\tabincell{c}{\textbf{8.68e-3}\\\textbf{(3.65)}}\\ 	
		&$\eta=0.1$        &\tabincell{c}{9.92e-2\\(7.65)}  &\tabincell{c}{1.02e-1\\(14.18)}   &\tabincell{c}{9.26e-2\\(180.74)}     &\tabincell{c}{2.50e-1\\(180.71)}   		&\tabincell{c}{\textbf{7.74e-2}\\\textbf{(3.07)}} \\				
	\midrule
		&$\eta=0.001$        &\tabincell{c}{9.95e-4\\(3.19)}  &\tabincell{c}{9.64e-4\\(6.09)}   &\tabincell{c}{1.27e-3\\(135.68)}        &\tabincell{c}{1.16e-2\\(63.35)}      &\tabincell{c}{\textbf{1.80e-4}\\\textbf{(0.57)}} \\ 								
		Star&$\eta=0.01$         &\tabincell{c}{9.54e-3\\(3.16)}  &\tabincell{c}{9.69e-3\\(6.02)}    &\tabincell{c}{9.75e-3\\(165.75)}      &\tabincell{c}{6.20e-2\\(68.14)}      &\tabincell{c}{\textbf{1.72e-3}\\\textbf{(0.71)}}\\ 
		&$\eta=0.1$        &\tabincell{c}{9.39e-2\\(3.38)}  &\tabincell{c}{9.73e-2\\(6.15)}   &\tabincell{c}{9.59e-2\\(135.17)}     &\tabincell{c}{2.66e-1\\(64.94)}   		&\tabincell{c}{\textbf{3.97e-2}\\\textbf{(1.57)}} \\										
		\midrule
		&$\eta=0.001$        &\tabincell{c}{9.58e-4\\(10.86)}  &\tabincell{c}{9.54e-4\\(18.82)}   &\tabincell{c}{9.12e-3\\(181.68)}          &\tabincell{c}{4.59e-2\\(180.59)}      &\tabincell{c}{\textbf{8.53e-4}\\\textbf{(4.21)}} \\ 								
		Cell&$\eta=0.01$         &\tabincell{c}{9.65e-3\\(10.94)}  &\tabincell{c}{9.68e-3\\(17.51)}    &\tabincell{c}{9.12e-3\\(181.68)}      &\tabincell{c}{4.59e-2\\(180.59)}      &\tabincell{c}{\textbf{8.20e-3}\\\textbf{(6.16)}}\\ 	
		&$\eta=0.1$        &\tabincell{c}{9.76e-2\\(11.34)}  &\tabincell{c}{9.95e-2\\(17.90)}   &\tabincell{c}{1.30e-1\\(181.46)}          &\tabincell{c}{2.67e-1\\(180.76)}      &\tabincell{c}{\textbf{8.00e-2}\\\textbf{(9.21)}} \\ 															
	\bottomrule		
	\end{tabular}
\end{table}

\begin{table}[!t]
	\centering
	\caption{Relative errors (running time) under the Type-III noise for tested images. The best results with the corresponding time are marked in bold.}\label{rel-out} 
	\begin{tabular}{l  l|c|c|c|c|c}
		\toprule	
	~ &~ &~~\verb"M-RWF"~~&~~\verb"M-MRAF"~~&~~\verb"L0L1-PR"~~ &~~\verb"L1/2-LAD"~~ &~~\verb"This work"~~ \\ 	\midrule
		&$\eta=0.15$        &\tabincell{c}{9.11e-6\\(0.95)}  &\tabincell{c}{\textbf{7.41e-6}\\\textbf{(0.83)}}   &\tabincell{c}{9.83e-5\\(178.68)}     &\tabincell{c}{1.28e-1\\(180.48)}      &\tabincell{c}{1.48e-2\\(6.89)} \\ 								
		HI&$\eta=0.2$         &\tabincell{c}{9.19e-6\\(0.85)}  &\tabincell{c}{\textbf{9.08e-6}\\\textbf{(0.86)}}    &\tabincell{c}{8.13e-5\\(180.58)}      &\tabincell{c}{1.40e-1\\(180.50)}      &\tabincell{c}{5.17e-2\\(4.81)}\\ 	
		&$\eta=0.25$        &\tabincell{c}{9.27e-2\\(8.53)}  &\tabincell{c}{\textbf{9.94e-6}\\\textbf{(1.46)}}   &\tabincell{c}{2.15e-2\\(180.32)}     &\tabincell{c}{1.58e-1\\(180.60)}   		&\tabincell{c}{7.48e-2\\(3.90)} \\				
	\midrule	
		&$\eta=0.15$        &\tabincell{c}{\textbf{8.99e-6}\\\textbf{(0.38)}}  &\tabincell{c}{9.26e-6\\(0.30)}    &\tabincell{c}{9.87e-3\\(133.15)}      &\tabincell{c}{1.25e-1\\(65.37)}      &\tabincell{c}{5.17e-4\\(0.53)}\\ 
		Star&$\eta=0.2$         &\tabincell{c}{\textbf{9.83e-6}\\\textbf{(0.51)}}  &\tabincell{c}{9.96e-6\\(0.43)}    &\tabincell{c}{1.21e-3\\(132.35)}      &\tabincell{c}{1.36e-1\\(57.81)}      &\tabincell{c}{3.05e-5\\(0.74)}\\ 
		&$\eta=0.25$        &\tabincell{c}{1.08e-2\\(2.75)}  &\tabincell{c}{\textbf{9.04e-6}\\\textbf{(0.55)}}   &\tabincell{c}{1.22e-3\\(136.63)}     &\tabincell{c}{1.58e-1\\(61.07)}   		&\tabincell{c}{2.55e-5\\(1.02)} \\										
		\midrule	
		&$\eta=0.15$        &\tabincell{c}{\textbf{9.51e-6}\\\textbf{(1.10)}}  &\tabincell{c}{9.53e-6\\(0.95)}   &\tabincell{c}{6.17e-2\\(181.72)}          &\tabincell{c}{1.28e-1\\(181.12)}      &\tabincell{c}{1.99e-2\\(6.43)} \\ 								
		Cell&$\eta=0.2$         &\tabincell{c}{9.11e-6\\(1.23)}  &\tabincell{c}{\textbf{8.31e-6}\\\textbf{(1.16)}}    &\tabincell{c}{9.60e-2\\(181.34)}      &\tabincell{c}{1.39e-1\\(180.72)}      &\tabincell{c}{6.59e-2\\(7.23)}\\ 	
		&$\eta=0.25$        &\tabincell{c}{9.78e-2\\(7.90)}  &\tabincell{c}{\textbf{9.35e-6}\\\textbf{(1.69)}}   &\tabincell{c}{1.49e-1\\(181.31)}          &\tabincell{c}{1.63e-1\\(180.71)}      &\tabincell{c}{1.02e-1\\(6.39)} \\ 															
		\bottomrule		
	\end{tabular}
\end{table}

\begin{table}[!t]
	\centering
	\caption{Relative errors (running time) under the Gaussian noise for tested images. The best results with the corresponding time are marked in bold.}\label{rel-gau} 
	\begin{tabular}{l  l|c|c|c|c|c}
		\toprule		
	~ &~ &~~\verb"M-RWF"~~&~~\verb"M-MRAF"~~&~~\verb"L0L1-PR"~~ &~~\verb"L1/2-LAD"~~ &~~\verb"This work"~~ \\ 	\midrule	
		&$\eta=0.001$        &\tabincell{c}{1.10e-3\\(9.02)}  &\tabincell{c}{1.10e-3\\(14.69)}   &\tabincell{c}{1.19e-3\\(180.98)}     &\tabincell{c}{6.61e-3\\(180.56)}      &\tabincell{c}{\textbf{9.84e-4}\\\textbf{(2.88)}} \\ 								
		HI&$\eta=0.01$         &\tabincell{c}{1.09e-2\\(7.16)}  &\tabincell{c}{1.09e-2\\(13.65)}    &\tabincell{c}{1.10e-2\\(180.72)}      &\tabincell{c}{4.30e-2\\(180.47)}      &\tabincell{c}{\textbf{9.37e-3}\\\textbf{(3.22)}}\\ 	
		&$\eta=0.1$        &\tabincell{c}{1.12e-1\\(8.18)}  &\tabincell{c}{1.12e-1\\(14.26)}   &\tabincell{c}{1.14e-1\\(180.91)}     &\tabincell{c}{2.70e-1\\(180.50)}   		&\tabincell{c}{\textbf{9.34e-2}\\\textbf{(5.21)}} \\				
		\midrule		
		&$\eta=0.001$        &\tabincell{c}{1.12e-3\\(3.05)}  &\tabincell{c}{1.08e-3\\(5.96)}    &\tabincell{c}{9.87e-3\\(133.15)}      &\tabincell{c}{8.71e-3\\(63.87)}      &\tabincell{c}{\textbf{2.81e-4}\\\textbf{(0.83)}}\\ 
		Star&$\eta=0.01$         &\tabincell{c}{1.06e-2\\(3.08)}  &\tabincell{c}{1.07e-2\\(6.51)}    &\tabincell{c}{1.08e-2\\(145.14)}      &\tabincell{c}{6.37e-2\\(63.01)}      &\tabincell{c}{\textbf{2.47e-3}\\\textbf{(0.83)}}\\ 
		&$\eta=0.1$        &\tabincell{c}{1.13e-1\\(6.04)}  &\tabincell{c}{1.13e-1\\(8.44)}   &\tabincell{c}{1.15e-1\\(180.35)}     &\tabincell{c}{2.64e-1\\(66.75)}   		&\tabincell{c}{\textbf{5.30e-2}\\\textbf{(2.45)}} \\										
		\midrule	
		&$\eta=0.001$        &\tabincell{c}{1.11e-3\\(11.66)}  &\tabincell{c}{1.10e-3\\(18.03)}   &\tabincell{c}{4.81e-3\\(180.48)}          &\tabincell{c}{3.95e-3\\(180.87)}      &\tabincell{c}{\textbf{9.69-4}\\\textbf{(4.75)}} \\ 								
		Cell&$\eta=0.01$         &\tabincell{c}{1.11e-2\\(11.44)}  &\tabincell{c}{1.11e-2\\(18.40)}    &\tabincell{c}{1.31e-2\\(181.96)}      &\tabincell{c}{4.75e-2\\(181.86)}      &\tabincell{c}{\textbf{9.84e-3}\\\textbf{(6.24)}}\\ 	
		&$\eta=0.1$        &\tabincell{c}{1.11e-1\\(11.82)}  &\tabincell{c}{1.11e-1\\(20.10)}   &\tabincell{c}{1.57e-1\\(180.67)}          &\tabincell{c}{2.83e-1\\(180.69)}      &\tabincell{c}{\textbf{9.36e-2}\\\textbf{(10.37)}} \\ 															
		\bottomrule		
	\end{tabular}
\end{table}

\subsection{Signal Examples}\label{s6.2}

In this study, we fix $p=128$ and $s=12$. Assume that $\mathbf{a}_i\in\mathbb{H}^p$ are generated from i.i.d. standard normal distributions, the $s$-sparse vector $\mathbf{x}_{true}\in\mathbb{H}^p$ is generated from Gaussian random vectors, and the number of measurements $n$ varies in the candidate set $\{p, 2p,\cdots, 10p\}$. 

Figure \ref{rate} shows the success rates versus $ n/p$ for real-valued cases marked with (R) and complex-valued cases marked with (C), respectively. 
It can be seen  that for complex-valued cases, the success rates of our proposed method are the highest among all comparison methods, and are far higher than other comparison methods for real-valued cases with Type-II noise. 
This suggests that the proposed method is robust for all selected noise, ranging from Gaussian noise, Laplace noise and outliers.

In addition, the corresponding running times for $n=2p, 4p, 6p, 8p, 10p$ are listed in Table \ref{time1} and Table \ref{time2}, respectively. Moreover, the best results are marked in bold.
For all cases, the proposed method performs very stably and takes less than 1 second, outperforming \verb"L0L1-PR", \verb"LAD-ADMM" and \verb"L1/2-LAD". 
In particular, for the complex-valued case $n=10p$ with Type-I noise, the running time is reduced by more than 5000 times compared to \verb"LAD-ADMM".
Although \verb"M-RWF" and \verb"M-MRAF" require less running time, their performance is poor as shown in Figure \ref{error}.

\subsection{Image Examples}

In this study, we compare the performance of image recovery. Figure \ref{images} shows the tested images, i.e., (a) HI, (b) Star and (c) Cell. Table \ref{rel-dense}, Table \ref{rel-lap} and Table \ref{rel-out} provide the relative errors and running time(s) in brackets for tested images with the Type-I, Type-II and Type-III noise under different $\eta$, respectively. Moreover, the smallest errors with the corresponding running times are labeled in bold. It can be concluded that for tested images with the Type-I and Type-II noise, the proposed method always achieves smaller errors using relatively cheap running time compared to other methods. It then validates the robustness of image recovery.
We would like to point out that for tested images with the Type-III noise, the proposed method performs better than \verb"L0L1-PR" and \verb"L1/2-LAD", but worse than \verb"M-RWF" and \verb"M-MRAF", which is consistent with the conclusion in Section \ref{s6.2}.

We here also consider the images corrupted by Gaussian noise \cite{weller15}, generated from $\varepsilon_i=\eta \|b\|/\sqrt{n} w_i$ for $i=1,\cdots,n$, where $\{w_i\}$ is the Gaussian noise with zero mean and unit variance; see Table \ref{rel-gau} for the computational results. Obviously, the proposed method is able to achieve the smallest relative errors with the smallest running times, which again shows the robustness.

From the above signal and image experiments, it can be seen that although the performance is different on real- and complex-valued datasets under different types of noise, compared with the other methods, our proposed method has higher robustness and accuracy with moderate running time. This, together with the statistical guarantee, convinces us to believe that the proposed method is robust.

\section{Conclusions}\label{conclusion}

In this paper, we have proposed a new robust PR method to recover signals with noise and outliers. Particularly, both real- and complex-valued cases have been considered. For the real-valued case, we have established the consistency of the estimator and derived its convergence rate. By applying the MM framework, we have established a fixed point inclusion for the global minimizer and  developed an efficient and convergent optimization algorithm. Further, we have shown that the whole generated sequence converges with a linear rate. Note that the consistency demonstrates that the proposed method possesses favorable statistical properties, and the fixed-point inclusion provides a theoretical foundation for designing our algorithm. Combined with the obtained results on the algorithm's convergence and convergence rate, we have essentially addressed the three initial questions. It is worth mentioning that the existing algorithms for regularized methods in complex-valued cases only gave the guarantee that the limit point of the generated sequence satisfies a necessary optimality condition of a real-valued optimization problem transformed by the original problem in the complex domain, by separating the real and imaginary part and lifting the dimension. However, we have stated that these two problems are not equivalent. Fortunately, we have proved that  the generated sequence converges to a vector satisfying the fixed point inclusion introduced by the original problem in the complex domain. Numerical examples under different settings for both signals and images verify the effectiveness of the proposed method.

There remain several important directions for future work. First, it is essential to establish non-asymptotic error bounds at the rate 
 	$$ O_{\mathbb{P}}\left( \sqrt{  \|\mathbf{x}^\diamond\|_0 \log (p / \|\mathbf{x}^\diamond\|_0)/ n  } \right) $$ 
 	for the proposed estimator in high-dimensional settings ($p > n$), covering both real and complex signal cases.In addition, a thorough investigation into the geometric properties of the objective function would provide valuable insight into whether the algorithm converges to statistically favorable solutions. Thirdly, in certain cases, such as images corrupted by the Type-III noise, the proposed method does not achieve the best performance, suggesting the need to integrate deep learning techniques to enhance robustness. Finally, we are interested in applying the proposed method to real-world scenarios, including optical imaging and crystallography.

\section*{Appendix}
\appendix

\section{Proof of Lemma \ref{prob-E} }

 For convenience, we let the sub-exponential random variable $\varepsilon_1$ be  with parameters $(\sigma^2,\kappa)$, i.e., for all $s\in\mathbb{R}$ such that $|s|\leq 1/\kappa$,
$$\mathbb{E}e^{s(|\varepsilon_1|-\mathbb{E}|\varepsilon_1|)}\leq\exp(\frac{s^2\sigma^2}{2}).$$
Based upon Bernstein's inequality, for any $t_1>0$, it follows that \begin{equation}
	\begin{aligned}
		\mathbb{P}\Big(\frac{1}{n}\sum_{i=1}^n(|\varepsilon_i|-\mathbb{E}|\varepsilon_i|)>t_1\Big)\leq \exp\Big\lbrace -\frac{n}{2}\min\Big\{\frac{t_1^2}{\sigma^2},\frac{t_1}{\kappa}\Big\}\Big\rbrace.
	\end{aligned}\nonumber
\end{equation}
The condition (\ref{cond-np-consistency}) results in $t_1^2/(\sigma^2)\leq t_n/\kappa$ for large enough $n$. 
Therefore, it follows that
\begin{equation}\label{prob-E1}
	\begin{aligned}
		\mathbb{P}\Big(\frac{1}{n}\sum_{i=1}^n(|\varepsilon_i|-\mathbb{E}|\varepsilon_i|)>t_n\Big)\leq \exp\big\lbrace -\frac{nt_n^2}{2\sigma^2} \big\rbrace
		=\frac{1}{(1+2n)^{(2p+1)/\sigma^2}}\to0.
	\end{aligned}
\end{equation}

 Notice that $|h'_\alpha(\cdot)|\leq\alpha$ which together with Example 2.4 in \cite{wainwright2019high} implies that  $h'_{\alpha}(\varepsilon_i)$ is sub-Gaussian with variance proxy $\alpha^2$. Using the assumption that the distribution of $\varepsilon_i$ is symmetric, one can see that $\mathbb{E}h'_{\alpha}(\varepsilon_i)=0$ for any $\alpha>0$.  Similar to the proof of Lemma 3.1 in \cite{fan2022oracle}, we can estimate the probability $\mathbb{P}(E_2)$.  To keep the paper self-contained, we give the details.
For any nonzero vectors $\mathbf{u}, \mathbf{v}\in\mathbb{R}^p$, it follows from Bernstein's inequality   that
\begin{equation}
	\begin{aligned}
		\mathbb{P}\Big(|\frac{1}{n}\sum_{i=1}^n\langle \mathbf{u}, \mathbf{a}_i\mathbf{a}_i^\top\mathbf{v}\rangle h'_{\alpha}(\varepsilon_i)|>t_n\alpha\Big)\leq
		2\exp\Big\lbrace -\frac{nt_n^2}{2\sum_{i=1}^n  \langle \mathbf{u}, \mathbf{a}_i\mathbf{a}_i^\top\mathbf{v}\rangle^2/n}\Big\rbrace,\nonumber
	\end{aligned}
\end{equation}
which together with the fact $\sum_{i=1}^n \langle \mathbf{u}, \mathbf{a}_i\mathbf{a}_i^\top\mathbf{v}\rangle^2/n\leq \|\mathbf{u}\|^2\|\mathbf{v}\|^2$ results in
\begin{equation}\label{uav-error}
	\begin{aligned}
		\mathbb{P}\Big(|\frac{1}{n}\sum_{i=1}^n \langle \mathbf{u}, \mathbf{a}_i\mathbf{a}_i^\top\mathbf{v}\rangle h'_{\alpha}(\varepsilon_i)|>t_n\alpha\Big)\leq
		2\exp\Big\lbrace -\frac{nt_n^2}{2\|\mathbf{u}\|^2\|\mathbf{v}\|^2}\Big\rbrace.
	\end{aligned}
\end{equation}
For any $\mathbf{u}\in S^{p-1}$, let $$ B(\mathbf{u},\delta):=\{\mathbf{u}\in S^{p-1}:\|\mathbf{v}-\mathbf{u}\|\leq{}\delta\}.$$
It follows from \cite[Example 5.8]{wainwright2019high} that there exists a covering net $\{\mathbf{u}^j\}\in S^{p-1}$ with cardinality which is upper bounded as $J:=(1+2n)^p$ such that   
\begin{equation}\nonumber
	\begin{aligned}
		S^{p-1}\subseteq\bigcup_{j=1}^{J}B\left( \mathbf{u}^j,\frac{1}{n}\right).
	\end{aligned}
\end{equation}
Using this result,  for any fixed $\mathbf{v}\in S^{p-1}$, we obtain that
\begin{equation}
	\begin{aligned}
		&\sup_{u\in S^{p-1}}|\frac{1}{n}\sum_{i=1}^n \langle \mathbf{u}, \mathbf{a}_i\mathbf{a}_i^\top\mathbf{v}\rangle h'_{\alpha}(\varepsilon_i)| \\
		\leq&\sup_{\mathbf{u}\in\bigcup_{j=1}^{J}B\left( \mathbf{u}^j,\frac{1}{n}\right)}|\frac{1}{n}\sum_{i\in\mathcal{I}_{0}^{\mathrm{in}}}\langle \mathbf{u}, \mathbf{a}_i\mathbf{a}_i^\top\mathbf{v}\rangle h'_{\alpha}(\varepsilon_i)| \\
		\leq&\max_{1\leq j\leq J}\Big(|\frac{1}{n}\sum_{i=1}^n \langle \mathbf{u}^j, \mathbf{a}_i\mathbf{a}_i^\top\mathbf{v}\rangle h'_{\alpha}(\varepsilon_i)|
		+\sup_{\mathbf{u}\in\bigcup_{j=1}^{J}B\left( \mathbf{u}^j,\frac{1}{n}\right)}|\frac{1}{n}\sum_{i=1}^n \langle \mathbf{u}-\mathbf{u}^j, \mathbf{a}_i\mathbf{a}_i^\top\mathbf{v}\rangle h'_{\alpha}(\varepsilon_i)|\Big) \\
		\leq&\max_{1\leq j\leq J}\Big(|\frac{1}{n}\sum_{i=1}^n \langle \mathbf{u}^j, \mathbf{a}_i\mathbf{a}_i^\top\mathbf{v}\rangle h'_{\alpha}(\varepsilon_i)|
		+\sup_{\mathbf{u}\in\bigcup_{j=1}^{J}B\left( \mathbf{u}^j,\frac{1}{n}\right)}\sqrt{\frac{1}{n}\sum_{i=1}^n  \langle \mathbf{u}-\mathbf{u}^j, \mathbf{a}_i\mathbf{a}_i^\top\mathbf{v}\rangle^2}\sqrt{\frac{1}{n}\sum_{i=1}^n   ( h'_{\alpha}(\varepsilon_i))^2}\Big) \\
		\leq&\max_{1\leq j\leq J}\Big(|\frac{1}{n}\sum_{i=1}^n \langle \mathbf{u}^j, \mathbf{a}_i\mathbf{a}_i^\top\mathbf{v}\rangle h'_{\alpha}(\varepsilon_i)|
		+\sup_{\mathbf{u}\in\bigcup_{j=1}^{J}B\left( \mathbf{u}^j,\frac{1}{n}\right)} \|\mathbf{u}-\mathbf{u}^j\|\|\mathbf{v}\|\alpha\Big) \\
		\leq&\max_{1\leq j\leq J}|\frac{1}{n}\sum_{i=1}^n \langle \mathbf{u}^j, \mathbf{a}_i\mathbf{a}_i^\top\mathbf{v}\rangle h'_{\alpha}(\varepsilon_i)|
		+\frac{\alpha}{n}.\nonumber
	\end{aligned}
\end{equation}
Notice that $\mathbf{v}\in S^{p-1}$. The same reasoning allows us to get, for each fixed $\mathbf{u}^j$,
\begin{equation}
	\begin{aligned}
		\sup_{\mathbf{v}\in S^{p-1}}|\frac{1}{n}\sum_{i=1}^n \langle \mathbf{u}^j, \mathbf{a}_i\mathbf{a}_i^\top\mathbf{v}\rangle h'_{\alpha}(\varepsilon_i)|
		\leq\max_{1\leq k\leq J}|\frac{1}{n}\sum_{i=1}^n \langle \mathbf{u}^j, \mathbf{a}_i\mathbf{a}_i^\top\mathbf{v}^k\rangle h'_{\alpha}(\varepsilon_i)|
		+\frac{\alpha}{n}. \nonumber
	\end{aligned}
\end{equation}
The above two inequalities enable us to derive that
\begin{equation}\label{norm2-number}
	\begin{aligned}
		\sup_{\mathbf{u},\mathbf{v}\in S^{p-1}} |\frac{1}{n}\sum_{i=1}^n \langle \mathbf{u}, \mathbf{a}_i\mathbf{a}_i^\top\mathbf{v}\rangle h'_{\alpha}(\varepsilon_i) |
		\leq
		\max_{1\leq j,k\leq J}|\frac{1}{n}\sum_{i=1}^n \langle \mathbf{u}^j, \mathbf{a}_i\mathbf{a}_i^\top\mathbf{v}^k\rangle h'_{\alpha}(\varepsilon_i)|
		+\frac{\alpha}{n}.
	\end{aligned}\nonumber
\end{equation}
According to the above inequality, we have
\begin{equation}
	\begin{aligned}
		&\mathbb{P}\Big( \sup_{\mathbf{u},\mathbf{v}\in S^{p-1}} |\frac{1}{n}\sum_{i=1}^n\langle \mathbf{u}, \mathbf{a}_i\mathbf{a}_i^\top\mathbf{v}\rangle h'_{\alpha}(\varepsilon_i) |>
		t_n\alpha+\frac{\alpha}{n}\Big)\\
	&	\leq \mathbb{P}\Big( \max_{1\leq j,k\leq J}|\frac{1}{n}\sum_{i=1}^n \langle \mathbf{u}^j, \mathbf{a}_i\mathbf{a}_i^\top\mathbf{v}^k\rangle h'_{\alpha}(\varepsilon_i)|
		+\frac{\alpha}{n}>t_n\alpha+\frac{\alpha}{n}\Big)\\
	&	\leq \mathbb{P}\big( \big| \max_{1\leq j,k\leq J}|\frac{1}{n}\sum_{i=1}^n \langle \mathbf{u}^j, \mathbf{a}_i\mathbf{a}_i^\top\mathbf{v}^k\rangle h'_{\alpha}(\varepsilon_i)|\big|>t_n\alpha\big)
		\\
	&	\leq \sum_{1\leq j,k\leq J}\mathbb{P}\big(  |\frac{1}{n}\sum_{i=1}^n \langle \mathbf{u}^j, \mathbf{a}_i\mathbf{a}_i^\top\mathbf{v}^k\rangle h'_{\alpha}(\varepsilon_i)|>
		t_n\alpha\big)
		\\
	&	\leq 2J^2\exp\big\lbrace -\frac{nt_n^2}{2}\big\rbrace,
	\end{aligned}\nonumber
\end{equation}
where the last inequality is due to (\ref{uav-error}) and $\|\mathbf{u}^j\|=\|\mathbf{v}^k\|=1$. Therefore, we  have
\begin{equation}
	\begin{aligned}
		2J^2\exp\Big\lbrace -\frac{nt_n^2}{2}\Big\rbrace =2\exp\Big\lbrace-\frac{nt_n^2}{2}+2p\log(1+2n) \Big\rbrace  
		= 2\exp\left\lbrace-\log(1+2n)\right\rbrace\to 0,	\end{aligned}\nonumber
\end{equation}
which together with (\ref{prob-E1}) leads to the desired result. \qed

\section{Proof of Lemma \ref{armi1}}
As shown in the proof of Theorem \ref{the1}, the level set $S=\{\mathbf{x}\in\mathbb{H}^p: F(\mathbf{x})\leq F(\mathbf{x}^0)\}$ satisfies $S\subseteq B_r$ with $r=\frac{4}{3}(1+\frac{2\alpha}{n}\sum_{i=1}^n\|\mathbf{a}_i\|^2)\sup_{\mathbf{x}\in S}\|\mathbf{x}\|$  for any given bounded initial vector $\mathbf{x}^0$. 
According to the definitions of $\underline{j}$ and $L$, it derives that
$\gamma\alpha^{\underline{j}}=\gamma\leq(L+\delta)^{-1}$
as $\gamma(L+\delta)\leq1$, and
$\gamma\beta^{\underline{j}}\leq\gamma\beta^{-[\log_\alpha\gamma(L+\delta)]}
\leq\gamma\beta^{-\log_\beta\gamma(L+\delta)}=(L+\delta)^{-1}$
as $\gamma(L+\delta)>1$. Take $\tau_0=\beta\gamma^{\underline{j}}$ which results in that $\tau_0\in(0,1]$ and $\frac{1}{\tau_0}\geq L+\delta$.
Note that
\begin{eqnarray}
\|g(\mathbf{x}^0)\|=\|\frac{1}{n}\sum\limits_{i=1}^n h'_{\alpha}(|\langle \mathbf{a}_i, \mathbf{x}^0\rangle|^2-b_i)\langle \mathbf{a}_i, \mathbf{x}^0\rangle \bar{\mathbf{a}}_i\|
\leq\frac{\alpha}{n}\sum\limits_{i=1}^n \|\mathbf{a}_i\|^2\|\mathbf{x}^0\|,\nonumber
\end{eqnarray}
which implies that \begin{eqnarray}
\|\mathbf{x}^0-2\tau g(\mathbf{x}^0)\|\leq\|\mathbf{x}^0\|+
2\|g(\mathbf{x}^0)\|\leq(1+\frac{2\alpha}{n}\sum\limits_{i=1}^n \|\mathbf{a}_i\|^2)\|\mathbf{x}^0\|\nonumber
\end{eqnarray}
for any $\tau\in(0,1]$. Combing this and the expression of the half-thresholding operator $\mathcal{H}$, we have
$$\|\mathcal{H}_{\lambda\tau}(\mathbf{x}^0-2\tau g(\mathbf{x}^0)\|\leq\frac{4}{3}\|\mathbf{x}^0-2\tau g(\mathbf{x}^0)\|
\leq\frac{4}{3}(1+\frac{2\alpha}{n}\sum\limits_{i=1}^n \|\mathbf{a}_i\|^2)\|\mathbf{x}^0\|,$$
which yields $\mathbf{x}^0(\tau):=\mathcal{H}_{\lambda\tau}(\mathbf{x}^0-2\tau g(\mathbf{x}^0)\in{}B_r$ for any $\tau\in(0,1]$.
The definition of the half-thresholding operator $\mathcal{H}$ also implies that
$$\mathbf{x}^0(\tau)=\mathrm{arg}\min_{\mathbf{x}\in\mathbb{H}^p}F_{\tau}(\mathbf{x},\mathbf{x}^0).$$
Similar to the proof of Theorem \ref{the1}, we obtain that
\begin{eqnarray}
F(\mathbf{x}^0(\tau))\leq F_{\tau}(\mathbf{x}^0(\tau),\mathbf{x}^0)-\frac{1}{2}\|\mathbf{x}^0(\tau)-\mathbf{x}^0\|^2
(\frac{1}{\tau}-L).\nonumber
\end{eqnarray}
Combining this, $\mathbf{x}^0(\tau)=\mathrm{arg}\min_{\mathbf{x}\in\mathbb{H}^p}F_{\tau}(\mathbf{x},\mathbf{x}^0)$ and $\frac{1}{\tau_0}\geq L+\delta$, we have
 \begin{equation}\label{fkdec}
\begin{aligned}
F(\mathbf{x}^0)-F(\mathbf{x}^0(\tau))
& \geq
F(\mathbf{x}^0)-F_{\tau}(\mathbf{x}^0(\tau),\mathbf{x}^0)+\frac{\delta}{2}\|\mathbf{x}^0(\tau)-\mathbf{x}^0\|^2\nonumber\\
&=F_{\tau}(\mathbf{x}^0,\mathbf{x}^0)-F_{\tau}(\mathbf{x}^0(\tau),\mathbf{x}^0)+\frac{\delta}{2}\|\mathbf{x}^0(\tau)-\mathbf{x}^0\|^2\nonumber\\
& \geq
\frac{\delta}{2}\|\mathbf{x}^0(\tau)-\mathbf{x}^0\|^2.\nonumber
\end{aligned}
 \end{equation}
Thus, one can find the smallest integer $j_k\geq\underline{j}$ satisfying the descent condition in (\ref{armijo}) for $k=0$. This descent condition implies that $F(\mathbf{x}^1)\leq F(\mathbf{x}^0)$ and $\mathbf{x}^1\in S$. Repeating the above step can show such a descent condition for any $k\geq1.$ From the above analysis, one also get $j_k\in[\underline{j}, 0]$. Hence, the desired result is obtained.
\qed

\paragraph{Acknowledgement}

This work was supported in part by the National Natural Science Foundation
of China under Grants 12371306 and 12271022, and in part by the Natural Science Foundation of Hebei Province under Grant A2023202038.

\bibliographystyle{elsarticle-num}
\bibliography{mybibfile}

\end{document}